\documentclass{article}
\usepackage[utf8]{inputenc}
\usepackage[T1]{fontenc}
\usepackage{amsthm}
\usepackage{mathrsfs}
\usepackage[english]{babel}
\usepackage{enumitem}
\setlist{nosep} 
\usepackage[a4paper]{geometry}
\usepackage{fullpage}
\usepackage{stackrel}

\let\OLDthebibliography\thebibliography
\renewcommand\thebibliography[1]{
  \OLDthebibliography{#1}
  \setlength{\parskip}{0pt}
  \setlength{\itemsep}{0pt plus 0.3ex}
}

\usepackage{comment}

\usepackage{mathtools}
\usepackage{amsfonts}
\usepackage{amssymb}
\usepackage{amsthm}
\usepackage[subnum]{cases}

\usepackage[colorlinks=true]{hyperref}
\hypersetup{urlcolor=blue, citecolor=red}

\numberwithin{equation}{section}

\newtheorem{theorem}{Theorem}[section]
\newtheorem{corollary}[theorem]{Corollary}
\newtheorem{lemma}[theorem]{Lemma}
\newtheorem{proposition}[theorem]{Proposition}

\newtheorem{definition}[theorem]{Definition}
\theoremstyle{remark}
\newtheorem{remark}[theorem]{Remark}

\usepackage{fullpage}
\usepackage{graphicx}
\graphicspath{ {./imagens/} }
\usepackage{multicol} 
\usepackage{wrapfig}
\usepackage{pgf,tikz,pgfplots}
\usetikzlibrary{shapes,arrows} 
\usetikzlibrary{automata,positioning}
\usetikzlibrary{decorations.markings}
\tikzset{->-/.style={decoration={
			markings,
			mark=at position #1 with {\arrow{>}}},postaction={decorate}}}
\usepackage{float}

\newcommand{\R}{\mathbb{R}}

\newcommand{\Hcal}{\mathcal{H}}

\newcommand{\Kcal}{\mathcal{K}}

\newcommand{\G}{\mathcal{G}}

\newcommand{\norm}[3]{\|#1\|_{L^{#2}(#3)}}

\DeclareMathOperator{\sech}{sech}

\DeclareMathOperator{\sign}{sign}

\DeclareMathOperator{\argsinh}{argsinh}

\title{ {A comprehensive study of bound-states for the nonlinear Schr\"odinger equation on single-knot metric graphs}}

\author{Francisco Agostinho, Sim\~ao Correia and
	Hugo Tavares}

\date{\today} 

\begin{document}
\maketitle

\begin{abstract}
We study the existence and qualitative properties of action ground-states (that is, bound-states with minimal action) {of the nonlinear Schrödinger equation} over single-knot metric graphs - {noncompact graphs} which are made of half-lines, loops and pendants, all connected at a single vertex. First, we prove existence of action ground-state for generic single-knot graphs, even in the absence of an associated variational problem. Second, for regular single-knot graphs of length $\ell$, we perform a complete analysis of positive monotone bound-states. Furthermore, we characterize all positive bound-states when $\ell$ is small and prove some symmetry-breaking results for large $\ell$. Finally, we apply the results to some particular graphs to illustrate the complex relation between action ground-states and the topological {and metric} features of the underlying metric graph.

{The proofs are nonvariational, using a careful phase-plane analysis, the study of sections of period functions, asymptotic estimates and blowup arguments. We show, in particular, how nonvariational techniques are complementary to variational ones in order to deeply understand bound-states of the nonlinear Schr\"odinger equation on metric graphs. }
\end{abstract}


\section{Introduction}

\subsection{Motivation and setting of the problem}
In the past decade, the analysis of bound-states of the nonlinear Schr\"odinger equation (NLS){, for superlinear nonlinearities,} in the context of metric graphs has attracted the attention of both the physical and mathematical communities. Physically, metric graphs can be used as accurate spatial models for optic networks (in the context of signal transmission) or ramified traps (in the framework of Bose-Einstein condensates) \cite{berkolaiko2013introduction,  bolte2014many, bulgakov2011symmetry, burioni2001bose,  dalfovo1999theory, hung2011symmetric}. In the last application, the existence of bound-states directly relates to the appearance of a condensate concentrated over the structure modeled by the metric graph.

From a mathematical point of view, given a metric graph $\mathcal{G}$, bound-states are the  nontrivial critical points of the \textit{action functional} $S(\cdot):H^1(\G)\to\R$ defined by
\begin{equation}\label{ActionFUnctionalChap3}
	S(u)=\frac{1}{2}\int_{\G}|{u'}|^2+|u|^2dx-\frac{1}{p}\int_{\G}|u|^pdx,\qquad p>2.
\end{equation}
 Similar to the euclidean case, we define the \textit{least action level} as
	\begin{equation}\label{ActionProbI}   \mathcal{S}_{\G}:=\inf_{u\in H^1(\G)\setminus\left\{0\right\}}\left\{S(u): u\mbox{ bound-state}\right\}.
	\end{equation}
	An \textit{action ground-state} is a bound-state attaining the least action level.

Comparing with the euclidean case, where there exists a unique bound-state (up to symmetries), new challenges and obstacles appear in the setting of metric graphs. Indeed, one sees that both the \textit{topological} and \textit{metric} properties of the graph are determining factors for existence, multiplicity and qualitative properties of bound-states. See, for example, {\cite{cacciapuoti2018variational,jeanjean_compact_2024, dovetta_compact_2018} for compact graphs; \cite{dovetta_honey_2019, adami_grid_2019,  dovetta_periodic_2019} for periodic graphs; \cite{dovetta_trees_2020} for tree graphs;} \cite{Adami_star2012,adami2014variational} for the star graph; \cite{NOJA2019147} for the double-bridge graph and $p=4$; \cite{PhysRevE.91.013206, noja2015bifurcations, noja2020standing} for the tadpole graph; \cite{kairzhan2021standing} for the flower graph; \cite{AgostinhoCorreiaTavares} for the $\mathcal{T}$-graph.
As a consequence, an all-encompassing theory seems currently out-of-reach and one must restrict oneself to the analysis of graphs with some additional structure. In this work, we introduce the class of\textit{ single-knot metric graphs}{ and show how nonvariational techniques such as phase-plane analysis and blow-up arguments are efficient tools to characterize the sets of bound-states and of action ground-states.}

\begin{definition}[Single-knot metric graph]
	A single-knot metric graph is a graph consisting of {$H\geq 1$} half-lines, $P\ge 0$ pendants and $L\ge 0$ loops, all attached at the same vertex $\mathbf{0}$. The graph is said to be regular (of length $\ell$) if each pendant has length $\ell$ and each loop has length $2\ell$.
\end{definition}


The class of single-knot metric graphs includes the $\mathcal{T}$-graph (which was thoroughly studied in our previous work \cite{AgostinhoCorreiaTavares}), the tadpole graph, the fork graph and the broom graph, depicted in Figures \ref{fig:T}, \ref{fig:tadpole}, \ref{fig:fork} and \ref{fig:broom}, as well as star graphs and flower graphs.
\begin{figure}[h!]
	\centering
	\begin{minipage}{0.45\textwidth}
		\centering\begin{tikzpicture}[node distance=2.5cm,  every loop/.style={}]
			\node(a)		{\Large $\infty$};
			\node[circle, fill=black, label=above:\large $\mathbf{0}$](b)[right of=a]{};
			\node[circle, fill=black, label=left: \large $\boldsymbol{v}$](c)[below of=b]{};
			\node(d)[right of=b]{\Large $\infty$};
			\path (a)edge node[above]{$h_1$}(b)
			(b)edge node[left]{$e_1$} (c)
			edge node[above]{$h_2$} (d)
			;	   	
		\end{tikzpicture}\label{fig:T}
		\caption{The class of $\mathcal{T}$-graphs.}\label{fig:T}	
	\end{minipage}
	\begin{minipage}{0.45\textwidth}
		\centering
		\begin{tikzpicture}[node distance=3cm,every loop/.style={}]
			\node(1){\Large $\infty$};
			\node[circle, fill=black, label=above:\large $\mathbf{0}$](3)[right of=1]{};
			\path
			(1) edge node[above]{$h_1$} (3)
			(3) edge[out=-30,in=30,loop,scale=6] node[left]{$e_1$} (3)
			;
		\end{tikzpicture}
		\vspace{7pt}
		\caption{The class of tadpole graphs.}\label{fig:tadpole}
	\end{minipage}
\end{figure}

\begin{figure}[h!]
	\centering
	\begin{minipage}{0.45\linewidth}
		\centering
		\begin{tikzpicture}[node distance=1.5cm,  every loop/.style={}]
			\node(a)		{\Large $\infty$};
			\node[circle, fill=black, label=above:\large $\mathbf{0}$](b)[right of=a]{};
			\node[circle, fill=black, label=right: \large $\boldsymbol{v}_i$](c)[right of=b]{};
			\node[circle, fill=black, label=right: \large $\boldsymbol{v}_{N-1}$](d)[below of=c]{};
			\node[circle, fill=black, label=right: \large $\boldsymbol{v}_N$](e)[below of=d]{};
			\node[circle, fill=black, label=right: \large $\boldsymbol{v}_2$](g)[above of=c]{};
			\node[circle, fill=black, label=right: \large $\boldsymbol{v}_1$](f)[above of=g]{};
			\path (a)edge node[above]{$h$}(b)
			(b)edge node[left]{$e_1$} (f)
			edge node[above]{$e_2$} (g)
			edge node[above]{$e_i$} (c)
			edge node[right]{$e_{N-1}$} (d)
			edge node[left]{$e_N$} (e)
			(g) edge[dotted] (c)
			(c) edge[dotted] (d)
			;	   	
		\end{tikzpicture}
        		\vspace{7pt}
		\caption{The class of fork graphs.}\label{fig:fork}
	\end{minipage}
		\begin{minipage}{0.45\linewidth}
		\centering
		\begin{tikzpicture}[node distance=1.5cm,  every loop/.style={}]
			\node[circle, fill=black, label=above: \large $\boldsymbol{v}$](a){}		;
			\node[circle, fill=black, label=above:\large $\mathbf{0}$](b)[right of=a]{};
			\node(c)[right of=b]{$\displaystyle \infty$};
			\node(d)[below of=c]{$\displaystyle \infty$};
			\node(e)[below of=d]{$\displaystyle \infty$};
			\node(g)[above of=c]{$\displaystyle \infty$};
			\node(f)[above of=g]{$\displaystyle \infty$};
			\path (a)edge node[above]{$h$}(b)
			(b)edge node[left]{$e_1$} (f)
			edge node[above]{$e_2$} (g)
			edge node[above]{$e_i$} (c)
			edge node[right]{$e_{N-1}$} (d)
			edge node[left]{$e_N$} (e)
			(g) edge[dotted] (c)
			(c) edge[dotted] (d)
			;	   	
		\end{tikzpicture}
        		\vspace{7pt}
		\caption{The class of broom graphs.}\label{fig:broom}
	\end{minipage}
\end{figure}

A classical approach to address action ground-states and their qualitative properties is to relate \eqref{ActionProbI} with the variational problem 
\begin{equation}\label{eq:variational}
    	\mathcal{I}_{\G}(\mu):=\inf_{ u\in H^1(\G)}\left\{ I (u):=\frac{1}{2}\norm{u'}{2}{\G}^2+\frac{1}{2}\norm{u}{2}{\G}^2:\  \frac{1}{p}\norm{u}{p}{\G}^p=\mu\right\}.
\end{equation}
Indeed, it was shown in \cite{AgostinhoCorreiaTavares} that, if the minimization problem \eqref{eq:variational} has a solution, then, {for a specific $\mu$,} action ground-state are the solutions to \eqref{eq:variational} - see also Lemma \ref{lem:minim} below. Afterwards, {for some graphs,} one can use the variational structure to derive several other properties of action ground-states, such as positivity, monotonicity and symmetry. 

It is worth to point out the difference between the existing definitions of ground-state in the literature. First, \cite{de2023notion} defines action ground-states\footnote{Therein, the notion of \textit{least action solution} corresponds to our definition of action ground-state.} as solutions to the Nehari-type minimization problem
\begin{equation}\label{eq:nehari}
    c(\mathcal{G})=\inf_{u\in H^1(\mathcal{G})\setminus \{0\}} \left\{ S(u) : \langle S'(u), u\rangle = 0 \right\}.
\end{equation}
{We have $c(\mathcal{G})\leq \mathcal{S}_\mathcal{G}$ and} it is easy to check that, whenever the above problem has a solution, then both definitions of action ground-state coincide. However, as shown in \cite{de2023notion}, there are cases where \eqref{eq:nehari} is not attained; {the same happens with \eqref{eq:variational}}. We have preferred to work with the more general definition \eqref{ActionProbI}, as it relates to a more physically desirable property.

Another definition concerns \emph{energy} ground-states, which are the solutions to the variational problem
\begin{equation}\label{eq:egs}
    	\mathcal{S}_{\G}(\mu):=\inf_{ u\in H^1(\G)}\left\{ S(u):\  \norm{u}{2}{\G}^2=\mu\right\},\quad 2<p\le 6
\end{equation}
While the notions of action and energy ground-state coincide in the euclidean case (for scaling reasons), the problems are not equivalent in the setting of metric graphs (see \cite{AgostinhoCorreiaTavares, dovetta2024nonuniquenessnormalizedgroundstates, dovetta2023action}). We refer to \cite{Simonegrid2018, adami2015nls, adami2016threshold, 
adami2017negative, 
Borthwick2023,
cacciapuoti2018variational,  chang2023normalized,
SimoneTree2020,   dovetta2024nonuniquenessnormalizedgroundstates, dovetta2020uniqueness, MR3959930} for a profound study of energy ground-states in the context of metric graphs.

\bigskip

In the case of the $\mathcal{T}$ or tadpole graphs, one can show that \eqref{eq:variational} has indeed a solution (see \cite{AgostinhoCorreiaTavares} and Corollary \ref{cor:tadpole} below). On the other hand, for star graphs, \eqref{eq:variational} is not attained (see the construction in \cite[Section 4.3]{dovetta2023action}). We also refer to \cite{Coster} for general criteria for existence/nonexistence of solutions to \eqref{eq:nehari}. For general single-knot metric graphs, it is simple to construct other examples for which \eqref{eq:variational} is \textit{not} attained, in particular, any regular single-knot graph with at least three half-lines and small $\ell$ (see Corollary \ref{coro:broom} and Remark \ref{rmk:falha_var}). As such, while \eqref{eq:variational} is certainly valuable in some particular cases, it fails as a general approach to study action ground-states on single-knot metric graphs (or even general metric graphs). 

In this work, we tackle directly \eqref{ActionProbI} using the bound-state equation
\begin{equation}\label{eq:bstate}
 -u'' + u = |u|^{p-2}u, \quad u\in H^1(\mathcal{G}),
\end{equation}
under the (natural) Neumann-Kirchoff conditions (see \eqref{eq:nk}). We perform a detailed analysis of the solutions to \eqref{eq:bstate} using phase-plane arguments. As a consequence, we are able to show the existence of action ground-states for generic single-knot graphs and prove monotonicity and symmetry properties for regular single-knot graphs. This represents a significant and fruitful departure from the classical variational approach.

\subsection{Description of the main results}

{Let $\mathcal{G}$ be a single-knot metric graph.} We identify each compact edge $\kappa$ of length $\ell_\kappa$ with the interval $[0,\ell_\kappa]$ and each half-line with the interval $[0,\infty)$. In either case, the vertex $\mathbf{0}$ is identified with $x=0$. For the sake of clarity, throughout this work we reserve the letter $\kappa$ for bounded edges (denoting by $\mathcal{K}$ the set of such edges, also referred to as the \emph{compact core}) and $h$ for half-lines (denoting by $\mathcal{H}$ the set of unbounded edges). The set of all edges is denoted by $\mathcal{E}$. {Given $e\in \mathcal{E}$ and $u:\mathcal{G}\to \R$, we denote by $u_e$ the restriction of $u$ to $e$.}

Our first set of results concerns the existence of action ground-states. As mentioned above, the variational approach \eqref{eq:variational} is not suitable for the whole class of single-knot graphs. In particular, to prove the existence of action ground-states, one must understand the behavior of minimizing bound-state sequences, which have a very rigid structure on half-lines: if $u(\mathbf{0})\neq 0$, they coincide with portions of the real-line bound-state (called \textit{soliton}), whereas if $u(\mathbf{0})=0$ then $u_h=0$ for every $h\in \Hcal$. Using the compactness properties of such sequences and a blow-up argument, we prove
\begin{theorem}\label{thm:ags}
    Let $\mathcal{G}$ be a single-knot metric graph. If $\mathcal{G}$ admits a bound-state and
\begin{equation}\label{eq:tanh}
        \sum_{\substack{\kappa\in \mathcal{K}\\ \kappa\ \text{pendant}}} \tanh{\ell_\kappa} +2\sum_{\substack{\kappa\in \mathcal{K}\\ \kappa\ \text{loop}}} \tanh{\frac{\ell_\kappa}{2}}  \neq H-2n,\quad n=0,\dots, [H/2], 
\end{equation}
where $[\cdot]$ is the integer part of a positive real number, 
   then there exists an action ground-state on $\mathcal{G}$.
\end{theorem}

 \begin{remark}
There are situations where it is simple to show the existence of bound-states:
\begin{enumerate}
\item if $\mathcal{G}$ has two pendants $e_1,e_2$ with the same length $\ell$, let $v$ be a positive solution of
\[
-v''+v=v^{p} \quad \text{ in } [0,\ell],\qquad v(0)=v'(\ell)=0.
\]
Then the function $u$ defined by
\[
u=v \text{ in } e_1,\qquad u=-v\text{ in } e_2,\qquad u=0 \text{ elsewhere,}
\]
is a bound-state. A similar construction provides a bound-state when $\mathcal{G}$ contains two loops with the same length\footnote{In both situations, there are actually infinitely many bound-states (all sign-changing).}. 
\item If the number of half-lines is even, choose a pairing of half-lines. Take $u$ to be constant equal to 1 on the compact core and equal to a real-line soliton on each pair of lines. Then $u$ is clearly a bound-state over $\mathcal{G}$.

\end{enumerate}
Therefore, we have the following direct consequence of Theorem \ref{thm:ags}.
\end{remark}
\begin{corollary}
If a single-knot metric graph $\mathcal{G}$ has either:
\begin{enumerate}
    \item an even number of half-lines;
    \item at least two pendants with the same length;
    \item at least two loops with the same length,
    \end{enumerate}
then there exists a bound-state on $\mathcal{G}$. In particular, if $\mathcal{G}$ also satisfies \eqref{eq:tanh}, then there exists an action ground-state on $\mathcal{G}$.
\end{corollary}

\medskip

If bound-states exist, Theorem \ref{thm:ags} shows that ground-states exist up to a finite number of conditions on the lengths of the egdes. We believe that Theorem \ref{thm:ags} is valid without assuming \eqref{eq:tanh}. In fact, the only known graphs for which $\mathcal{S}_\mathcal{G}$ is not attained are graphs with infinitely many edges (see \cite[Figure 1 and 3]{dovetta2023action}) To illustrate this conjecture, we prove a definitive result in the case of \textit{regular} single-knot graphs.

\begin{theorem}\label{thm:agsregular}
    Let $\mathcal{G}$ be a regular single-knot metric graph. Then there exists an action ground-state on $\mathcal{G}$.
\end{theorem}

\begin{remark}
The strategy for both proofs is the same: taking a minimizing sequence $u_n$, one sees that it weakly convergences in $H^1(\mathcal{G})$ (up to a subsequence) to a function $u$. If $u\not \equiv 0$, then this turns out to be an action ground-state. 

However, there is a major difference between the proofs of Theorems \ref{thm:ags} and \ref{thm:agsregular} in checking that weak limits are nontrivial. In the first case, we study a limiting problem satisfied by non-compact sequences of bound-states (which results in condition \eqref{eq:tanh}), while in the second, we analyze directly non-compact sequences  through the phase-plane portrait. This last argument hinges on the fact that all compact segments have equal length. It would be interesting to extend the procedure to non-regular single-knot (or even general) metric graphs.
\end{remark}

\begin{remark}
In the case of \emph{regular} single-knot graphs, the existence of bound-states is a direct consequence of Theorem \ref{thm:amin} below.
\end{remark}

Having settled the existence of action ground-states, we focus on the qualitative properties of bound-states on regular single-knot metric graphs. For convenience, for these graphs, we make the following convention:
 \begin{equation}\tag{D}\label{dummy}
\text{``For each loop $\kappa$, we introduce a dummy vertex at $x=\ell$.''}
 \end{equation}
In particular, all compact edges in a regular single-knot graph have the same length $\ell$.

\begin{definition}
	Let $\mathcal{G}$ be a regular single-knot metric graph. A function $u\in H^1(\mathcal{G})$ is:
\begin{enumerate}
	\item 	core-symmetric  if, given any two compact edges $\kappa_1$, $\kappa_2$, $u_{\kappa_1}=u_{\kappa_2}$.
\item symmetric if it is core symmetric and $u_{h_1}=u_{h_2}$, for all $h_1,h_2\in \Hcal$.
\item  core-increasing (resp. core-decreasing) if $u_\kappa$ is increasing (resp. decreasing) for all compact edges $\kappa$. 
\item core-monotone if it is either core-increasing or core-decreasing.
{\item increasing if it is core-increasing and decreasing of each half-line\footnote{Recall that the vertex $\mathbf{0}$ is always identified with $x=0$ on each edge.}.}
\end{enumerate}

\end{definition}

In the variational approach, the above properties can be studied for action ground-states by constructing suitable competitors to \eqref{eq:variational} in the space of core-symmetric (or core-monotone) functions. In our case, since we are dealing directly with \eqref{ActionProbI} {and not with a purely variational characterization}, one must construct suitable \textit{bound-state} competitors in order to study the qualitative properties of action ground-states. Therefore, it becomes necessary to first analyze the set of positive core-symmetric and core-monotone bound-states. 

On a single-knot graph, if $u$ is a  bound-state with $u(\mathbf{0})>0$, then there exists $y>0$ such that, on each half-line $h$, 
$$
u_h=\varphi(\cdot+y)\quad \mbox{or}\quad u_h=\varphi(\cdot-y),
$$  
where $\varphi$ is the real-line soliton (see \eqref{eq:portion}). We let $H^+$ (resp. $H^-$) be the number of half-lines on which the first (resp. second) alternative holds. We define the \emph{incidence index of $u$} as 
\begin{equation}\label{eq:defitheta}
    \theta= \frac{H^+-H^-}{P+2L}.
\end{equation}
 The next result gives a characterization of all core-monotone solutions of
\begin{equation}\label{eq:bs}
 	-u''+u=|u|^{p-2}u,\quad u>0, \quad u\in H^1(\mathcal{G})\mbox{ core-symmetric},
 \end{equation}
in terms of their incidence index $\theta$.

\begin{theorem}\label{thm:amin}
	Consider a regular single-knot metric graph $\mathcal{G}$ of length $\ell$. Set
    $$
    \ell^*_1:=\frac{\pi}{\sqrt{(p-2)}},\qquad \ell^*_2=\ell^*_2(\theta):=\argsinh\left(\frac{\theta}{\sqrt{1-\theta^2}}\right).
    $$
	\begin{enumerate}
		\item For $\theta=0$,
		\begin{enumerate}
			\item For every $\ell>0$, there exists a unique solution to \eqref{eq:bs} which is constant on compact edges.
			\item If $\ell\leq\ell^*_1$, there are no strictly core-monotone solution to \eqref{eq:bs}.
			\item  If $\ell>\ell^*_1$, there exists a unique core-increasing solution to \eqref{eq:bs} and a unique core-decreasing solution to \eqref{eq:bs}.
		\end{enumerate}
		\item For $\theta\in(0,1)$, there exists a \textit{unique} core-monotone solution to \eqref{eq:bs}, which is core-increasing.
        \item For $\theta\in(-1,0)$, 
		if $\ell<\ell^*_2$, there are no core-monotone solutions to \eqref{eq:bs}. If $\ell\geq\ell^*_2$, there exists a unique core-monotone solution to \eqref{eq:bs}, which is core-decreasing.
		\item For $\theta=1$, for every $\ell>0$, there exists a unique solution to \eqref{eq:bs}, which is core-increasing.

        \item For $\theta\leq -1$, there are no core-monotone solution to \eqref{eq:bs}.
		\item For $\theta>1$, \eqref{eq:bs} admits core-monotone solutions for every $\ell>0$. Moreover:
		\begin{enumerate}
			\item If $\theta\leq2$, for every $\ell>0$, there exists a unique core-monotone solution to \eqref{eq:bs}.
			\item there exists $\theta^*>2$ such that, for any $\theta>\theta^*$, there exist $0<\ell_1(\theta)<\ell_2(\theta)$ for which, for every $\ell\in[\ell_1(\theta),\ell_{2}(\theta)]$, there exist \textit{multiple} core-increasing solution to \eqref{eq:bs}.
		\end{enumerate}
	\end{enumerate}
\end{theorem}

As a consequence of Theorem \ref{thm:amin}, we can fully characterize core-monotone positive bound-states on some particular single-knot graphs (see Subsection \ref{sec:exemplos}).

\begin{remark}
    Theorem \ref{thm:amin} will follow from a detailed analysis of the orbits of positive symmetric core-monotone bound-states on the phase-plane portrait. In particular, one must relate the energy level of the orbit with the length $\ell$. A related analysis was carried out in \cite{BenguriaDepassierLoss} (see also \cite{Chicone}), where the authors show a monotonic relation between the energy level and the \emph{period} of solutions to \eqref{eq:bs} on the real line. We also refer to \cite{Esteban,kairzhan2021standing, noja2020standing} for other phase-plane analysis involving parts of the period function. As we shall see in Section \ref{sec:PositiveSol}, in our situation the graph structure induces jumps in the energy levels and the problem becomes much more complex.
\end{remark}

\medskip

For small $\ell$, we can greatly improve Theorem \ref{thm:amin}, by proving that each incidence index corresponds to (at most) a unique positive bound-state, up to permutations in the half-lines. In particular, action ground-state are core-symmetric. From now on, the topological structure of the graph will remain fixed and only the length $\ell$ will vary.

\begin{theorem} \label{thm:small_ell}
{Take $\mathcal{G}$ a regular single-knot graph of length $\ell$ with $H$ half-lines.} For $\ell$ sufficiently small, given $H^+,H^-\ge 0$ with $H=H^++H^-$,
    \begin{enumerate}
        \item if $H^+>H^-$, there is a unique positive bound-state with incidence index $\theta$. This bound-state  is core-symmetric and core-increasing. 
        \item If $H^+=H^-$,  there is a unique positive bound-state with incidence index $\theta$. This bound-state is constant equal to 1 on the compact core. 
        \item If $H^+<H^-$, there are no positive bound-states with incidence index $\theta$.  
    \end{enumerate}
Regarding action ground-states for small $\ell$:
    \begin{enumerate}
\item[(a)] Action ground-states of $\mathcal{G}$ are (up to sign multiplication) positive. In particular, they are core-symmetric, and either core-increasing or constantly equal to 1 on compact edges. 
\item[(b)] If $H=1$, then there exists a unique action ground-state, up to sign multiplication.
\end{enumerate}
\end{theorem}

\begin{remark}\label{rmk:open}
    Despite this precise characterization of bound-states for small $\ell$, even though they are core-symmetric, it is not clear whether the action ground-state is always symmetric (that is, if $H^+=H$). As $\ell$ tends to 0, the action of all bound-states converges to $H$ times the action of a half-soliton. As such, to determine the action ground-state, one must perform a first-order expansion on the action near $\ell=0$, depending on the incidence index $\theta$. We leave this as an open problem. 
\end{remark}

For large $\ell$, assuming that the variational approach is available,  it is natural to expect that action ground-states concentrate on a single pendant (if $P\neq 0$) or on a single loop (if $P=0$). Indeed, as $\ell\to \infty$, in each edge, the action ground-state should approach either zero or a half-soliton. Since the minimization problem \eqref{eq:variational} is sub-additive, minimizers tend to not split into separate bubbles, which implies the concentration on a single pendant or loop.

In our context, the proof of the concentration phenomenon must rely on an explicit construction of \textit{bound-states} concentrating on a given edge. Our next result states that, for large $\ell$, given any number of pendants and loops, one may find bound-states concentrating exactly on those edges. For other concentration results, see \cite{pistoia_2, pistoia}.

\begin{theorem}
\label{prop:existsymbreak1} Let $\mathcal{G}$ be a regular single-knot metric graph of length $\ell$. Fix $0\le P_0 \le P$ and $0\le L_0\le L$ such that $E_0:=P_0+2L_0>0$. Take a subgraph $\mathcal{G}_0$ of $\mathcal{G}$ consisting of $P_0$ pendants and $L_0$ loops. Given $\epsilon>0$, if $\ell$ is sufficiently large, there exists a bound-state $v$ such that
\begin{enumerate}
    \item $0<v(\mathbf{0})<\epsilon$.
    \item $v$ is core-symmetric in $\mathcal{G}_0$ and $\mathcal{G}\setminus \mathcal{G}_0$.
        \item for each compact edge $e \notin \mathcal{G}_0$, $v_e$ is decreasing and $\|v_e\|_{L^\infty(0,\ell)}<\epsilon$.
            \item for each edge $e\in \mathcal{G}_0$, 
    $\|v_{\kappa}-\varphi(\ell-\cdot)\|_{L^\infty(0,\ell)}=O(\epsilon)$.
    \item for each half-line $h$, $v_h$ is decreasing.

    \end{enumerate}
    In particular,
    $$
    S(v) = \frac{E_0}{2}S(\varphi)+ O(\epsilon).
    $$
\end{theorem}

Combining this result with the fact that, for $\ell$ large, the action of symmetric bound-states is at least  $\min\left\{H, (P+2L)/2\right\}S(\varphi)$ (see Proposition \ref{prop:actionsym} below), we are able to deduce the following symmetry-breaking result of action ground-states for large $\ell$.

\begin{theorem}\label{cor:symmetry}
	Let $\mathcal{G}$ be a regular single-knot graph of length $\ell$, with $P+L>1$. If $P=0$, suppose also that $H\ge 2$. If $\ell$ is sufficiently large, the action ground-state is not {core}-symmetric.
\end{theorem}

When $P\neq 0$, this result can be proved using the variational problem \eqref{eq:variational} (see Remark \ref{rmk:P>0}). However, this approach fails when $P=0$ and one must rely on Theorem \ref{prop:existsymbreak1}.

\begin{remark}
    The case $P=0$ and $H=1$ corresponds to the so-called flower graphs (see, for example, \cite{kairzhan2021standing}). As $\ell\to\infty$, the ground-state should concentrate either on the half-line or on a single loop. In the first case, one expects the solution to be symmetric, while in the second it is not. However, both cases have the same limit action level $S(\varphi)$. Similarly to Remark \ref{rmk:open}, in order to decide which alternative holds for ground-states, one needs to perform a first-order expansion on the action. This is an interesting technical open problem.
\end{remark}

Finally, it remains to consider the case $P+L=1$. 
\begin{theorem}\label{prop:symmetricPL1}
    Let $\mathcal{G}$ be a single-knot graph of length $\ell$ with $P+L=1$ and $H\ge 2$. If $\ell$ is sufficiently large, the action ground-state is symmetric. 
\end{theorem}

\subsection{Applications}\label{sec:exemplos}

We now apply the results to specific regular single-knot graphs. Recall that, for these graphs, Theorem \ref{thm:agsregular} implies the existence of action ground-states. The graphs chosen below are topologically quite simple; however, the properties of action ground-states (and positive increasing symmetric bound-states) are different, which is another indicator of the complexity of the action ground-state problem on single-knot metric graphs.

\medskip

    \paragraph{Example 1 ($T$-graph).} In the $\mathcal{T}$-graph (Figure 1), the incidence index of a bound-state is either $2$, $-2$ or $0$. Comparing with our previous results in \cite{AgostinhoCorreiaTavares}, this corresponds, respectively, to type $\mathcal{A}$, $\mathcal{B}$ and $\mathcal{C}$ solutions. In particular, combining Theorem \ref{thm:amin} with a variational characterization, we recover the results shown  in \cite[Theorems 1.2 and 1.3]{AgostinhoCorreiaTavares}, to which we refer for more details.

    \medskip
    
    \paragraph{Example 2 (tadpole graph).} In the case of the tadpole graph (Figure 2), since $H=1$, the incidence index of a positive bound-state is $\theta=\pm 1/2$. Then, using Theorem \ref{thm:amin} and the variational characterization \eqref{eq:variational}, we have

\begin{proposition}
\label{cor:tadpole}
	Fix $\ell>0$, $p>2$ and consider the tadpole graph. Then any bound-state is symmetric. Moreover:
	\begin{enumerate}
		\item there exists a unique increasing positive bound-state, which corresponds to the (unique) action ground-state.
		\item if $\ell<\argsinh\left(1/\sqrt{3}\right)$, there are no core-decreasing positive  bound-states;
		\item if $\ell\ge\argsinh\left(1/\sqrt{3}\right)$, there exists a unique core-decreasing positive bound-state.
	\end{enumerate}
\end{proposition}

\medskip

\paragraph{Example 3 (fork graph).}

Similarly, in the case of fork-graphs (Figure 3), again since $H=1$, the incidence index is either $\theta=1/P$ or  $\theta=-1/P$ and we may use Theorem \ref{thm:amin} to characterize positive symmetric increasing bound-states. Complementing with Theorems \ref{thm:small_ell} and \ref{cor:symmetry}, we have

\begin{corollary}\label{coro:fork}
	Fix $\ell>0$, $p>2$ and consider the regular fork graph. Then
	\begin{enumerate}
		\item there exists a unique increasing symmetric positive  bound-state.
		\item if $\ell<\argsinh\left(1/\sqrt{P^2-1}\right)$, there are no core-decreasing symmetric positive  bound-states.
		\item if $\ell\ge\argsinh\left(1/\sqrt{P^2-1}\right)$, there exists a unique core-decreasing symmetric positive  bound-state.
	\end{enumerate}
    Concerning action ground-states,
    \begin{enumerate}
        \item [(a)] if $\ell$ is small, then the action ground-state is the unique increasing symmetric positive bound-state.
        \item [(b)] if $\ell$ is sufficiently large, then action ground-states are not core-symmetric.
    \end{enumerate}
\end{corollary}

The existence of energy ground-states in fork graphs with either two or three pendants has been considered in \cite{adami2016threshold}, where the authors used the variational characterization and a clever construction of test functions. Their argument can be easily adapted to the variational problem \eqref{eq:variational}. However, for larger number of half-lines, their construction fails. It is an open problem whether this variational problem has indeed a solution for general $\ell$. This is yet another evidence of the applicability of our nonvariational methods.

\medskip

\paragraph{Example 4 (broom graph).} In the previous three examples, we discussed three classes of graphs where there is uniqueness of increasing symmetric bound-states. This is not true for broom graphs (Figure 4), as one can see by taking $\theta=H$ in Theorem \ref{thm:amin}. Moreover, combining this result with Remark \ref{rmk:falha_var}, these graphs give an example for the failure of the variational approach:
\begin{corollary}\label{coro:broom}
    Consider the broom graph of length $\ell$. 
    \begin{enumerate}
        \item if the number of half-lines is sufficiently large,  there exist multiple increasing symmetric positive bound-states, all with incidence index $\theta=H=H^+$.
        \item if $H\ge 3$ and $\ell$ is small, then the variational problem \eqref{eq:variational} does not have a solution.
    \end{enumerate}
\end{corollary}
Even if one proves that the action ground-state is increasing, it still remains an open problem to determine which, among all bound-states, corresponds to action ground-states.

\subsection{Structure of the paper}

The structure of this work is as follows. In Section \ref{sec:PositiveSol}, we study the  solutions to \eqref{eq:bs}, through a phase-plane analysis, and prove Theorem \ref{thm:amin}. Section \ref{sec:exist} is devoted to the existence of action ground-states, either directly (Theorems \ref{thm:ags} and \ref{thm:agsregular}) or through the variational approach \eqref{eq:variational}. In particular, the proof of Proposition \ref{cor:tadpole} can be found at the end of this section. In Section \ref{ExUniqofPosSolution}, we study regular single-knot graphs for small $\ell$ (proving Theorem \ref{thm:small_ell}, while the analysis for large $\ell$ (Theorems \ref{prop:existsymbreak1} \ref{cor:symmetry} and \ref{prop:symmetricPL1}) is carried out in Section \ref{sec:large_ell}.

\section{Characterization of positive symmetric bound-states}\label{sec:PositiveSol}

 We call a bound-state of
\begin{equation}\label{StationaryNLS_Chapter2}
-u''+u=|u|^{p-2}u \qquad \text{ in } \mathcal{G},
\end{equation}
to a nontrivial weak solution $u\in H^1(\mathcal{G})$, that is,
	\begin{equation}\label{Weak-Solution}
		\int_{\G}u'\eta'+ u\eta-|u|^{p-2}u\eta=0,\ \text{for all}\ \eta\in H^1(\G).
	\end{equation}
We recall that $H^1(\mathcal{G})=\left\{u: \mathcal{G} \rightarrow \R \mid u \in C(\mathcal{G}) \text { and } u_e \in H^{1} \text{ for every edge } \mathcal{E}\right\}$. 
An immediate corollary is that $u_e$ is an $H^2$ function in each edge, satisfies the equation in the strong sense and, at each vertex $\mathbf{v}$, we have the Neumann-Kirchoff conditions
\begin{equation}\label{eq:nk}
    \sum_{e \prec \mathbf{v}} \frac{d u_e}{dx_e}(v)=0,
\end{equation}
where $e\prec \mathbf{v}$ means that the edge $e$ is incident to the vertex $\mathbf{v}$, and the derivatives are taken in the inward direction of each edge.

Define
$$ F(x,y)=-\frac{1}{2}y^2+f(x),\quad f(x)=\frac{1}{2}x^2-\frac{1}{p}|x|^p,\ \text{for all}\ (x,y)\in\R^2.$$
The contour plot of $F$ is the $(u,u')$-phase-plane diagram for  \eqref{StationaryNLS_Chapter2}. 
\begin{figure}[h]
	\centering
	\includegraphics[height=6cm]{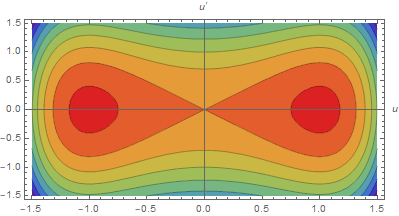}
	\caption{Contour plot of $F$.}
	\label{fig:Phase-Plane}
\end{figure}

As a phase-plane diagram, we observe that the level set of $F$ corresponding to $C=0$ consists of two homoclinic orbits. These orbits correspond to $\pm \varphi$, where $\varphi$ is the real-line soliton given explicitly by
	\begin{equation*}
	\varphi(x)=\left(\frac{ p}{2}\right)^\frac{1}{p-2}\sech^{\frac{2}{p-2}}\left(\frac{p-2}{2}x\right),\ x\in\R.
\end{equation*}
The region of the phase-plane bounded by these homoclinic orbits contains all the level curves of $F$ for positive values of $C$, while the unbounded region contains all the level curves with negative values of $C$.

Given a bound-state $u\in H^1(\G)$, on each compact edge $\kappa$, we have the existence of a constant $C_\kappa$ such that
\begin{equation}\label{ODE_Chapter2}
	-\frac{1}{2}{u'_\kappa}^2(x)+\frac{1}{2}u^2_\kappa(x)-\frac{1}{p}|u_\kappa(x)|^p=C_\kappa,\ \text{for all}\ x\in [0,\ell_\kappa].
\end{equation}
On the other hand, for unbounded edges $h$, since $u_h$ decays to zero as $|x|\to+\infty$, one must have $C_h=0$ and $u_h$ must be a portion of the soliton, that is, there exists $y_h\ge 0$ such that
\begin{equation}\label{eq:portion}
    u_h(x)=\pm\varphi(x\pm y_h),\  \text{for all}\  x\geq0.
\end{equation}

Evidently, if $\kappa$ is a compact edge, the value of $C_\kappa$ can be arbitrary. However, at a given vertex $\mathbf{v}$, the Neumann-Kirchoff conditions impose particular relations between the values of $C_\kappa$ on each edge $e\prec\mathbf{v}$. Consequently, it becomes clear that the specific topology of the graph determines both the number of constants $C_\kappa$ (one per compact edge) and the number of restrictions on these constants (one per nonterminal  vertex). Unless the topology is relatively simple, one finds a large dimensional problem, which may even be underdetermined. 

At this point, we are forced to make some simplifications on the problem at hand. In fact, even the complete analysis of the problem when all the values $C_e$ coincide (as it is the case of single-knot graphs) becomes quite intricate, as we will see below. It is a quite interesting (and hard) problem to perform an analogous analysis when various levels $C_\kappa$ are involved.

\begin{lemma}\label{characterization of bound-states}
    Let $\mathcal{G}$ be a regular single-knot metric graph of lenght $\ell$, and let $u$ be a positive core-symmetric bound-state. Define $\theta$ as in \eqref{eq:defitheta} Then, there exists $y>0$ such that:
    \begin{enumerate}
    \item for each half-line $h$, either
    $u_h(x)=\varphi(x+y)$, or $u_h(x)=\varphi(x-y)$.
    \item Define $z=\varphi(y)=\varphi(-y)\in \left(0,(p/2)^{1/(p-2)}\right)$. For every compact edge $\kappa$, we have
    \begin{equation}\label{gen.IVP2_0}
	-u_\kappa''+u_\kappa=u_\kappa^{p-1}\ \text{in}\ [0,\ell],\quad
	u_\kappa(0)=z,\quad
	u_\kappa'(0)=\theta\sqrt{2f(z)},\quad u_\kappa'(\ell)=0.
\end{equation}
In particular,
\begin{equation}\label{eq:hamil}
    F(u_\kappa,u_\kappa')=C=(1-\theta^2)f(z).
\end{equation}
    \end{enumerate}
\end{lemma}
\begin{proof} 
 On the half-lines, by \eqref{eq:portion} and the continuity of $u$, one must have $u_h(x)=\varphi(x\pm y)$ for some $y\ge0$ independent on the half-line. If $y=0$, then the Neumann-Kirchoff conditions at $\mathbf{0}$ imply $u_\kappa(0)=\varphi(0),\  u_\kappa'(0)=0=\varphi'(0)$
and thus, by uniqueness, $u_\kappa=\varphi$. However, this is not compatible with Neumann-Kirchoff at $x=\ell$, $u'(\ell)=0$. Consequently, $y>0$.

As for the second point, the condition $u_e'(\ell)=0$ is a direct consequence of the Neumann-Kirchoff condition at the associated vertex. Moreover, since, on each half-line,
$$
u_h'(0)=\pm\sqrt{2f(u_h(0))}=\pm\sqrt{2f(z)},
$$
the Neumann-Kirchoff condition at $\mathbf{0}$ reads
\[
(P+2L)u_\kappa'(0) + (H^--H^+)\sqrt{2f(z)} = 0,\quad \mbox{i.e.,}\quad u_\kappa'(0)=\frac{H^+-H^-}{P+2L}\sqrt{2f(z)}.
\]
Finally, we have
\[
C=F(u(0),u'(0))=-\frac{1}{2}u'(0)^2+f(u(0))=-\theta^2f(z)+f(z)=(1-\theta^2)f(z).\qedhere
\]
\end{proof}
 Therefore, in order to characterize core-symmetric, monotone positive bound-states (and to prove Theorem \ref{thm:amin}), we wish to understand if, given $\ell>0$ and $\theta\in \R$, there exists  $z\in \left(0,(p/2)^{1/(p-2)}\right)$ for which the overdetermined problem
\begin{equation}\label{gen.IVP2}
	-u''+u=u^{p-1}\ \text{in}\ [0,\ell],\quad
	u(0)=z,\quad
	u'(0)=\theta\sqrt{2f(z)},\quad u'(\ell)=0,\quad {u>0}.
\end{equation}
has a monotone solution (and, if so, if $z$ is unique). 

\begin{theorem}\label{th:SolIVPGeralTheta}
	Let $\ell>0$ and $p>2$. The existence of monotone solutions to the overdetermined problem \eqref{gen.IVP2} can be discussed in terms of $\theta$ as follows.\setlist{nolistsep}
	\begin{enumerate}[noitemsep]
		\item For $\theta=0$, set $\ell^*_1:=\frac{\pi}{\sqrt{(p-2)}}>0$. Then:
		\begin{enumerate}
			\item For every $\ell>0$, $u\equiv1$ is the only constant solution of \eqref{gen.IVP2} and it is associated to the value $z=1$.
			\item If $\ell\leq\ell^*_1$, there are no  monotone solutions of \eqref{gen.IVP2}.
			\item  If $\ell>\ell^*_1$, there exists a unique \textit{increasing} solution of \eqref{gen.IVP2}, which is associated with $z\in(0,1)$, and there exists a unique \textit{decreasing} solution of \eqref{gen.IVP2}, which is associated with $z>1$.
		\end{enumerate}
		\item For $\theta\in(0,1)$: for every $\ell>0$, there exists a \textit{unique} monotone (increasing) solution of \eqref{gen.IVP2}.
        \item For $\theta\in(-1,0)$, set $\ell^*_2=\ell^*_2(\theta):=\argsinh\left(\frac{\theta}{\sqrt{1-\theta^2}}\right)$. Then:
        \begin{enumerate}
        \item if $\ell<\ell^*_2$, there are no monotone solutions of \eqref{gen.IVP2}. 
        \item If $\ell\geq\ell^*_2$, there exists a \textit{unique} monotone (decreasing) solution of \eqref{gen.IVP2}.
        \end{enumerate}
		\item For $\theta=1$: for every $\ell>0$, there exists a \textit{unique} solution of \eqref{gen.IVP2}, which is increasing and a portion of the soliton.
		\item 		For $\theta\leq -1$:  there are \textit{no} solutions of \eqref{gen.IVP2} for any $\ell>0$.
		\item For $\theta>1$, there exist monotone (increasing) solutions of \eqref{gen.IVP2} for any $\ell>0$. Moreover:
		\begin{enumerate}
			\item If $\theta\leq2$ then, for every $\ell>0$, there exists a \textit{unique} solution of \eqref{gen.IVP2}, which is increasing.
			\item there exists $\theta^*>2$ such that, for any $\theta>\theta^*$, there exist $0<\ell_1(\theta)<\ell_2(\theta)$ for which, for every $\ell\in[\ell_1(\theta),\ell_{2}(\theta)]$, there exist \textit{multiple} monotone (increasing) solutions of \eqref{gen.IVP2}.
		\end{enumerate}
	\end{enumerate}
\end{theorem}

With this result at hand, we can do the following.
\begin{proof}[Proof of Theorem \ref{thm:amin}] This is now a direct consequence of Lemma \ref{characterization of bound-states}, which reduces the study of core-symmetric solutions of \eqref{eq:bs} to the study of the overdetermined problem \eqref{gen.IVP2}, and Theorem \ref{th:SolIVPGeralTheta}.
\end{proof}

The remainder of this section is devoted to the proof of Theorem \ref{th:SolIVPGeralTheta}. It will be a consequence of Propositions \ref{prop:2.61} and \ref{prop:2.74} below, as well as some phase-plane analysis. Before we focus on constructing solutions to \eqref{gen.IVP2}, we make some general considerations concerning the phase-plane of the ODE in \eqref{gen.IVP2}.

\subsection{Preliminaries: phase-plane analysis and nonexistence results.}

In this subsection, we perform a preliminary phase-plane analysis, splitting the discussion with respect to $\theta$, and prove the existence result in item 4. and the nonexistence results in items 5. of Theorem  \ref{th:SolIVPGeralTheta}. Observe that, given $\theta\in \R$ and $z=(p/2)^{\frac{1}{p-2}}$, the IVP
\begin{equation}\label{eqn.probu_3.3}
	-u''+u=|u|^{p-2}u,\ \text{in}\ \R_0^+,\quad  u(0)=z,\quad
	u'(0)=\theta\sqrt{2f(z)}.
\end{equation}
has a solution in the whole real line, and $F(u,u')=C=(1-\theta^2)f(z)$.

Before we begin, we observe that $f$ is strictly increasing in $(0,1)$ and strictly decreasing in $(1,\infty)$, with maximum $f(1)=(p-2)/2p$. As such, we can define 
$$
f_1^{-1}:\left[0,\frac{p-2}{2p}\right]\to [0,1]
$$
as the inverse of $f$ over $[0,1]$ and
$$
f_2^{-1}:\left(-\infty,\frac{p-2}{2p}\right]\to [1,\infty)
$$
 as the inverse of $f$ over $[1,\infty)$.

\paragraph{The case $\theta^2=1$.} In this case, the solution $u$ of \eqref{eqn.probu_3.3} satisfies $u'(0)=\pm \sqrt{2f(z)}$ and $C=0$. Therefore, we have the following result, whose proof is immediate (see also Figure \ref{fig:Theta_1PhasePlane}).

\begin{lemma}\label{lemma:Theta=1}
Let $u$ be a solution of \eqref{eqn.probu_3.3}. Then, for $\theta^2=1$,  $u$ is a portion of the real line soliton. In particular:
\begin{itemize}
\item if $\theta=-1$, then $u'$ does not vanish  in $\R^+$. The overdetermined problem \eqref{gen.IVP2} does not have a solution. 
\item if $\theta=1$, then $u'$ has a unique positive zero $L$, with  $u(L)=\varphi(0)=(p/2)^{1/(p-2)}$. Given $\ell>0$, \eqref{gen.IVP2} has a unique solution, which is increasing.
\end{itemize}
\end{lemma}
\noindent This lemma proves item 4. in Theorem \ref{th:SolIVPGeralTheta}, and item 5. with $\theta=-1$.

\begin{figure}[h!]
\centering
		\includegraphics[height=4cm]{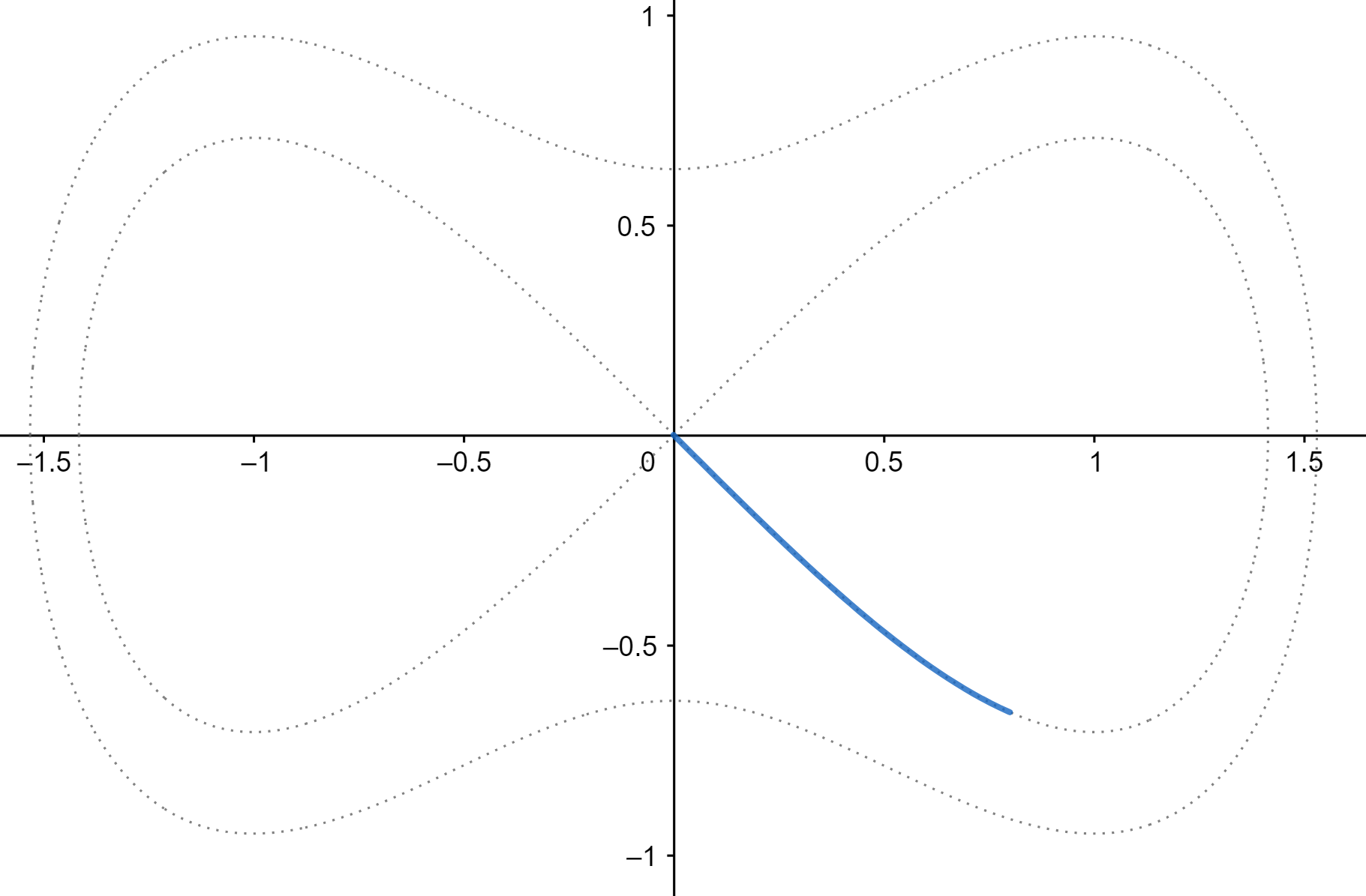}\hspace{2cm}
		\includegraphics[height=4cm]{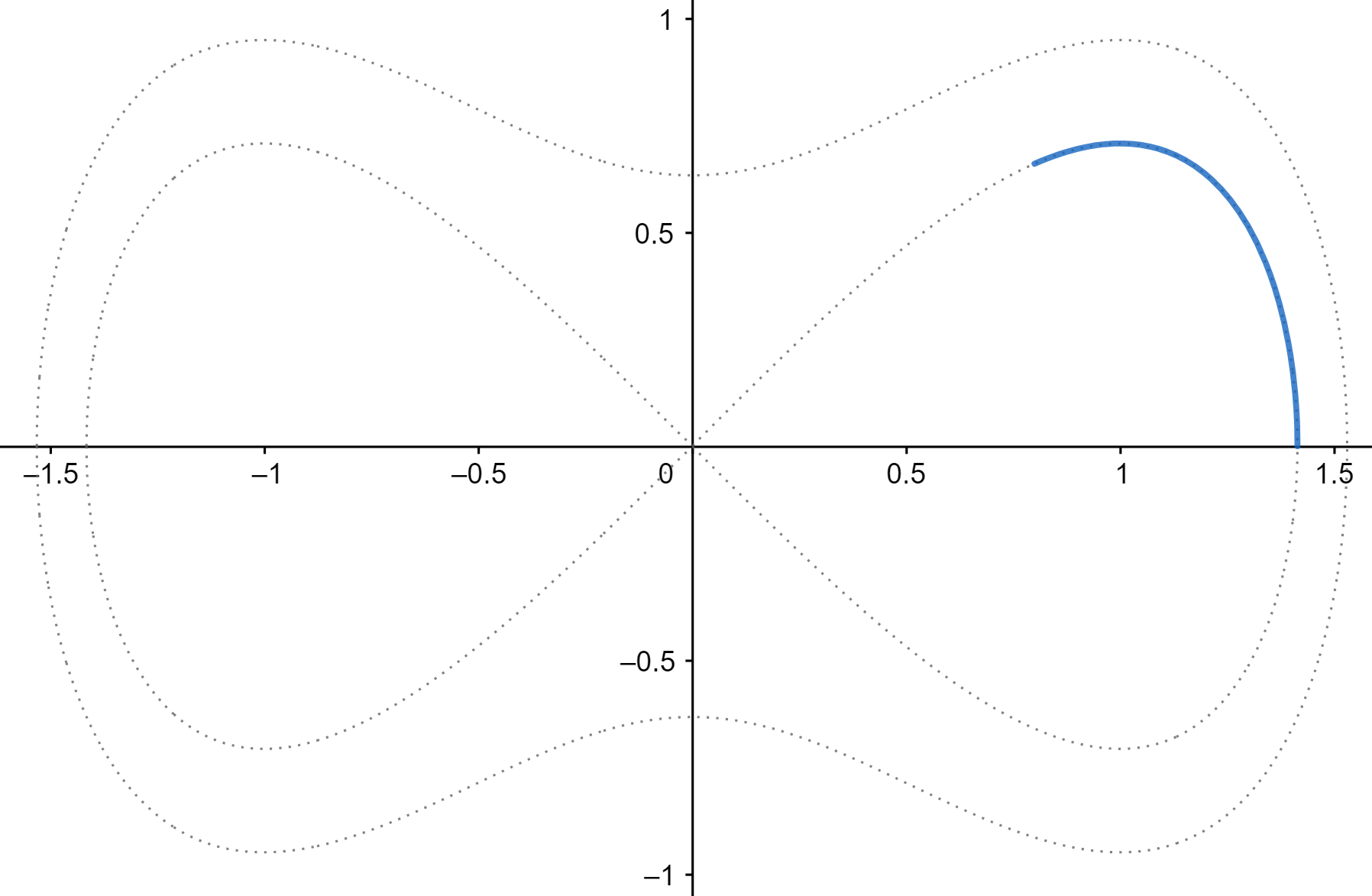}
		\caption{The orbit of a  solution of \eqref{eqn.probu_3.3} with $\theta=-1$ (on the left) and with $\theta=1$ (on the right).}\label{fig:Theta_1PhasePlane}
\end{figure}

\paragraph{The case $\theta=0$.} In this case, a solution $u$ of \eqref{gen.IVP2} is such that $u'(0)=0$. Observe that necessarily $z\in\left(0,(p/2)^{1/(p-2)}\right)$, otherwise the solution is explicitly given by half a soliton, whose derivative does not vanish for $x>0$. 

Since we are interested in monotone solutions, let $L>0$ be the \emph{first} positive zero of $u'$.
Since $u'(0)=0=u'(L)$, we have
\[
0<f(u(0))=F(u(0),0)=C=F(u(L),0)=f(u(L)).
\]
Since $u$ is strictly monotone on  $(0,L)$,  $u(L)$ and $u(0)$ are the two unique and distinct (positive) solutions of the equation $f(z)=C$. We have the following cases to consider:
\begin{lemma}\label{PHASEPLANETHETA=0}
Let $u$ be a solution of \eqref{gen.IVP2} with $\theta=0$.
	\begin{enumerate}
		\item If $z\in (0,1)$, then $u(0)=z<1<u(L)$ and $u'>0$ in $(0,L)$, see Figure \ref{fig:Theta=0PhasePlane} (on the left);
				\item If $z=1$, then necessarily $u\equiv 1$;
		\item If $z\in (1,(p/2)^{\frac{1}{p-2}})$, then $u(0)=z>1>u(L)$ and $u'<0$ in $(0,L)$, see Figure \ref{fig:Theta=0PhasePlane} (on the right).
	\end{enumerate}
\end{lemma}

\begin{figure}[h!]
\centering
		\includegraphics[height=4cm]{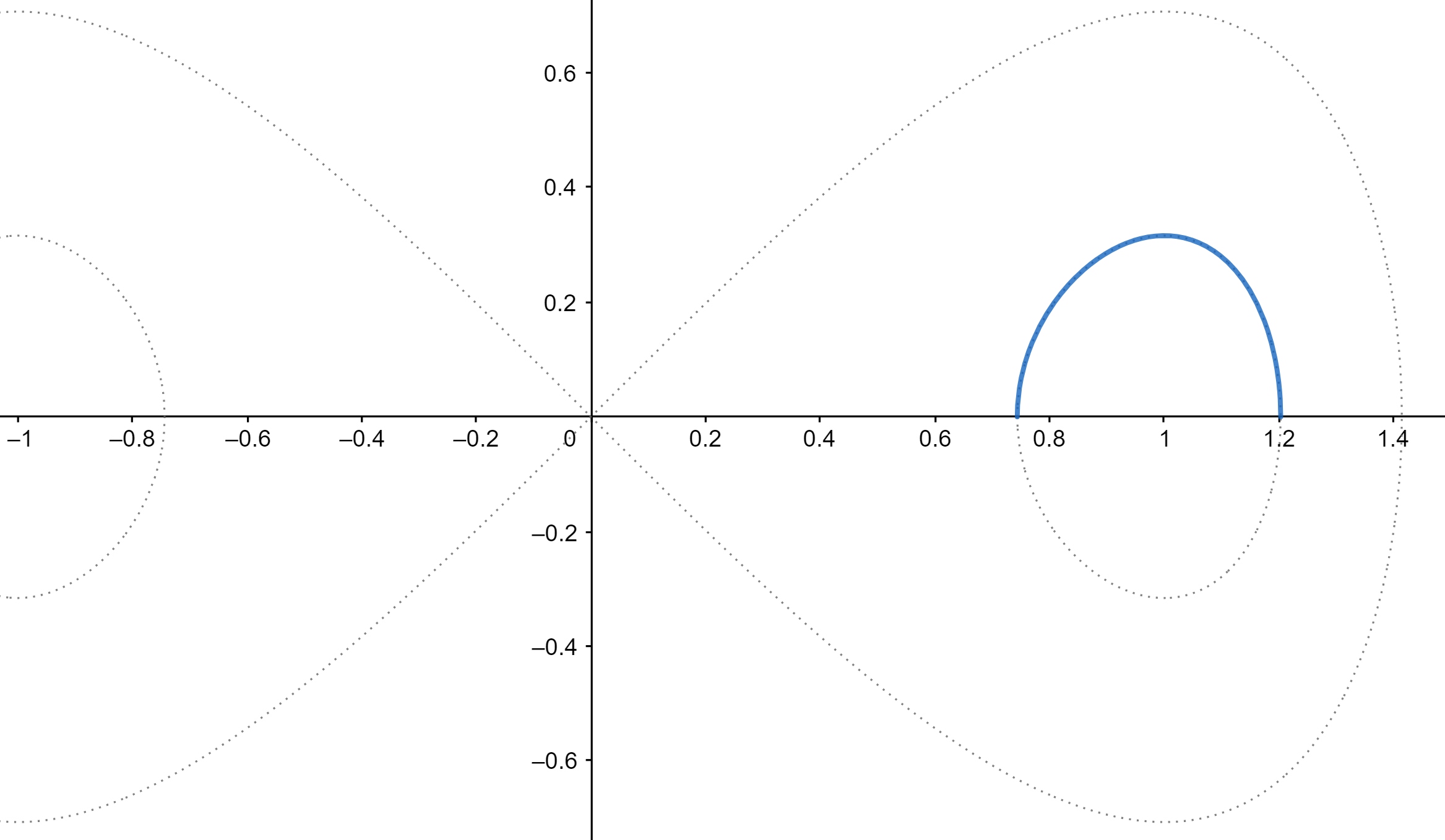}
\hspace{1cm}
		\includegraphics[height=4cm]{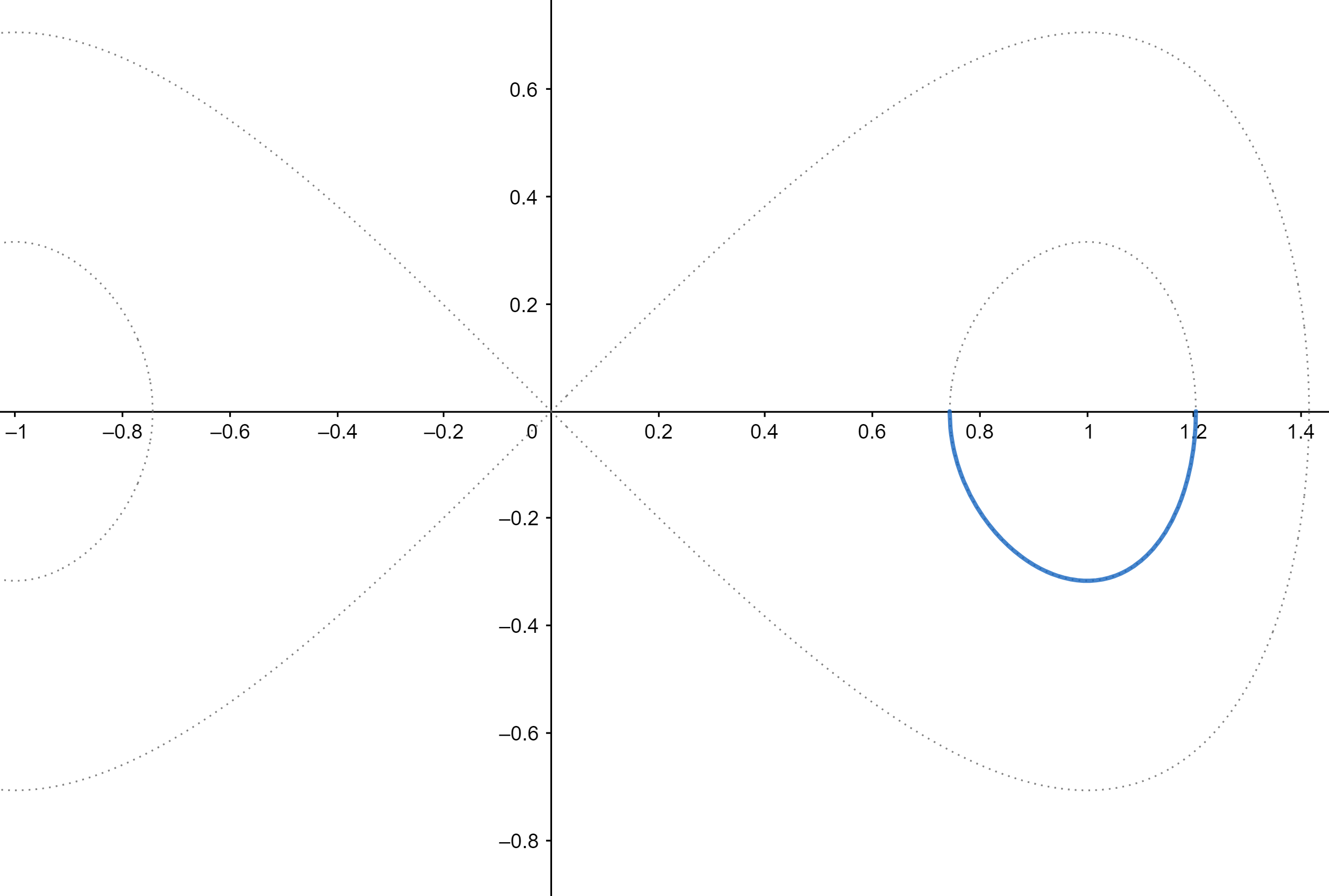}
			\caption{The orbit of an increasing (on the left) and decreasing (on the right) solution of \eqref{gen.IVP2}, with $\theta=0$.}\label{fig:Theta=0PhasePlane}
\end{figure}

\begin{remark}  It is clear from the phase-plane that multiple solutions of \eqref{gen.IVP2} may exist if we drop the requirement of monotonicity.  It suffices that the length $\ell>0$ coincides with any of the zeros of $u'$. 
\end{remark}

\paragraph{The case $\theta^2\in(0,1)$.} For each $z\in \left(0,(p/2)^{1/(p-2)}\right)$, let $u$ be the solution of \eqref{eqn.probu_3.3}. For these values of $\theta$, we have $C=(1-\theta^2)f(z)>0$.

We argue as in the previous case. Let $L$ be the first positive zero of $u'$. As such, it satisfies
$$ f(u(L))=F(u(L),0)=C=(1-\theta^2)f(z).$$
Now, if $u$ is increasing in $(0,L)$, we have that $z=u(0)<u(L)$. Otherwise, we obtain $z=u(0)>u(L)$. In the first scenario, we obtain $ u(L)=f_2^{-1}\left((1-\theta^2)f(z)\right),$ and, in the second, $ u(L)=f_1^{-1}\left((1-\theta^2)f(z)\right)$. We depict the orbits of these solutions on the phase-plane in Figure \ref{fig:Theta_0_1PhasePlane} below.
\begin{figure}[h!]
		\centering
		\includegraphics[height=4cm]{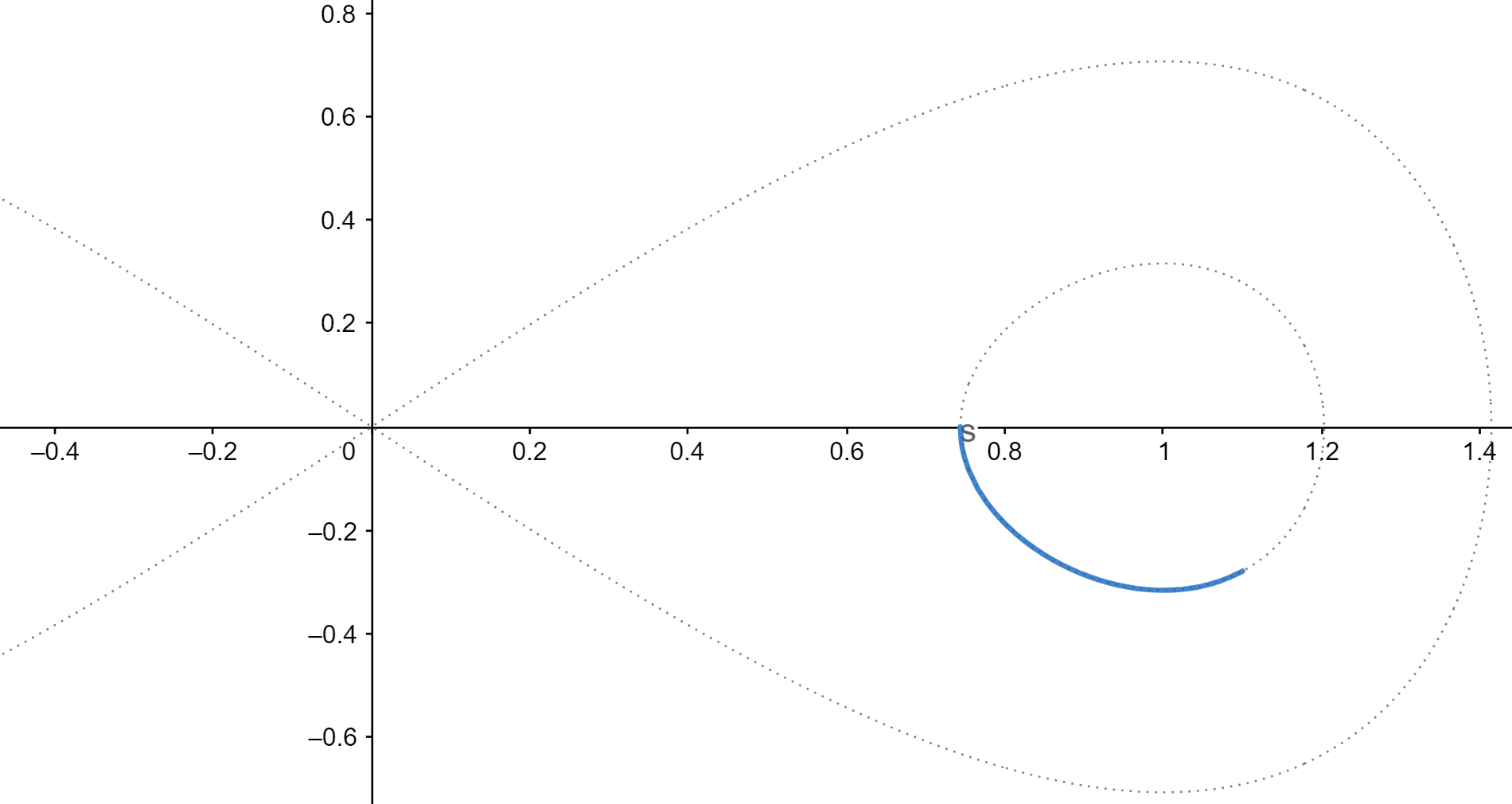}
\hspace{1cm}
		\includegraphics[height=4cm]{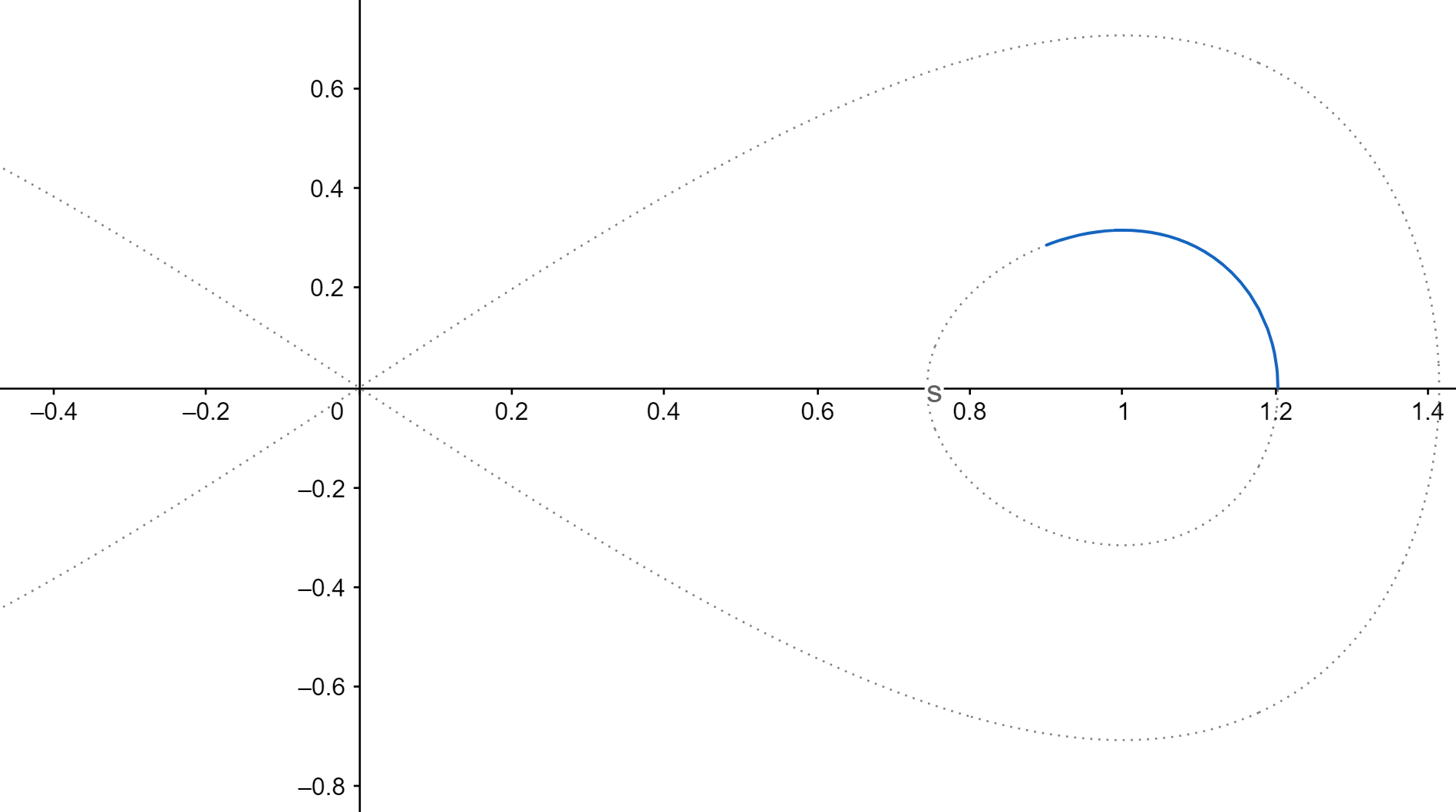}
		\caption{The orbit of a decreasing (on the left) and increasing (on the right) solution of \eqref{gen.IVP2} with $\theta\in(0,1)$.}\label{fig:Theta_0_1PhasePlane}
\end{figure}

\paragraph{The case $\theta^2>1$.} In this case, the orbit of a solution $u$ of \eqref{eqn.probu_3.3} is in the exterior of the homoclinic curve, as $C=(1-\theta^2)f(z)<0$ (see Figure \ref{fig:Theta_Maior_1PhasePlane} for the difference between $\theta>0$ and $\theta<0$). With a simple phase-plane analysis, it is immediate to conclude the following.
\begin{lemma}\label{lemma:Theta<1}
 If $\theta^2>1$, then any solution of \eqref{eqn.probu_3.3} necessarily changes sign. For $\theta<-1$, the solution changes sign before $u'$ vanishes and, in particular, the overdetermined problem \eqref{gen.IVP2} does not admit a positive solution. If $\theta>1$, any positive solution of  \eqref{gen.IVP2} is necessarily increasing.
\end{lemma}

\begin{figure}[h!]
		\centering
		\includegraphics[height=4cm]{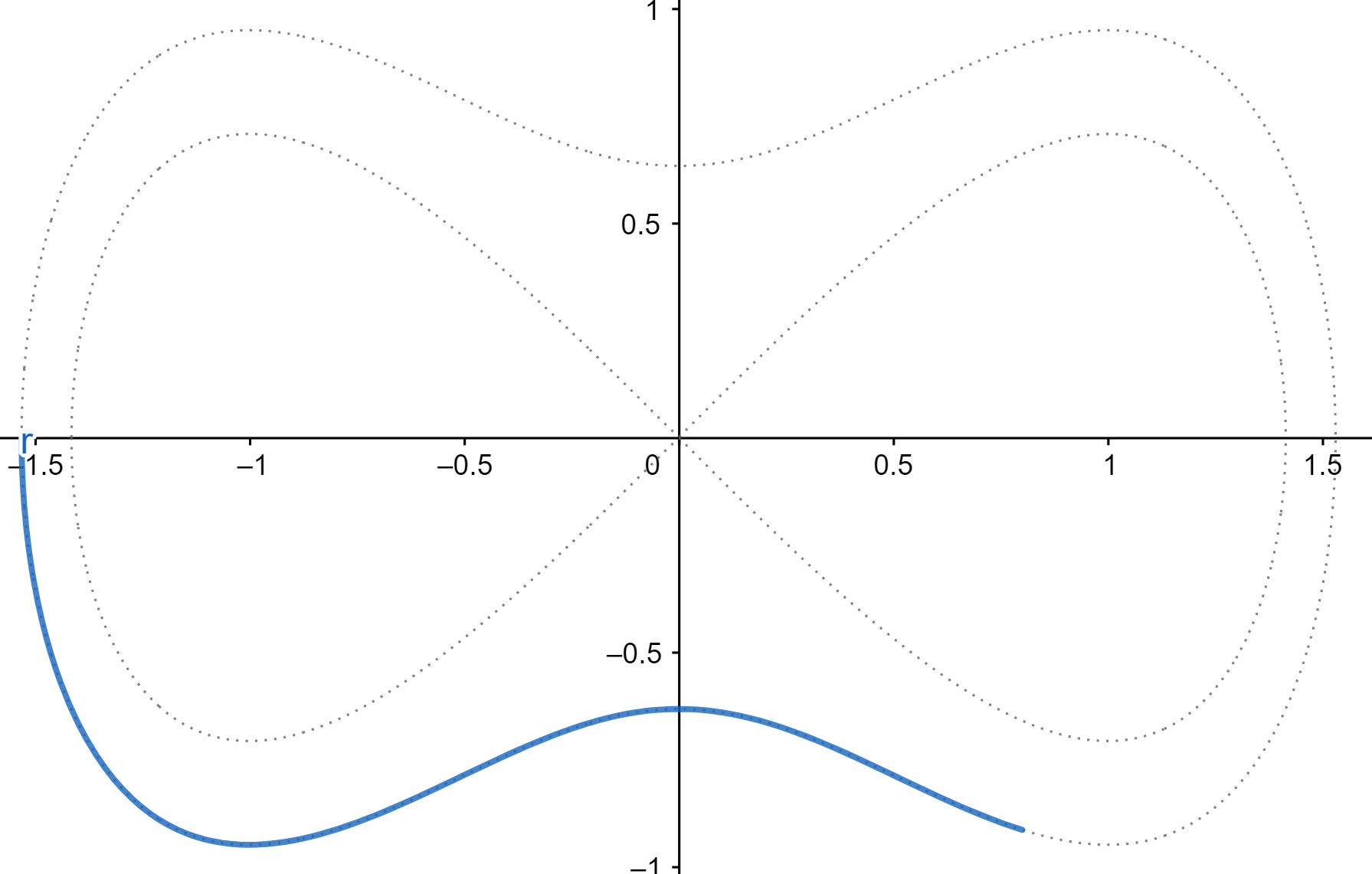}
	\hspace{1cm}
		\includegraphics[height=4cm]{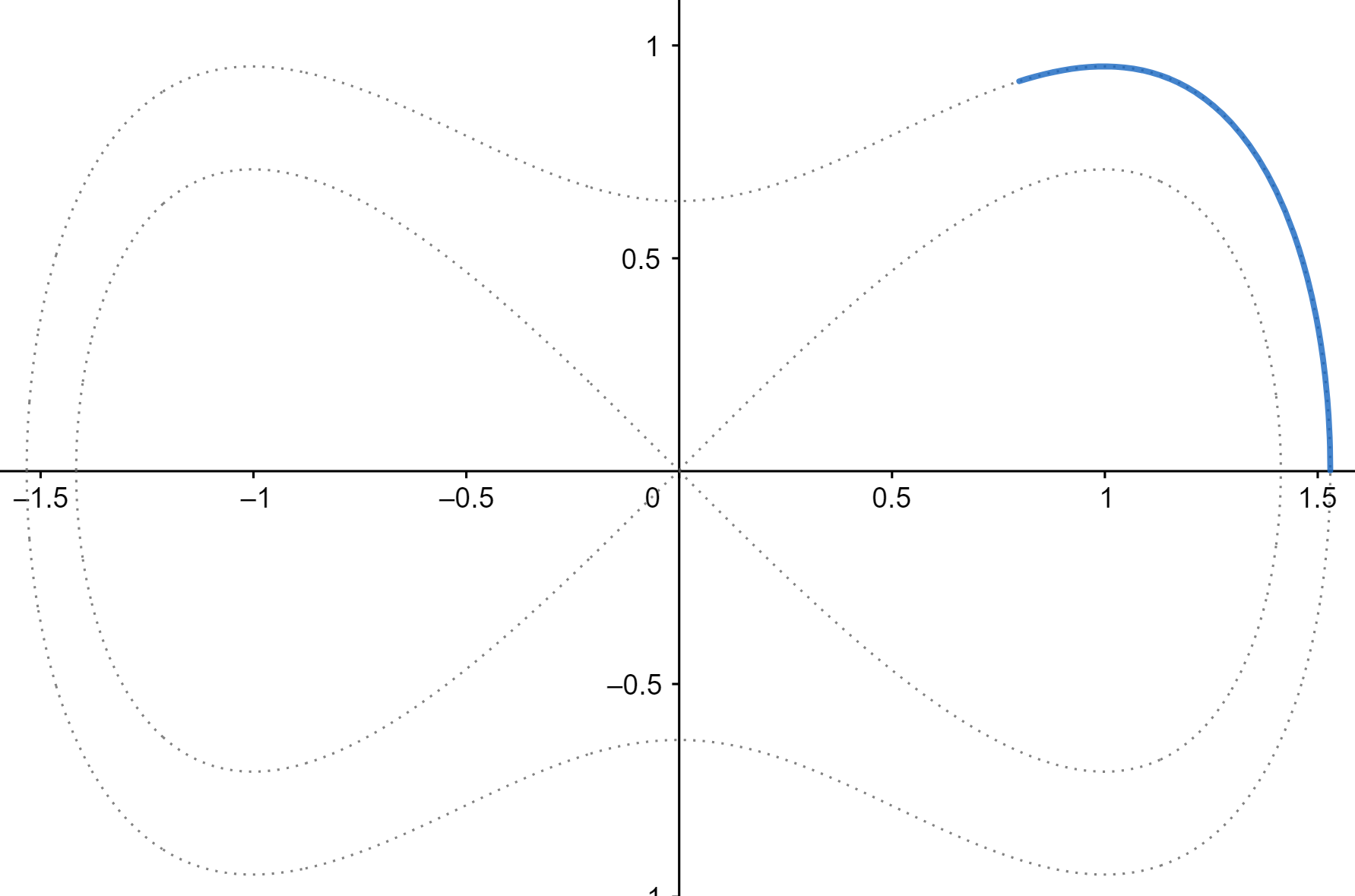}
		\caption{The orbit of a decreasing (on the left) and increasing (on the right) solution of \eqref{gen.IVP2} with $\theta>1$.}\label{fig:Theta_Maior_1PhasePlane}
\end{figure}

\subsection{Classification of increasing solutions of the overdetermined problem \eqref{gen.IVP2}}
To construct increasing solutions of \eqref{gen.IVP2}, let $z\in\left(0,(p/2)^{1/(p-2)}\right)$, we consider once again a solution $u$ of the IVP
\begin{equation}\label{eqn.probu_IVP}
 -u''+u=u^{p-1},\ \text{in}\ \R^+_0,\quad u(0)=z,\quad u'(0)=\theta\sqrt{2f(z)},
 \end{equation}
with $\theta\geq 0$. 

Let $L>0$ be the first positive zero of $u'$. From \eqref{eq:hamil},  we have $(1-\theta^2)f(z)=f(u(L))$. Since $u(L)>u(0)>0$, we have $u(L)=f_2^{-1}\left((1- \theta^2)f(z)\right)$.

Now, we show that there exists a relation between $u(L)$ and $u(0)=z$. Since $u'>0$ in $(0,L)$, we have 
$$ L=\int_0^Ldx=\int_{u(0)}^{u(L)}\frac{dt}{u'(u^{-1}(t))}=\int_{u(0)}^{f_2^{-1}\left((1-\theta^2)f(u(0))\right)}\frac{dt}{u'(u^{-1}(t))}.$$
Using {\eqref{eq:hamil}}, and solving for $u'>0$, we obtain
\begin{equation}\label{eq:L_increasing_aux}
L=\frac{1}{\sqrt{2}}\int_z^{f_2^{-1}\left((1-\theta^2)f(z)\right)}\frac{dt}{\sqrt{f(t)-(1-\theta^2)f(z)}}.
\end{equation}

Observe that this argument is valid for each $z\in\left(0,(p/2)^{1/(p-2)}\right)$ if $\theta>0$. If $\theta=0$, the previous construction works for $z\in(0,1)$, but fails for $z\in\left[1,(p/2)^{1/(p-2)}\right)$, recall Lemma \ref{PHASEPLANETHETA=0}.

\begin{definition}[Length function of increasing solutions of the IVP \eqref{eqn.probu_IVP}]\label{def:2.60}
	For $p>2$ and $\theta>0$, we consider the \textit{length function} of \eqref{eqn.probu_IVP} under  as being $L_{1}:\left(0,(p/2)^{1/(p-2)}\right)\to\R^+$ defined by
	\vspace{-0.3cm}
	$$L_{1}(z)=\frac{1}{\sqrt{2}}\int_z^{f_2^{-1}\left((1-\theta^2)f(z)\right)}\frac{dt}{\sqrt{f(t)-(1-\theta^2)f(z)}}.$$
	If $\theta=0$, the function $L_1$, defined by the same expression, is considered for $z\in(0,1)$.
\end{definition}

The existence of increasing solutions to the overdetermined problem \eqref{gen.IVP2}, with $u'(0)\geq 0$, can now be discussed in terms of the function $L_1$. Indeed, given $\ell>0$, if there exists a solution $z$ to the equation
$ L_1(z)=\ell,$
then a positive increasing solution of \eqref{gen.IVP2} exists. Moreover, if such a solution $z$ is unique, then the increasing solution to \eqref{gen.IVP2} is also unique. Thus, the \textit{length function} allows one not only to discuss the existence but also the uniqueness of increasing 
solutions to \eqref{gen.IVP2}.

\begin{proposition}[Qualitative Behavior of $L_1$]\label{prop:2.61}
	For any $p>2$ and $\theta\geq0$, the function $L_{1}$ is positive, differentiable and satisfies: 
    
    $$
    \lim\limits_{z\to 0^+}L_{1}(z)=+\infty,\quad \begin{cases}
        \lim\limits_{z\to (p/2)^\frac{1}{p-2}}L_{1}(z)=0, & \theta>0\\
          L_1(1):=\lim\limits_{z\to1^-}L_1(z)=\frac{\pi}{\sqrt{p-2}}, & \theta=0
    \end{cases}.
    $$
	Moreover, the dependence on $\theta$ is as follows:\setlist{nolistsep}
	\begin{enumerate}[noitemsep]
		\item If either $\theta=0$ and $z\in(0,1)$, or $0<\theta\leq 1$ and $z\in\left(0,(p/2)^{1/(p-2)}\right)$, then $L_{1}$ is strictly decreasing.
		\item If $\theta>1$, then $L_{1}$ is strictly decreasing in $(0,1)$.
		\item For $\theta\in(1,2]$, $L_{1}$ is strictly decreasing in $\left(1,(p/2)^{1/(p-2)}\right)$ and, if $\theta>2$ is sufficiently large, $L_{1}$ is not monotone in $\left(1,(p/2)^{1/(p-2)}\right)$.		 
	\end{enumerate}
\end{proposition}

\begin{remark}
For $\theta=0$, this result is included in \cite[Proposition 2.16]{AgostinhoCorreiaTavares}, while the particular case $\theta=2$ is included in \cite[Proposition 2.5]{AgostinhoCorreiaTavares}. The proof of the items 1. and 2. with a general $\theta$ case uses the same ideias, and in the proof of this part we will only highlight the differences arising in the treatment of a general $\theta$. The knowledge of the case $\theta=2$ is of key importance to show the first part of 3., while the second part is new.
\end{remark}

The rest of this section is devoted to the proof of Proposition \ref{prop:2.61} for $\theta>0$. The following technical lemma, inspired by \cite{kairzhan2021standing}, collects a few useful identities. 
\begin{lemma}\label{lemma:2.62}
	For $\theta>0$ and $z\in \left(0,(p/2)^{1/(p-2)}\right)$,
	\begin{align}\label{eqn.LDiffForm1}
		\sqrt{2}L_{1}(z)\left((1-\theta^2)\right.&f(z)-\left.f(1)\right)=-2\theta\frac{f(z)-f(1)\sqrt{f(z)}}{f'(z)}\\
		&-\int_z^{f_2^{-1}\left((1-\theta^2)f(z)\right)}\left(3-2\frac{(f(t)-f(1))}{f'(t)^2}f''(t)\right)\sqrt{f(t)-(1-\theta^2)f(z)}dt.\nonumber
	\end{align}
	In particular, $L_1$ is differentiable. Moreover, defining $H(z,\theta):=\sqrt{2}L_1'(z)\left((1-\theta^2)f(z)-f(1)\right)$,
	\begin{align}
&H(z,\theta)=\theta\frac{f(1)}{\sqrt{f(z)}}+(1-\theta^2)f'(z)\int_z^{f_2^{-1}\left((1-\theta^2)f(z)\right)}\frac{A(t)dt}{\sqrt{f(t)-(1-\theta^2)f(z)}}\label{LDerivativeForm}\\
		&=\frac{\theta\left(f(1)-2(1-\theta^2)A(z)f(z)\right)}{\sqrt{f(z)}}-(1-\theta^2)f'(z)\int_z^{f_2^{-1}\left((1-\theta^2)f(z)\right)}\frac{g(t)}{f'(t)^4}\sqrt{f(t)-(1-\theta^2)f(z)}dt,\label{H_p(z,theta)NotSingular}
	\end{align}
	where $A:\R^+\to \R$ is given by $A(t)=\frac{1}{2}-\frac{(f(t)-f(1))f''(t)}{{f'}^2(t)},$ and  $g:\R^+\setminus\{1\}\to\R$ is defined by
	\begin{equation*}
		g(t)=-3f''(t)f'(t)^2+6(f(t)-f(1))f''(t)^2-2(f(t)-f(1))f'''(t)f'(t).
	\end{equation*}
\end{lemma}

\begin{proof}
	Observe that
	\begin{align}
		&\sqrt{2}L_1(z)\left((1-\theta^2)f(z)-f(1)\right)=\int_z^{f_2^{-1}\left((1-\theta^2)f(z)\right)}\frac{(1-\theta^2)f(z)-f(1)}{\sqrt{f(t)-(1-\theta^2)f(z)}}dt \nonumber \\
		=&-\int_z^{f_2^{-1}\left((1-\theta^2)f(z)\right)}\sqrt{f(t)-(1-\theta^2)f(z)}dt+\int_z^{f_2^{-1}\left((1-\theta^2)f(z)\right)}\frac{f(t)-f(1)}{\sqrt{f(t)-(1-\theta^2)f(z)}}dt.\label{eq:aux_lemma:2.62}
	\end{align}
	For fixed $z,\theta$ as in the statement, define $\gamma:[z,f_2^{-1}\left((1-\theta^2)f(z)\right)]\to\R^2$ by $$\gamma(t)=(x(t),y(t))=(t,\sqrt{2}\sqrt{f(t)-(1-\theta^2)f(z)}).$$ This parametrizes the level curve $\Gamma$, $-\frac{1}{2}y^2+f(x)=(1-\theta^2)f(z)$, for $y\geq0$. Consider also $W(x,y)=\frac{f(x)-f(1)}{f'(x)}y$, which is of class $C^1$ in $\R^+\setminus\{1\}\times\R$ and can be extended by continuity to $\R^+\times\R$. Then,
	\begin{align*}
		&W\left(f_2^{-1}\left((1-\theta^2)f(z)\right),0\right)-W\left(z,\theta\sqrt{2}\sqrt{f(z)}\right)=\int_{\Gamma}dW(x,y)=\int_\Gamma\frac{\partial W}{\partial x}(x,y)dx+\frac{\partial W}{\partial y}(x,y)dy\\
		=&\int_\Gamma \left(1-\frac{(f(x)-f(1))}{f'(x)^2}f''(x)\right)ydx+\frac{f(x)-f(1)}{f'(x)}dy+\frac{\sqrt{2}}{2}\int_z^{f_2^{-1}\left((1-\theta^2)f(z)\right)}\frac{f(t)-f(1)}{\sqrt{f(t)-(1-\theta^2)f(z)}}dt.
	\end{align*}
Plugging the latter identity in \eqref{eq:aux_lemma:2.62} and using the fact $W\left(f_2^{-1}\left((1-\theta^2)f(z)\right),0\right)=0$, we are able to obtain \eqref{eqn.LDiffForm1}. In particular, since the integrand function contains no singularities, this shows that $L_1$ is differentiable.
	
	Differentiating both sides of \eqref{eqn.LDiffForm1}, we obtain
	\begin{align*}
		&H(z,\theta)+\sqrt{2}L(z)(1-\theta^2)f'(z)=-2\theta\left(\sqrt{f(z)}+\frac{f(z)-f(1)}{2\sqrt{f(z)}}-\frac{f(z)-f(1)}{f'(z)^2}f''(z)\sqrt{f(z)}\right)\\
		&+\left(3-2\frac{f(z)-f(1)}{f'(z)^2}f''(z)\right)\theta\sqrt{f(z)}\\&+(1-\theta^2)f'(z)\int_{z}^{f_2^{-1}\left((1-\theta^2)f(z)\right)}\left(\frac{3}{2}-\frac{f(t)-f(1)}{{f'}(t)^2}f''(t)\right)\frac{dt}{\sqrt{f(t)-(1-\theta^2)f(z)}}\\
		=&\theta\frac{f(1)}{\sqrt{f(z)}}+(1-\theta^2)f'(z)\int_{z}^{f_2^{-1}\left((1-\theta^2)f(z)\right)}\left(\frac{3}{2}-\frac{f(t)-f(1)}{{f'}(t)^2}f''(t)\right)\frac{dt}{\sqrt{f(t)-(1-\theta^2)f(z)}}.
	\end{align*}
	Passing $\sqrt{2}L_1(z)(1-\theta^2)f'(z)$ to the right hand side, \eqref{LDerivativeForm} follows. 
	
	Finally, integration by parts in \eqref{LDerivativeForm} will yield us \eqref{H_p(z,theta)NotSingular}:
	\begin{align*}
		&H(z,\theta)=\theta\frac{f(1)}{\sqrt{f(z)}}+(1-\theta^2)f'(z)\int_z^{f_2^{-1}\left((1-\theta^2)f(z)\right)}\frac{2A(t)}{f'(t)}\frac{d}{dt}\sqrt{f(t)-(1-\theta^2)f(z)}dt\nonumber\\
		=&\theta\frac{f(1)}{\sqrt{f(z)}}+(1-\theta^2)f'(z)\left[-\frac{2A(z)}{f'(z)}\theta\sqrt{f(z)}-\int_z^{f_2^{-1}\left((1-\theta^2)f(z)\right)}\frac{g(t)}{f'(t)^4}\sqrt{f(t)-(1-\theta^2)f(z)}dt\right]\nonumber\\
		=&\frac{\theta}{\sqrt{f(z)}}\left(f(1)-2(1-\theta^2)A(z)f(z)\right)-(1-\theta^2)f'(z)\int_z^{f_2^{-1}\left((1-\theta^2)f(z)\right)}\frac{g(t)}{f'(t)^4}\sqrt{f(t)-(1-\theta^2)f(z)}dt.\  \qedhere
	\end{align*}
\end{proof}

The following result contains useful properties of the function $A$ defined in the previous lemma.
\begin{lemma}{}\label{lemma:2.63}
	Let $A:{\R^+}\setminus\{1\}\to\R$ be the function defined by
	$$A(t)=\frac{1}{2}-\frac{(f(t)-f(1))f''(t)}{{f'}^2(t)},\ \text{for}\  t\neq1.$$ Then, $A$ has a continuous extension to $\R^+$, which we still denote by $A$, with the following properties: $A(t)>0$ for $t\in(0,1)$, $A(t)<0$ for $t\in(1,+\infty)$ and $A(1)=0$. Moreover, $\lim\limits_{t\to\infty}A(t)=-\frac{p-2}{2p}<0$.
\end{lemma}
\begin{proof}
	Firstly, observe that
	\begin{equation*}
		\lim_{t\to1}A(t)=\frac{1}{2}-f''(1)\lim_{t\to1}\frac{f(t)-f(1)}{{f'}^2(t)}=\frac{1}{2}-f''(1)\lim_{t\to1}\frac{f'(t)}{2f'(t)f''(t)}=\frac{1}{2}-\frac{f''(1)}{2f''(1)}=0,
	\end{equation*}
	and thus $A$ admits a continuous extension. To check the sign properties we write $$A(t)=\frac{{f'}^2(t)-2(f(t)-f(1))f''(t)}{2{f'}^2(t)}=:\frac{\eta(t)}{2{f'}^2(t)}.$$
	Since $f'''(z)<0$, a direct computation yields $\eta'(t)=-2(f(t)-f(1))f'''(t)<0$. Moreover, $\eta(1)=0$ and thus, $\eta>0$ in $(0,1)$ and $\eta<0$ in $(1,\infty)$.
	
	Finally, observe that, near infinity, we have $f(t)\sim -\frac{1}{p}t^p$, $f'(t)\sim -t^{p-1}$ and $f''(t)\sim -(p-1)t^{p-2}$. Then,
	\begin{equation*}
		\lim_{t\to\infty}A(t)=\frac{1}{2}-\lim_{t\to\infty}\frac{\left(f(t)-f(1)\right)f''(t)}{{f'(t)^2}}=\frac{1}{2}-\lim_{t\to\infty}\frac{p-1}{p}\frac{t^{2p-2}}{t^{2p-2}}=-\frac{p-2}{2p}<0.\qedhere		
	\end{equation*}
\end{proof}

The qualitative properties of the function $g$ defined in Lemma \ref{lemma:2.62} (particularly, its sign) will play a crucial role in the study of the monotonicity of $L_1$. It is clear, by definition of $g$, that $g(1)=0$. Everywhere else, the sign is determined as follows.
\begin{proposition}\label{prop:gNegativa}
	Let $g:\R^+\to\R$ be defined by the expression in Lemma \ref{lemma:2.62}. Then, $g(t)<0$ for $t\neq1$.
\end{proposition}
\begin{proof}
This corresponds to Claim 2.18 in \cite{AgostinhoCorreiaTavares}, whose proof follows from Lemmas 2.20-2.23 therein.
\end{proof}

\subsubsection{Monotonicity analysis of $L_1$: the case $\theta\in(0,1]$.} 
\begin{lemma}\label{lemma:2.64}
	If $0<\theta\leq 1$ and $z\in\left(0,(p/2)^{1/(p-2)}\right)$, then $z\mapsto L_1(z)$ is strictly decreasing.
\end{lemma}
\begin{proof}
	Recall the definition of $H$ in Lemma \ref{lemma:2.62}. Given that $\theta\leq1$ and $f(1)=\max_{\R^+}f$, we have $((1-\theta^2)f(z)-f(1))<0,\ \text{for all}\ z\in\left(0,\left(p/2\right)^{1/(p-2)}\right).$ Then, it suffices to show that, for fixed $\theta\leq 1$, we have $z\mapsto H(z,\theta)>0$. The case $\theta=1$ is straightforward, since
	$H(z,1)=\frac{f(1)}{\sqrt{f(z)}}>0$.
	
	We split our analysis in the cases $0<z<1$ and $z\in[1,(p/2)^{1/(p-2)})$. We start with the latter. Then we have $(1-\theta^2)f'(z)\leq0$ and, moreover, from Lemma \ref{lemma:2.63}, we have $A\leq0$. Therefore, using \eqref{LDerivativeForm}, $H(z,\theta)\geq \theta f(1)/\sqrt{f(z)}>0$.
	
	Assume now that $0<\theta<1$ and $0<z<1$, and consider the expression of $H(z,\theta)$ given by \eqref{H_p(z,theta)NotSingular}. Since we are in the case where $f'(z),\ (1-\theta^2)>0$, we have
	\begin{equation*}
		-(1-\theta^2)f'(z)\int_z^{f_2^{-1}\left((1-\theta^2)f(z)\right)}\frac{g(t)}{f'(t)^4}\sqrt{f(t)-(1-\theta^2)f(z)}dt>0.
	\end{equation*}
	We claim that $\psi(z,\theta):=f(1)-2(1-\theta^2)A(z)f(z)>0$. From Lemma \ref{lemma:2.63}, it follows that
	$$\frac{\partial}{\partial\theta}\psi(z,\theta)=4\theta A(z)f(z)>0,\ \text{for all}\ (z,\theta)\in(0,1)^2.$$
	To finish, we show that $\psi(z,0)=f(1)-2A(z)f(z)>0$ for any $z\in(0,1)$. Indeed, using the explicit form of $f$, $f'$ and $f''$,
	\begin{align*}
		\psi(z,0)&=f(1)-2\left(\frac{1}{2}-\frac{(f(z)-f(1))f''(z)}{{f'}^2(z)}\right)f(z)=-(f(z)-f(1))+2(f(z)-f(1))\frac{f''(z)f(z)}{{f'}^2(z)}\\
		&=\frac{(f(z)-f(1))}{{f'}^2(z)}\left(-{f'}^2(z)+2f(z)f''(z)\right)
		=\frac{(f(z)-f(1))}{{f'}^2(z)}\frac{(p-2)}{p}z^p\left[-(p-1)+z^{p-2}\right].
	\end{align*}
	Since
	$$\frac{(f(z)-f(1))}{{f'}^2(z)}\frac{(p-2)}{p}z^p<0,$$
	and $-(p-1)+z^{p-2}\leq -(p-2)<0$,
	it follows that $\psi(z,0)>0$, which finishes the proof.
\end{proof}
\subsubsection{Monotonicity analysis of $L_1$: the case $\theta>1$.}

\begin{lemma}\label{lemma:2.65}
	Suppose that $\theta>1$. Then, $z\mapsto L_1(z)$ is strictly decreasing in the interval $(0,1)$.
\end{lemma}
\begin{proof}
This proof is similar to the proof of \cite[Lemma 2.8]{AgostinhoCorreiaTavares}, which deals with $\theta=2$. We include here some computations for convenience of the reader. 	Fix $\theta>1$ and let $z\in(0,1)$. We write
	\begin{align*}
		L_1(z)&=\frac{1}{\sqrt{2}}\int_z^1\frac{dt}{\sqrt{f(t)-(1-\theta^2)f(z)}} +\frac{1}{\sqrt{2}}\int_1^{f_2^{-1}\left((1-\theta^2)f(z)\right)}\frac{dt}{\sqrt{f(t)-(1-\theta^2)f(z)}}\\
		&=\frac{1}{\sqrt{2}}\int_z^1\frac{dt}{\sqrt{f(t)-(1-\theta^2)f(z)}}+\int_0^1K_1(z,s,\theta)ds.
	\end{align*}
	for
	\[
	K_1(z,s,\theta):=\frac{I(z,\theta)ds}{\sqrt{f(1+I(z,\theta)s)-(1-\theta^2)f(z)}},
	\]where we used the change of variable $t=sf^{-1}_2\left((1-\theta^2)f(z)\right)+(1-s)=:1+sI(z,\theta)$ in the second integral.
	The derivative of the first integral is
	\begin{equation*}
		-\frac{1}{\theta\sqrt{2f(z)}}+\frac{(1-\theta^2)f'(z)}{2\sqrt{2}}\int_z^1\frac{dt}{(f(t)-(1-\theta^2)f(z))^{\frac{3}{2}}}<0.
	\end{equation*}
We now show that the second integral  is decreasing with $z$. 	Using the fact that $\frac{\partial I}{\partial z}(z,\theta)=\frac{(1-\theta^2)f'(z)}{f'(1+I(z,\theta))}$,
	\begin{multline*}
		\frac{\partial K_1}{\partial z}(z,s,\theta)
		=\frac{(1-\theta^2)f'(z)}{2f'(1+I(z,\theta))\left(f(1+I(z,\theta)s)-(1-\theta^2)f(z)\right)^\frac{3}{2}}\left[2\left(f(1+I(z,\theta)s)-(1-\theta^2)f(z)\right)\right.\\
		\qquad\left.-I(z,\theta)\left(f'(1+I(z,\theta)s)s-f'(1+I(z,\theta))\right)\right].
	\end{multline*}
	For $(z,\theta)\in(0,1)\times(1,\infty)$, the term outside brackets is positive. Then, it is enough to show that	
	$$k_1(s)=2\left(f(1+I(z,\theta)s)-(1-\theta^2)f(z)\right)-I(z,\theta)\left(f'(1+I(z,\theta)s)s-f'(1+I(z,\theta))\right)$$
	is negative for fixed $(z,\theta)\in(0,1)\times(1,\infty)$ and $s\in(0,1)$. For all $s\in(0,1)$,
	\[k_1'(s)=f'(1+I(z,\theta)s)I(z,\theta)-I(z,\theta)^2f''(1+I(z,\theta)s)s,\quad k_1''(s)=-I(z,\theta)^3f'''(1+I(z,\theta)s)s>0.\]
	Since $k_1'$ is increasing and $k_1'(0)=0$, we have $k_1'>0$ and thus $k_1$ is also increasing. Together with $k_1(1)=0$, we conclude that $k_1$ is negative in $(0,1)$. 
\end{proof}

In the next results, we analyze the monotonicity of $L_1$ in the interval $\left(1,(p/2)^{1/(p-2)}\right)$ for values $\theta>1$.

\begin{lemma}\label{lemma:2.66}
	There exists $\theta^*>1
	$ such that, for $\theta>\theta^*$, the function $z\mapsto L_1'(z)$ changes sign in the interval $(1,(p/2)^{1/(p-2)})$. In particular, $L_1$ is not monotone in $(1,(p/2)^{1/(p-2)})$.
\end{lemma}
\begin{proof}
	By direct inspection of \eqref{LDerivativeForm}, one observes that $H(1,\theta)=\theta\sqrt{f(1)}>0$. 
	On the other hand, since  $z>1$, we know that $f$ is decreasing in $[z,f_2^{-1}((1-\theta^2)f(z))]$ and thus, for all $t\in[z,f_2^{-1}((1-\theta^2)f(z))]$, 
	\[\frac{1}{\sqrt{f(t)-(1-\theta^2)f(z)}}=\frac{1}{\sqrt{f(t)-f(z)+\theta^2f(z)}}>\frac{1}{\theta\sqrt{f(z)}}.\]
	Given that $(1-\theta^2)f'(z)>0$ and, by Lemma \ref{lemma:2.63}, $A(t)<0$ for all $t\in[z,f_2^{-1}((1-\theta^2)f(z))]$, we get	
	\[H(z,\theta)<\theta\frac{f(1)}{\sqrt{f(z)}}+\frac{(1-\theta^2)}{\theta\sqrt{f(z)}}f'(z)\int_z^{f_2^{-1}((1-\theta^2)f(z))}A(t)dt.\]
	Therefore, for any fixed $z\in\left(1,(p/2)^{1/(p-2)}\right)$, we have, due to asymptotic behavior of $A$ in Lemma \ref{lemma:2.63}, that
	\[\limsup_{\theta\to+\infty}\frac{H(z,\theta)\sqrt{f(z)}}{\theta}\leq f(1)+\limsup_{\theta\to+\infty}\left(\frac{1}{\theta^2}-1\right)f'(z)\int_z^{f_2^{-1}((1-\theta^2)f(z))}A(t)dt=f(1)-f'(z)(-\infty)=-\infty.\]
	This implies that $H(z,\theta)<0$ for large values of $\theta$ and the conclusion follows.
\end{proof}

The following result will help us establish a range of $\theta$ for which $L_1$ will be increasing.

\begin{proposition}[Sufficient condition for $L_1$ to be decreasing]\label{prop:2.67}
	Let $p>2$. Consider $\psi:\left(\R^+\right)^2\times(1,\infty)\to\R$ to be defined by $$\psi(z,x,\theta)=-f'(x)^2+(1-\theta^2)f'(z)\left(f'(x)+(z-x)f''(x)\right).$$ Then, for $\theta\in(1,2]$, $\psi(z,x,\theta)\leq0$ for all $1\leq z\leq x$.
\end{proposition}

This result will be a direct consequence of Lemmas \ref{lemma:2.70} and \ref{lemma:2.72}, which we prove below. Before doing so, we show that, if Proposition \ref{prop:2.67} holds, then the monotonicity of $L_1$ will easily follow.

\begin{lemma}\label{lemma:2.68}
	Suppose that Proposition \ref{prop:2.67} holds true. Then, $z\mapsto L_1(z)$ is strictly decreasing for $z\geq 1$.
\end{lemma}
\begin{proof}
	Fix $z\geq 1$. Take $t=z+sI(z,\theta)$, where $I(z,\theta)=f_2^{-1}\left((1-\theta^2)f(z)\right)-z>0$. Then
	\begin{equation*}\label{Lzmaiorque1}
		L_1(z)=\frac{1}{\sqrt{2}}\int_0^1\frac{I(z,\theta)ds}{\sqrt{f(z+sI(z,\theta))-(1-\theta^2)f(z)}}=:\frac{1}{\sqrt{2}}\int_0^1K_1(z,s,\theta)ds.
	\end{equation*}
	Differentiating $K_1$ with respect to $z$, we obtain
	\begin{align*}
		\frac{\partial K_1}{\partial z}(z,s,\theta)&=\frac{1}{2(f(z+sI(z,\theta))-(1-\theta^2)f(z))^{3/2}}\left[2\frac{\partial I}{\partial z}(z,\theta)(f(z+I(z,\theta)s)-(1-\theta^2)f(z))\right.\\   -&\left.I(z,\theta)\left(f'(z+I(z,\theta)s)\left(1+\frac{\partial I}{\partial z}(z,\theta)s\right)-(1-\theta^2)f'(z)\right)\right]\\
		&=:\frac{1}{2(f(z+sI(z,\theta))-(1-\theta^2)f(z))^{3/2}}g(z,s,\theta).
	\end{align*}
Since $\frac{\partial I}{\partial z}(z,\theta)=-1+\frac{(1-\theta^2)f'(z)}{f'(z+I(z,\theta))}$, by definition of $I(z,\theta)$, $g(z,1,\theta)=0$ for any $z\geq1$ and $\theta>1$, and
	\begin{align*}
		\frac{\partial g}{\partial s}(z,1,\theta)=\frac{I(z,\theta)}{f'(z+I(z,\theta))}\psi(z,z+I(z,\theta),\theta)\geq 0,
	\end{align*}
since $z+I(z,\theta)s>z$ and by Proposition \ref{prop:2.67}. From this, reasoning exactly as in \cite[Lemma 2.8]{AgostinhoCorreiaTavares}, we show that, for each $z\geq1$ and $\theta>1$, the map $s\mapsto g(z,s,\theta)$ is an increasing function, which then yields $g<0$ for $s\in(0,1)$, and the conclusion follows. \end{proof}

\begin{remark}
	It is important to stress that, in the previous proof, the range of $\theta\in(1,2]$ is not relevant. In fact, Lemma \ref{lemma:2.68} will be true whenever, for a fixed $\theta>1$, the conclusion of Proposition \ref{prop:2.67} is valid. We managed to prove this only for $\theta\in(1,2]$.
\end{remark}

Now, we focus on proving Proposition \ref{prop:2.67}. We start with an auxiliary lemma.
\begin{lemma}{}\label{lemma:2.69}
	Let $p>2$ and define $h:\left[1,\infty\right)\times(1,2]\to\R$ by $h(x,\theta)=-2+(p-1)(4-(p-2)(\theta^2-3))x+(p-1)(-2(p-1)+(p-2)(\theta^2-3))x^2.$
	Then, $h<0$ for all $x\in[1,\infty)$ and $\theta\in(1,2]$.
\end{lemma}
\begin{proof}
By \cite[Lemma 2.9]{AgostinhoCorreiaTavares}, we have $h(x,2)<0$ for $x\geq1$.	For fixed $x\geq1$,
	$ \frac{\partial h}{\partial\theta}(x,\theta)=2\theta(p-1)(p-2)x(x-1)\geq0.$
	Thus, for every fixed $x\geq1$, the function $\theta\mapsto h(x,\theta)$ is non-decreasing and the conclusion follows.
	\end{proof}
\begin{lemma}{}\label{lemma:2.70}
	Suppose that $p\in(2,3]$. Then Proposition \ref{prop:2.67} holds true.
\end{lemma}
\begin{proof}
	Fix $z>1$ and $\theta\in(1,2]$. Note that
		$f'(z)=f'(x)+f''(x)(z-x)+\frac{f'''(c)}{2}(z-x)^2$
	for some $c\in (z,x)$, which gives
	\begin{equation*}
		\psi(z,x,\theta)=-f'(x)^2+(1-\theta^2)f'(z)^2-\frac{(1-\theta^2)}{2}f'(z)f'''(c)(z-x)^2.
	\end{equation*}
For $p\in(2,3]$,  $f'''$ is non-decreasing; moreover, since $\theta>1$,
	$-\frac{(1-\theta^2)}{2}f'(z)(z-x)^2<0$. Thus,
	\begin{equation*}	\psi(z,x,\theta)\leq -f'(x)^2+(1-\theta^2)f'(z)^2-\frac{(1-\theta^2)}{2}f'(z)f'''(z)(z-x)^2=:w(x,\theta).
	\end{equation*}
To finish the proof, it is enough to show that $w(x,\theta)\leq0$. For fixed $\theta>1$, and some $\xi\in(z,x)$,
	$$w(x,\theta)=w(z,\theta)+\frac{\partial w}{\partial x}(z,\theta)(x-z)+\frac{1}{2}\frac{\partial^2 w}{\partial x^2}(z,\theta)(x-z)^2+\frac{1}{3!}\frac{\partial^3 w}{\partial x^3}(\xi,\theta)(x-z)^3,$$ where 
	$w(z,\theta)=-\theta^2f'(z)^2<0$, and $\frac{\partial w}{\partial x}(z,\theta)=-2f'(z)f''(z)<0$. Also, by Lemma \ref{lemma:2.69}, a direct computation using the explicit expression of $f'$, $f''$ and $f'''$ yields
	\begin{align*}
		\frac{\partial^2 w}{\partial x^2}(z,\theta)=-2(1-(p-1)z^{p-2})^2-(\theta^2-3)(p-1)(p-2)z^{p-2}(1-z^{p-2})=h(z^{p-2},\theta)<0.
	\end{align*}
	Additionally, since $\xi>1$,
	\begin{equation*}
		\frac{\partial^3 w}{\partial x^3}(\xi,\theta)=-2\left[3f''(\xi)f'''(\xi)+f'(\xi)f^{(4)}(\xi)\right]=-2\left[(p-1)(p-2)\xi^{p-3}\left(-p+2(2p-3)\xi^{p-2}\right)\right]\leq0.\qedhere
	\end{equation*}
\end{proof} 

In the following results, we prove that Proposition \ref{prop:2.67} also holds true for values $p>3$. 

\begin{lemma}\label{lemma:2.71}
	Let $n\in\mathbb{N}$ be such that $n\geq3$ and $p\in(n,n+1]$. Then, for each $k\in\left\{0,\cdots,n-1\right\}$,
	$$\frac{\partial^k\psi}{\partial x^k}(z,z,\theta)<0,\ \text{for}\ z\geq1\ \text{and}\ \theta\in(1,2). $$
\end{lemma}
\begin{proof}
	We have $\psi(z,z,\theta)=-\theta^2f'(z)^2<0$, $\frac{\partial\psi}{\partial x}(z,z,\theta)=-2f'(z)f''(z)<0$, and
	$$\frac{\partial ^2\psi}{\partial x^2}(z,z,\theta)=-2f''(z)^2+(\theta^2-3)f'(z)f'''(z)=h(z^{p-2},\theta)<0.$$
	Fix  $z>1$ and  $ k\geq 3$. Then
	\begin{align*}
		\frac{\partial^k\psi}{\partial x^k}(z,x,\theta)&=-\sum_{j=0}^{k}\left[{\binom{k}{j}}f^{(k+1-j)}(x)f^{(j+1)}(x)\right]\\
		+&(1-\theta^2)f'(z)\left(f^{(k+1)}(x)+\sum_{j=0}^k{\binom{k}{j}}f^{(k+2-j)}(x)\frac{d^j}{dx^j}(z-x)\right)\\
		&=-\sum_{j=0}^{k}\left[{\binom{k}{j}}f^{(k+1-j)}(x)f^{(j+1)}(x)\right]+(1-\theta^2)f'(z)\left(f^{(k+2)}(x)(z-x)-(k-1)f^{(k+1)}(x)\right).
	\end{align*}
	For  fixed $z\geq1$, consider
	$$ \frac{\partial ^k\psi}{\partial x^k}(z,z,\theta)=-\sum_{j=0}^{k}\left[{\binom{k}{j}}f^{(k+1-j)}(z)f^{(j+1)}(z)\right]-(1-\theta^2)(k-1)f'(z)f^{(k+1)}(z).$$ 
	For $z\geq1$ and $3\leq k<n<p$, differentiation with respect to $\theta$ yields
	\[\frac{\partial}{\partial\theta}\frac{\partial^k}{\partial x^k}\psi(z,z,\theta)=2\theta f'(z)(k-1)f^{(k+1)}(z)\geq0.\]
	Thus, the map $\theta\mapsto\frac{\partial^k}{\partial x^k}\psi(z,z,\theta)$ is non-decreasing. From \cite[Lemma 2.11]{AgostinhoCorreiaTavares} we know that, given $z\geq1$, $\frac{\partial^k}{\partial x^k}\psi(z,z,2)<0$, which concludes the proof. 
	\end{proof}
\begin{lemma}\label{lemma:2.72}
	Let $n\in\mathbb{N}$ be such that $n\geq3$ and $p\in(n,n+1]$. Then Proposition \ref{prop:2.67} holds.
\end{lemma} 
\begin{proof}
	For fixed $z\geq 1$ and $\theta\in(1,2]$, we have
	$$\psi(z,x,\theta)=\sum_{k=0}^{n-1}\frac{\partial^k\psi}{\partial x^k}(z,z,\theta)\frac{(x-z)^k}{k!}+\frac{\partial^n\psi}{\partial x^n}(z,c,\theta)\frac{(x-z)^n}{n!}, $$
	for some $c\in(z,x)$. The first term on the right-hand side is negative by Lemma \ref{lemma:2.71}. It is enough to show that $\frac{\partial ^n\psi}{\partial x^n}(z,x,\theta)<0$  for $x\geq z\geq1$ and $\theta\in(1,2]$. We have
	\begin{align*}
		\frac{\partial ^n\psi}{\partial x^n}(z,x,\theta)&=-\sum_{j=0}^{n}\left[{ \binom{n}{j}}f^{(n+1-j)}(x)f^{(j+1)}(x)\right]\\
		+&(1-\theta^2)f'(z)\left(f^{(n+1)}(x)+\sum_{j=0}^n{\binom{n}{j}}\frac{d^j(z-x)}{dx^j}f^{(n+2-j)}(x)\right)\nonumber \\
		&=-\sum_{j=2}^{n-2}\left[{ \binom{n}{j}}f^{(n+1-j)}(x)f^{(j+1)}(x)\right]+f^{(n+1)}(x)\left[\frac{-2n}{p-n}f''(x)x-2f'(x)\right.\\
		+&\left.(1-\theta^2)f'(z)(z-x)\frac{p-(n+1)}{x}-(1-\theta^2)(n-1)f'(z)\right].
	\end{align*}
	Note that 
	$f^{(n+1)}(x)<0$, $\frac{(1-\theta^2)f'(z)(z-x)(p-(n+1))}{x}>0$ for $z>1$, and $f'$ is decreasing. Furthermore, $f^{(j)}(z)<0$ for $3\leq j\leq n+1$. Thus,
	\begin{align*}
		\frac{\partial ^n\psi}{\partial x^n}(z,x,\theta)&\leq f^{(n+1)}(x)\left[\frac{-2n}{(p-n)}f''(x)x-2f'(x)-(1-\theta^2)(n-1)f'(x)\right]\\
		&=\frac{f^{(n+1)}(x)x}{p-n}\left[-2nf''(x)+(-(n+1)+\theta^2(n-1))(p-n)(1-x^{p-2})\right].
	\end{align*}
	Define $g:(1,\infty)\times(1,2]\to\R$ by $$g(x,\theta)=-2nf''(x)+(-(n+1)+\theta^2(n-1))(p-n)(1-x^{p-2}).$$ We claim that $g(x,\theta)\geq 0$. To that end, observe that for each fixed $x\in(1,\infty)$, the map $\theta\mapsto g(x,\theta)$ is decreasing. Indeed,
	\[\frac{\partial g}{\partial \theta}(x,\theta)=2\theta(n-1)(p-n)(1-x^{p-2})<0,\]
	since $n\geq 3$, $x,\theta>1$, and $p\in (n,n+1]$. To conclude the proof, it suffices to show that $g(x,2)\geq0$. Note that
	$ g(x,2)=-2n(1-(p-1)x^{p-2})+(3n-5)(p-n)(1-x^{p-2})\geq 0$
	which is a consequence of 
	\begin{equation*}
		\frac{2n}{(p-n)(3n-5)}\geq\frac{1}{p-1}\geq\frac{1-x^{p-2}}{1-(p-1)x^{p-2}}.\qedhere 
	\end{equation*}
\end{proof} 

As mentioned before,
\begin{proof}[Proof of Proposition \ref{prop:2.67}]
	The conclusion follows as a direct application of Lemmas \ref{lemma:2.70} and \ref{lemma:2.72}.
\end{proof}

We are now in position to conclude the analysis of increasing solutions.
\begin{proof}[Proof of Proposition \ref{prop:2.61}]
	We recall that the case $\theta=0$ was shown in  \cite[Proposition 2.16]{AgostinhoCorreiaTavares}. Take $\theta>0$. The function $L_1$ is positive by definition and differentiable by Lemma \ref{lemma:2.62}. Item 1 follows from Lemma \ref{lemma:2.64}. Item 2 is a consequence of Lemma \ref{lemma:2.65}. Item 3 follows from Lemmas \ref{lemma:2.66} and \ref{lemma:2.68}.
	
	As $z\to0^+$, Fatou's Lemma and Definition \ref{def:2.60} yield
	\[
	\liminf_{z\to 0^+}L_1(z)\geq\int_0^{\left(\frac{p}{2}\right)^\frac{1}{p-2}}\frac{dt}{\sqrt{f(t)}}=\infty.
	\]
	As $z\to (p/2)^{1/(p-2)}$, since $L_1$ is continuous and  the integral term in \eqref{eqn.LDiffForm1} has no singularities, we can take limits in \eqref{eqn.LDiffForm1} to find
	$$-f(1)\sqrt{2}\lim_{z\to \left(\frac{p}{2}\right)^\frac{1}{p-2}}L_1(z)=0.\qedhere $$
\end{proof}

\subsection{Classification of decreasing solutions of the overdetermined problem \eqref{gen.IVP2}}

To construct \emph{decreasing} solutions of \eqref{gen.IVP2}, let $z\in\left(0,(p/2)^{1/(p-2)}\right)$, we consider once again a solution $u$ of the IVP \eqref{eqn.probu_IVP}, but this time with $\theta\leq 0.$ Recalling from Lemmas \ref{lemma:Theta=1} and \ref{lemma:Theta<1} that for $\theta\leq -1$ there are no positive decreasing solutions of the overdetermined problem \eqref{gen.IVP2}, for the rest of this subsection we restrict our attention on the case $\theta\in (-1,0]$.

We proceed in a similar fashion as in the increasing case. If $u$ is decreasing, let $L>0$ be the first positive zero of $u'$. Since $(1-\theta^2)f(z)=f(u(L))$ and $0<u(L)<u(0)$,  this time we have $u(L)=f_1^{-1}\left((1-\theta^2)f(z)\right)$.

Since $u'<0$ in $[0,L)$, using exactly the same arguments which led to \eqref{eq:L_increasing_aux} in the increasing case, we obtain
\[L=\frac{1}{\sqrt{2}}\int^z_{f_1^{-1}\left((1-\theta^2)f(z)\right)}\frac{dt}{\sqrt{f(t)-(1-\theta^2)f(z)}}.\]

Moreover, the previous argument is valid for each $z\in\left(0,(p/2)^{1/(p-2)}\right)$ if $\theta>0$. If $\theta=0$ the previous construction works for $z\in\left(1,(p/2)^{1/(p-2)}\right)$, but fails for $z\in\left(0,1\right]$ (recall Lemma  \ref{PHASEPLANETHETA=0}).
\begin{definition}[Length function of decreasing solutions of the IVP \eqref{gen.IVP2}]
	For $p>2$, $\theta\in(-1,0]$, we consider the \textit{length function} of \eqref{eqn.probu_IVP}  as the function $L_2:\left(0,(p/2)^{1/(p-2)}\right)\to\R^+$ defined by
	$$L_2(z)=\frac{1}{\sqrt{2}}\int^z_{f_1^{-1}\left((1-\theta^2)f(z)\right)}\frac{dt}{\sqrt{f(t)-(1-\theta^2)f(z)}}.$$
	If $\theta=0$, the function $L_2$, defined by the same expression, is considered for $z\in\left(1,\left(\frac{p}{2}\right)^\frac{1}{p-2}\right)$.
\end{definition}

Again, we focus on studying the qualitative properties of $L_2$.

\begin{proposition}\label{prop:2.74}
	For $p>2$ and $\theta\in(-1,0]$, the function $L_2$ is positive, differentiable, strictly increasing and satisfies 
 \begin{equation}\label{eq:limitesL2}
         \lim\limits_{z\to (p/2)^\frac{1}{p-2}}L_2(z)=+\infty,\quad \begin{cases}
        \lim\limits_{z\to 0^+}L_2(z)=\argsinh\left(\frac{\theta}{\sqrt{1-\theta^2}}\right), & \theta\in(-1,0)\\
        L_2(1):=\lim\limits_{z\to1^+}L_2(z)=\frac{\pi}{\sqrt{p-2}}, & \theta=0
    \end{cases}.
 \end{equation}
\end{proposition}

We devote the rest of this section to the proof of this result. Since the case $\theta=0$ has been treated in \cite[Proposition 2.16]{AgostinhoCorreiaTavares}, we fix $\theta\in(-1,0)$.

The first result is the equivalent to Lemma \ref{lemma:2.62} applied to $L_2$.

\begin{lemma}{}\label{lemma:2.75}
	For any $\theta\in(-1,0)$ and $z\in \left(0,(p/2)^{1/(p-2)}\right)$,
	\begin{align}
		\sqrt{2}L_2(z)\left((1-\theta^2)f(z)\right.&\left.-f(1)\right)=-2\theta\frac{(f(z)-f(1))\sqrt{f(z)}}{f'(z)}\label{eqn.LDiffForm}\\
		&-\int^z_{f_1^{-1}\left((1-\theta^2)f(z)\right)}\left(3-2\frac{(f(t)-f(1))}{f'(t)^2}f''(t)\right)\sqrt{f(t)-(1-\theta^2)f(z)}dt.\nonumber
	\end{align}
	In particular, $L_2$ is differentiable. Moreover, defining $H(z,\theta):=\sqrt{2}L_2'(z)\left((1-\theta^2)f(z)-f(1)\right)$,
	\begin{align}
		&H(z,\theta)=\theta\frac{f(1)}{\sqrt{f(z)}}+(1-\theta^2)f'(z)\int^z_{f_1^{-1}\left((1-\theta^2)f(z)\right)}\frac{A(t)dt}{\sqrt{f(t)-(1-\theta^2)f(z)}}\label{L2DerivativeForm}\\
		=&\frac{\theta}{\sqrt{f(z)}}\left(f(1)-2(1-\theta^2)A(z)f(z)\right)-(1-\theta^2)f'(z)\int^z_{f_1^{-1}\left((1-\theta^2)f(z)\right)}\frac{g(t)}{f'(t)^4}\sqrt{f(t)-(1-\theta^2)f(z)}dt,\label{H_p(z,theta)NotSingular2}
	\end{align}
	where $A:\R^+\to \R$ and  $g:\R^+\setminus\{1\}\to\R$ are defined in Lemma \ref{lemma:2.62}, and $g(t)<0$.\end{lemma}
\begin{proof}
	The proof is exactly the same as in Lemma \ref{lemma:2.62}. The only difference is the parametrization of the curve $\gamma(t)=(x(t),y(t))$, which parametrizes the level curve $\Gamma$ given by $-\frac{1}{2}y^2+f(x)=(1-\theta^2)f(z)$ with $y<0$. The conclusion will follow by taking, for $t\in[f_1^{-1}\left((1-\theta^2)f(z)\right),z]$, the curve $$\gamma(t)=(x(t),y(t))=\left(f_1^{-1}((1-\theta^2)f(z))+z-t,-\sqrt{2}\sqrt{f\left(f_1^{-1}((1-\theta^2)f(z))+z-t\right)-(1-\theta^2)f(z)}\right).$$
	When compared with Lemma \ref{lemma:2.62}, the change in the signs of the non-integral terms is justified by the fact that the integration along the curve $\Gamma$ will now yield $W(z,\theta\sqrt{2}\sqrt{f(z)})- W(f_1^{-1}\left((1-\theta^2)f(z)\right),0)$ instead of $ W(f_1^{-1}\left((1-\theta^2)f(z)\right),0)-W(z,\theta\sqrt{2}\sqrt{f(z)}),$ as was the case in Lemma \ref{lemma:2.62}. We recall that $g$ is negative by Proposition \ref{prop:gNegativa}.
\end{proof}

\subsubsection{Monotonicity Analysis of $L_2$.}

\begin{lemma}{}\label{lemma:2.76}
	If $\theta\in(-1,0)$ and $z\geq 1$, then $L_2$ is strictly increasing.
\end{lemma}
\begin{proof}
	By Lemma \ref{lemma:2.75}, it is enough to show that, for all $(z,\theta)\in (0,1]\times(-1,0]$, the expression for $H(z,\theta)$ in \eqref{H_p(z,theta)NotSingular2} satisfies $H(z,\theta)\leq0$.
	
	Fix $\theta\in(-1,0]$ and recall that for $z\geq 1$, we have $f'(z)\leq0$, with strict inequality if $z>1$. Since $g$, given in Lemma \ref{lemma:2.75}, is negative, we have $H(z,\theta)\leq 0$. From Lemma \ref{lemma:2.63}, 
	$f(1)-2(1-\theta^2)A(z)f(z)>0$ and thus, both terms in \eqref{H_p(z,theta)NotSingular2} are non-positive, which finishes the proof.
\end{proof}

\begin{lemma}\label{lemma:2.77}
	If $\theta\in(-1,0)$ and $z\in(0,1)$, then we have $L_2(z)=\frac{\theta}{\sqrt{2}}\int_0^1K(z,\theta,s)ds,$
	where
\[
   K(z,\theta,s):=\frac{\sqrt{f(z)}}{\sqrt{s}f'\left(f_1^{-1}\left(\theta^2f(z)s+(1-\theta^2)f(z)\right)\right)}.
    \]
	The map $z\mapsto K(z,\theta,s)$ is increasing and, in particular, $L_2$ is increasing in $(0,1)$. 
\end{lemma}
\begin{proof}
	Since $z<1$, we write $z=f_1^{-1}\left(f(z)\right)$ and consider the change of variable $t=f_1^{-1}(u)$. Then $u(s)=\theta^2f(z)s+(1-\theta^2)f(z)$ and
	\begin{align*}
    L_2(z)&=\frac{1}{\sqrt{2}}\int_{f_1^{-1}\left((1-\theta^2)f(z)\right)}^{f_1^{-1}(f(z))}\frac{dt}{\sqrt{f(t)-(1-\theta^2)f(z)}}=\frac{1}{\sqrt{2}}\int_{(1-\theta^2)f(z)}^{f(z)}\frac{1}{\sqrt{u-(1-\theta^2)f(z)}}\frac{du}{f'\left(f_1^{-1}(f(u))\right)}\\
        &=\frac{1}{\sqrt{2}}\int_0^1\frac{\theta^2f(z)}{\sqrt{s\theta^2f(z)}}\frac{ds}{f'\left(f_1^{-1}(u(s))\right)}=\frac{\theta}{\sqrt{2}}\int_0^1K(z,\theta,s)ds,
    \end{align*}
	where $K$ is defined as in the statement. We show that, for fixed $(\theta,s)\in(-1,0)\times(0,1)$, the map $z\mapsto K(z,\theta,s)$ is increasing.
	
	To simplify, let $X(z,\theta,s):=f_1^{-1}\left(\theta^2f(z)s+(1-\theta^2)f(z)\right)$, which satisfies $\frac{\partial}{\partial z}X(z,\theta,s)=\frac{f'(z)(1-\theta^2(1-s))}{f'\left(X(z,\theta,s)\right)}$. Note that, by construction, $X(z,\theta,s)\leq 1$. Differentiating $K$ with respect to $z$ yields
	\begin{align*}
		\frac{\partial}{\partial z}K(z,\theta,s)&=\frac{1}{\sqrt{s}f'\left(X(z,\theta,s)\right)^2}\left[\frac{f'(z)}{2\sqrt{f(z)}}f'(X(z,\theta,s))-\sqrt{f(z)}f''\left(X(z,\theta,s)\right)\frac{\partial}{\partial z}X(z,\theta,s)\right]\\
		&=\frac{f'(z)}{2\sqrt{sf(z)}f'\left(X(z,\theta,s)\right)^3}\left[f'\left(X(z,\theta,s)\right)^2-2f(z)f''\left(X(z,\theta,s)\right)(1-\theta^2(1-s))\right]\\
		&=:\frac{f'(z)}{2\sqrt{sf(z)}f'\left(X(z,\theta,s)\right)^3}g(z,\theta,s).
	\end{align*}
	Since $z\in(0,1)$ and $X(z,\theta,s)\in (0,1)$, it is then enough to show that $g(z,\theta,s)\geq 0$.
	Using the fact that $f'''\left(X(z,\theta,s)\right)\leq0$ and  $f'\left(X(z,\theta,s)\right)\geq0$, we obtain
	\begin{align*}
		\frac{\partial}{\partial z}g(z,\theta,s)&=2f'(X(z,\theta,s))f''(X(z,\theta,s))\frac{\partial}{\partial z}X(z,\theta,s)\\
		-&2(1-\theta^2(1-s))\left(f'(z)f''\left(X(z,\theta,s)\right)+f(z)f'''\left(X(z,\theta,s)\right)\frac{\partial}{\partial z}X(z,\theta,s)\right)\\
		&=-2f(z)f'(z)f'''\left(X(z,\theta,s)\right)\frac{(1-\theta^2(1-s))^2}{f'\left(X(z,\theta,s)\right)}>0\ \text{for all}\ z,\theta,s\in(0,1).
	\end{align*}
	Whence, the map $z\mapsto g(z,\theta,s)$ is (strictly) increasing. Moreover, since $X(0,\theta,s)=f_1^{-1}(0)=0$ for all $(\theta,s)\in(-1,0)\times (0,1)$, we have that $g(0,\theta,s)=0$ for all $(\theta,s)\in(-1,0)\times (0,1)$. Thus, $g(z,\theta,s)>0$ for all $(z,\theta,s)$, which finishes the proof.
\end{proof}

\begin{proof}[Proof of Proposition \ref{prop:2.74}]
	The facts that $L_2$ is positive and differentiable follow from its construction and Lemma \ref{lemma:2.75}, respectively. The monotonicity is a consequence of Lemmas \ref{lemma:2.76} and \ref{lemma:2.77}. It remains to prove \eqref{eq:limitesL2}. We begin with $z\to0^+$.
	
	Recall the characterization of $L_2$ in Lemma \ref{lemma:2.77}.
	Observe that, for each $(\theta,s)\in (-1,0)\times(0,1)$ the map $z\mapsto K(z,\theta,s)$ is strictly increasing and,
    for  $X(z,\theta,s):=f_1^{-1}\left(\theta^2f(z)s+(1-\theta^2)f(z)\right)$,
	\begin{align*}
		\lim_{z\to 0^+}&(K(z,\theta,s))^2=\lim_{z\to 0^+}\frac{1}{s}\frac{f(z)}{f'(X(z,\theta,s))^2}=\frac{1}{s}\lim_{z\to 0^+}\frac{f'(z)}{2f'\left(X(z,\theta,s)\right)f''\left(X(z,\theta,s)\right)\frac{\partial X}{\partial z}(z,\theta,s)}\\
		=&\frac{1}{s}\lim_{z\to 0^+}\frac{1}{2f''\left(X(z,\theta,s)\right)\left(1-\theta^2(1-s)\right)}=\frac{1}{s}\frac{1}{2f''(0)\left(1-\theta^2(1-s)\right)}=\frac{1}{s}\frac{1}{2\left(1-\theta^2(1-s)\right)}.
	\end{align*}
    It noe follows by the Monotone Convergence Theorem that
	\begin{align*}
    \lim_{z\to 0^+}L_2(z)&=\frac{\theta}{\sqrt{2}}\int_0^1\lim_{z\to 0^+}K(z,\theta,s)ds=\int_0^1\frac{\theta ds}{2\sqrt{s\left(1-\theta^2(1-s)\right)}}\\&=\left.\argsinh\left(\frac{\theta\sqrt{s}}{\sqrt{1-\theta^2}}\right)\right|^{s=1}_{s=0}=\argsinh\left(\frac{\theta}{\sqrt{1-\theta^2}}\right).
    \end{align*}

    The limit as $z\to (p/2)^{1/(p-2)}$ is proved exactly as in Proposition \ref{prop:2.61}. 
\end{proof}

We finish this section with the proof of Theorem \ref{th:SolIVPGeralTheta}.

\begin{proof}[Proof of Theorem \ref{th:SolIVPGeralTheta}]
	We start with item 1. Fix\footnote{Even though the result for $\theta=0$ is included in the proof of \cite[Theorem 1.2, part 3]{AgostinhoCorreiaTavares}, this fact is slighlty hidden and, for that reason, we present a proof here.} $\theta=0$.  By Propositions \ref{prop:2.61} and \ref{prop:2.74}, define $\ell^*_1:=L_1(1)=L_2(1)=\frac{\pi}{\sqrt{p-2}}>0$.
	The sub-item (a) follows from Lemma \ref{PHASEPLANETHETA=0}-3. Sub-items (b) and (c) are both a consequence of asymptotic behavior of $L_1$ and $L_2$, items 1 and 3 of Propositions \ref{prop:2.61} and \ref{prop:2.74} and items 1 and 2 of the Lemma \ref{PHASEPLANETHETA=0}.

	Item 2 ($\theta\in(0,1)$) is a consequence of Proposition \ref{prop:2.61}-1. For item 3 ($\theta\in(-1,0)$), observe that $\ell^*_2=L_2(0)$ and use  Proposition \ref{prop:2.74}-2.
	Items 4 ($\theta=1$) and 5 ($\theta\le -1$) are a direct consequence of Lemmas \ref{lemma:Theta=1} and \ref{lemma:Theta<1}.
	Finally, item 6 ($\theta>1$) is a consequence of Lemma \ref{lemma:Theta<1}, Proposition \ref{prop:2.61} and the fact that the range of $L_1$ is $\R^+$.
\end{proof}

\section{Action ground-states on single-knot graphs}

\subsection{Existence of action ground-states}\label{sec:exist}

In this section, we prove our main results regarding the existence of action ground-states, namely Theorem \ref{thm:ags} and \ref{thm:agsregular}. We start with a general result, needed for both proofs.

\begin{lemma}\label{lemma:minimizing_sequences}
Let $\mathcal{G}$ be a single-knot metric graph that admits bound-states. Let $u_n$ be a  minimizing sequence of $\mathcal{S}_\mathcal{G}$, that is,
\[
u_n \text{ is a bound-state and } S(u_n)\to \mathcal{S}_\mathcal{G}.
\]
Then:
\begin{enumerate}
    \item There exists $u\in H^1(\mathcal{G})$ such that $u_n\to u$ weakly in $H^1(\mathcal{G})$ and strongly in $L^q_{\rm loc}(\mathcal{G})$ for every $q\in [2,\infty)$.
    \item Moreover, if $u\not \equiv 0$, then $u$ is an action ground-state.
\end{enumerate}
\end{lemma}
\begin{proof}
1. The bound-state equation implies that $\|u_n\|_{H^1(\mathcal{G)}}^2=\|u_n\|_{L^p(\mathcal{G})}^p$, and thus
$$
\left(\frac{1}{2} - \frac{1}{p}\right)\|u_n\|_{H^1(\mathcal{G})}^2=\frac{1}{2}\|u_n\|_{H^1(\mathcal{G})}^2-\frac{1}{p}\|u_n\|_{L^p(\mathcal{G})}^p =S(u_n)\to \mathcal{S}_\mathcal{G}.
$$
In particular, $u_n$ is bounded in $H^1(\mathcal{G})$ and, up to a subsequence, $u_n\rightharpoonup u$ and, on compact sets, we have strong-$L^q$ convergence. From this, the equation $-u_n''+u_n=|u_n|^{p-2}u_n$ passes to the limit in a weak sense, and $u$ is either a bound-state or $u=0$.    

2. Let us assume that $u\not \equiv 0$.  Recall that, on each half-line $h$, $u_n$ must either be zero or coincide with a portion of a soliton: there exists $y_n>0$ such that
$$
(u_n)_h=\pm\varphi(\cdot \pm y_n), \quad y_n\ge 0.
$$
If $u(\mathbf{0})\neq 0$, then the sequence $(y_n)$ is uniformly bounded and thus $u_n\to u$ in $L^p(\mathcal{G})$. This implies that
$$
S(u)=\left(\frac{1}{2}-\frac{1}{p}\right)\|u\|_{L^p(\mathcal{G})}^p =\lim \left(\frac{1}{2}-\frac{1}{p}\right)\|u_n\|_{L^p(\mathcal{G})}^p =\mathcal{S}_\mathcal{G}
$$
and $u$ is an action ground-state.

Suppose now that $u(\mathbf{0})=0$, so that $y_n\to \infty$. Without loss of generality, we may suppose that $u_n(\mathbf{0})>0$. Up to a subsequence, there exist $H^+$ (constant) half-lines where $(u_n)_h=\varphi(\cdot + y_n)$, and $H^-=H-H^+$ half-lines where $(u_n)_h=\varphi(\cdot - y_n)$. Since the action of a soliton is positive,
$$
\mathcal{S}_\mathcal{G}=\lim S(u_n) = S(u) + H^- S(\varphi) \ge S(u)\geq \mathcal{S}_\mathcal{G}. \qedhere
$$

\end{proof}

\begin{proof}[Proof of Theorem \ref{thm:ags}]
 Let $\mathcal{G}$ be a single-knot graph that admits bound-states, and take $u_n$ to be a minimizing sequence associated with the level $\mathcal{S}_\mathcal{G}$. Up to a subsequence, either $u_n(\mathbf{0})=0$ or $u_n(\mathbf{0})\neq 0$. 
 
 In the first case, $u_{n,h}=0$ for all $h\in \Hcal$, and $u_n\to u$ strongly in $L^p(\mathcal{G})$. From the identity
 \[
 \|u_n\|_{H^1(\mathcal{G})}^2=\|u_n\|^p_{L^p(\mathcal{G})},
 \]
 we deduce that $u\not \equiv 0$ and the proof is finished.

 \smallbreak
 
 From now on, we consider the second case, $u_n(\mathbf{0})\neq 0$. Without loss of generality, we may suppose that $u_n(\mathbf{0})>0$ and, recalling \eqref{eq:portion}, let $y_n>0$ be such that
$$
u_{n,h}=\varphi(\cdot \pm y_n)\quad \text{ for every unbounded edge $h$}.
$$
By Lemma \ref{lemma:minimizing_sequences}  we have, up to a subsequence, $u_n\rightharpoonup u\in H^1(\mathcal{G})$ and strongly in $L^q_{\rm loc}(\mathcal{G})$. 
In order to conclude, we only need to prove that $u\not \equiv 0$. 

Suppose, by contradiction, that $u\equiv 0$. In particular, $y_n\to \infty$. Let $H^-$ be the number of half-lines on which $u_{n,h}=\varphi(x-y_n)$ (this number is constant, up to a subsequence), and $H^+:=H-H^-$. Take $\tilde{\mathcal{K}}$ to be the subgraph whose edges are the ones of the compact core of $\mathcal{G}$, together with $h\cap \overline{B_1(0)}$ for every half-line $h$ (in other words, the unit segments in each half-line starting at $\mathbf{0}$).  Observing that $\|u_n\|_{L^\infty(\tilde{ \mathcal{K}})}\to 0$, define
$$
v_n(x)=\frac{u_n}{\|u_n\|_{L^\infty(\tilde{\mathcal{K})}}},\quad x\in \tilde{\mathcal{K}},
$$
which solves 
\[
    -(v_n)_e''+(v_n)_e = \|u_n\|_{L^\infty(\tilde{ \mathcal{K})}}^{p-2}|(v_n)_e|^{p-2}(v_n)_e,\qquad \sum_{e} (v_n)_e'(0)=0,
\]
where $e$ runs through the set of edges of $\tilde{\mathcal{K}}$. The integration of the equation on all edges yields, by the Neumann-Kirchoff condition at $\mathbf{0}$,
\begin{equation}\label{eq:boundh1}
    \|v_n\|_{H^1(\tilde{\mathcal{K}})}^2 =  \|u_n\|_{L^\infty(\tilde{\mathcal{K}})}^{p-2}\|v_n\|_{L^p(\tilde{\mathcal{K}})}^p + \sum_{h\in \Hcal} \frac{u_{n,h}'(1)u_{n,h}(1)}{\|u_n\|_{L^\infty(\tilde{\mathcal{K})}}^2} 
\end{equation}
For each half-line $h\in \Hcal$,
\begin{equation}\label{eq:unhalfline}
(u_{n})_h'(x)=\pm \sqrt{2f(u_{n,h}(x))}
\end{equation}
so that    $|(u_{n})_h'(x)|\leq C |(u_n)_h(x)|\leq C\|u_n\|_{L^\infty(\tilde{\mathcal{K}})}$.
 Then, by \eqref{eq:boundh1}, $v_n$ is bounded in $H^1(\tilde{\mathcal{K}})$ and, up to a subsequence, $v_n\rightharpoonup v$ in $H^1(\tilde{\mathcal{K}})$ and $v_n\to v$ in $L^q(\tilde{\mathcal{K}})$, $q\in [2,\infty]$. In particular, for every edge $e$ of $\tilde{\mathcal{K}}$,
 \[
    -v_e''+v_e = 0,\qquad \sum_{e} v_e'(0)=0,\qquad \|v\|_{L^\infty(\tilde{\mathcal{K}})}=1.
\]
By \eqref{eq:unhalfline}, on each half-line $h\in \Hcal$, we have
$$
v_h'(0)=\pm v(\mathbf{0}),
$$
with $H^-$ half-lines with positive sign and $H^+$ half-lines with the negative sign.

On each segment $\kappa\in \Kcal$, the solution is given by $v_\kappa(x)=c_1e^x+c_2e^{-x}$. If $\kappa$ is a pendant, using the condition $v'_\kappa(\ell_\kappa)=0$, we have $c_1=\frac{v(\mathbf{0})}{1+e^{\ell_\kappa}}$, $c_2=\frac{v(\mathbf{0})e^{2\ell_\kappa}}{1+e^{\ell_\kappa}}$, and so
$$
v_\kappa(x)=\frac{v(\mathbf{0})\cosh(\ell_\kappa-x)}{\cosh(\ell_\kappa)},\quad x\in[0,\ell_\kappa].
$$
If $\kappa$ is a loop, the condition $v_\kappa(0)=v_\kappa(\ell_\kappa)$ yield $v'(\ell_\kappa/2)=0$ and
\[
v_\kappa(x)=\frac{v(\mathbf{0})\cosh(\ell_\kappa/2-x)}{\cosh(\ell_\kappa/2)},\quad x\in[0,\ell_\kappa].
\]
The Neumann-Kirchoff condition at $\mathbf{0}$ reads
$$
-\sum_{\substack{\kappa\in \Kcal \\ \kappa\ \text{pendant}}} v(\mathbf{0})\tanh(\ell_\kappa)-2\sum_{\substack{\kappa\in \Kcal \\ \kappa\ \text{loop}}} v(\mathbf{0})\tanh(\ell_\kappa/2) + \sum_{h\in \Hcal} v_h'(0) = 0.
$$
Since $v(\mathbf{0})\neq 0$ (since $\|v\|_{L^\infty(\tilde{\mathcal{K}})}=1)$), then
$$
\sum_{\substack{\kappa\in \Kcal \\ \kappa\ \text{pendant}}} \tanh(\ell_\kappa)+2\sum_{\substack{\kappa\in \Kcal \\ \kappa\ \text{loop}}} \tanh(\ell_\kappa/2) = H^--H^+ = H-2H^+.
$$
This is impossible, by assumption \eqref{eq:tanh}.
\end{proof}

Before going through the proof of Theorem \ref{thm:agsregular}, we show the following.

\begin{lemma}\label{lem:funcaol_zfixo}
    Fix $0<z<\left(\frac{1}{p-1}\right)^\frac{1}{p-2}<1$ and consider the mapping $\ell:(0,z)\to \R^+$ given by
    \begin{equation}
        \ell(w)=\frac{1}{\sqrt{2}}\int_{w}^z \frac{dt}{\sqrt{f(t)-f(w)}}.
    \end{equation}
    Then $\ell$ is strictly decreasing.
\end{lemma}
\begin{proof}
    Setting $t=w+s(z-w)$,
    $$
    \ell(w)=\frac{1}{\sqrt{2}}\int_0^1 \frac{z-w}{\sqrt{f(w+s(z-w))-f(w)}}ds.
    $$
    and, differentiating in $w$,
    \begin{align*}
        \ell'(w)&=-\frac{1}{\sqrt{2}}\int_0^1 \frac{1}{(f(t)-f(w))^{3/2}}\left[ f(t)-f(w) + \frac{z-w}{2}((1-s)f'(t)-f'(w)) \right]ds\\
        &=-\frac{1}{\sqrt{2}}\int_0^1 \frac{1}{(f(w+s(z-w))-f(w))^{3/2}}g(w+s(z-w))ds,
    \end{align*}
    where
    $$
    g(t):=f(t)-f(w) + \frac{z-w}{2}\left(\frac{z-t}{z-w}f'(t)-f'(w)\right),\quad \text{ for } t\in[w,z].
    $$
    Since $g(w)=0$ and
    \begin{align*}
        g'(t)=f'(t)+\frac{z-w}{2}\left( -\frac{1}{z-w}f'(t) + \frac{z-t}{z-w}f''(t) \right) = \frac{1}{2}f'(t)+\frac{z-t}{2}f''(t)>0
    \end{align*}
   ($(\frac{1}{p-1})^\frac{1}{p-2}$ is the zero of $f''$), we conclude that $g$ is strictly positive on $(w,z]$ and, thus, $\ell'<0$ in $(0,z)$.
\end{proof}

\begin{proof}[Proof of Theorem \ref{thm:agsregular}]
First of all, we observe that the existence of bound-states is guaranteed by Theorem \ref{thm:amin}. Indeed, if $\mathcal{G}$ is a regular graph with $H$ half-lines, this theorem guarantees the existence of core-increasing solutions for any choice where $H^+\geq H^-$ (which corresponds to $\theta\geq 0$). 

Recalling once again  Lemma \ref{lemma:minimizing_sequences} and reasoning as in the beginning of the proof of Theorem \ref{thm:ags}, we want to show that, if $u$ is a weak $H^1$-limit of a minimizing sequence $u_n$ such that $u_n(\mathbf{0})>0$, then $u\not \equiv 0$.

By contradiction, suppose that $u=0$. Since $u_n\to 0$ in $H^1$ on each compact edge, a direct phase-plane analysis shows that the orbit of each $u_n$ needs to be in the interior of the homoclinic curve (that is, the energy level $C_e$ must be positive). Moreover, up to asubsequence,
\begin{itemize}
    \item on each loop, $u_n$ must be symmetric (otherwise, $\|(u_n)_e\|_{L^\infty}\geq 1$). 
    \item $u_n$ must be strictly decreasing in $(0,\ell)$ with $(u_n)_\kappa'(0)<0$ if $\kappa$ is a terminal edge (otherwise, $(u_n)_\kappa(\ell)\geq 1$).
\end{itemize}

Set $z_n=\varphi(y_n)=u_n(0)$ and fix a compact edge $\kappa$. Let $w_n=(u_n)_\kappa(\ell)$. By the Neumann-Kirchoff conditions, $(C_n)_\kappa=f(w_n)$ and then
$$
\ell = \frac{1}{\sqrt{2}}\int_{w_n}^{z_n}\frac{dt}{\sqrt{f(t)-f(w_n)}}.
$$
By Lemma \ref{lem:funcaol_zfixo}, $w_n$ is determined by the values of $\ell$ and $z_n$ and therefore independent of the edge $\kappa$. In particular, $u_n$ is symmetric. By Theorem \ref{thm:amin}, there exists at most one positive decreasing bound-state on $\mathcal{G}$. In particular, $u_n$ must be a constant sequence (equal to $u$), which is absurd since $u=0$.
\end{proof}

While Theorems \ref{thm:ags} and \ref{thm:agsregular} are a definitive answer to the existence of action ground-states for single-knot graphs, they do not give us any information on the qualitative properties of action ground-states. As such, we characterize $\mathcal{S}_\mathcal{G}$ with a variational approach, whenever available.
Consider the auxiliary minimization problem  
\begin{equation}\label{eq:agscharact}
	\mathcal{I}_{\G}(\mu):=\inf_{ u\in H^1(\G)}\left\{ I (u):=\frac{1}{2}\norm{u'}{2}{\G}^2+\frac{1}{2}\norm{u}{2}{\G}^2:\  \frac{1}{p}\norm{u}{p}{\G}^p=\mu\right\}.
\end{equation}
The relation between action ground-states and minimizers of \eqref{eq:agscharact} was shown in \cite[Section 3]{AgostinhoCorreiaTavares}, adapting in a simple way the ideias from \cite{adami2016ground, adami2015nls} (see also \cite{de2023notion}). Set
$$
\mu_0=\left(\frac{2}{p}\mathcal{I}_{\G}(1)\right)^{\frac{p}{p-2}}
$$
and observe that, by scaling, $\mathcal{I}_{\G}(\mu)=\mu^{2/p}\mathcal{I}_{\G}(1)$. The following result applies to general non-compact graphs.

\begin{lemma}\label{lem:minim} Let $\mathcal{G}$ be a non-compact metric graph.
\begin{enumerate}
    \item Suppose that $\mathcal{I}_\mathcal{G}(\mu_0)$ is attained. Then $\mathcal{S}_\mathcal{G}=\frac{p}{p-2}\mathcal{I}_\mathcal{G}(\mu_0)$ is achieved. Moreover, $u\in H^1(\mathcal{G})$ is an action ground-state if and only if $u$ is a minimizer of $\mathcal{I}_\mathcal{G}(\mu_0)$,  in which case $u$ is a signed function. 
    \item  One has $\mathcal{I}_{\G}(\mu)\leq \mathcal{I}_{\R}(\mu)$. If there exists $v\in H^1(\mathcal{G})$ with $\|v\|_{L^p(\mathcal{G})}^p=p\mu$ such that $\mathcal{I}(v,\mathcal{G})\le \mathcal{I}_\R(\mu)$  for some $\mu>0$,
then the minimization problem \eqref{eq:agscharact} has a solution for any $\mu>0$.
\end{enumerate}
\end{lemma}

It is standard to show that $I_\mathcal{\mathcal{G}}(\mu_0)$ can be characterized as a minimization problem on a Nehari manifold, and that $\mathcal{S}_\mathcal{G}\geq \frac{p}{p-2}I_\mathcal{\mathcal{G}}(\mu_0)$. However, it is possible that $I_\mathcal{\mathcal{G}}(\mu_0)$ is not achieved and $\mathcal{G}$ admits a ground-state, or also that both $I_\mathcal{\mathcal{G}}(\mu_0)$ and $\mathcal{S}_\mathcal{G}$ are not achieved. We refer to \cite{de2023notion} for a detailed discussion on this topic.

Combining Lemma \ref{lem:minim} with Theorem \ref{thm:amin}, we can prove the full characterization of action ground-states in the tadpole graphs:

\begin{proof}[Proof of Proposition \ref{cor:tadpole}]
First, we show that all bound-states are symmetric. Let $u_2:[0,2\ell]\to \R$ be the restriction of the bound-state to the loop.
	Evaluating \eqref{ODE_Chapter2} at $x=0$ and $x=2\ell$, and subtracting both equations, we obtain $u'_\kappa(0)^2=u'_2(2\ell)^2$.
	If $u_2'(0)=u_2'(2\ell)$, the Neumann-Kirchoff conditions yield $u_h'(0)=0$, which leads to a contradiction. Therefore, $u_2'(0)=-u_2'(2\ell)$ and $u_2'(0)=-\frac{1}{2}u'_h(0)=-\frac{1}{2}\varphi'(y)$. In particular, by reflection symmetry, $u_2$ is symmetric with respect to $x=\ell$.

Next, we show that  $\mathcal{I}_{\mathcal{G}}(\mu)$ is achieved for any $\mu>0$. Here, we closely follow the same argument in \cite{adami2016threshold} (which deals with \emph{energy} ground-states) and, for that reason, we just present a sketch. Take $v\in H^1(\R)$ to be a minimizer of $\mathcal{I}_{\R}(\mu)$, which satisfies, up to multiplication by a constant, $v=\varphi$. Recall that the loop has a total length of $2\ell$. Consider a graph $\widetilde G\simeq \R/\{-\ell,\ell\}$, obtained as the quotient space of $\R$ identifying the points $x=-\ell$ and $x=\ell$. By exploiting the symmetry of $v$ we can now glue the points $v(-\ell)=v(\ell)$, which gives rise to a function, denoted by $w$, such that $w\in H^1(\widetilde \G)$ with $\norm{w}{p}{\widetilde \G}^p=\norm{v}{p}{\R}^p=p\mu$ and $I(w)=\mathcal{I}_{\R}(\mu)$.
	
	Since $\widetilde \G$ contains two half-lines, say $e_1$ and $e_2$, where $w_{1}(x)=w_2(x)=\varphi(x+\ell)$ (due to the symmetry of the soliton $\varphi$), by doing a decreasing rearrangement of the function on both  half-lines into a single $H^1(\R^+)$ function, we obtain a function $u\in H^1(\mathcal{G})$. Moreover, $\norm{u}{p}{\G}^p=\norm{w}{p}{\G}^p=p\mu$ and, since $w$ is strictly decreasing on the half-lines and decays to zero at infinity, $n(t)=\#\{x\in\G:w(x)=t\}=2$, for all $t\in(0,w(\ell))$. By Pólya-Szeg\"o's Inequality, \cite[Proposition 3.1]{adami2015nls}, it follows that $I(\tilde{u},\G_\ell)<I(w)=\mathcal{I}_{\R}(\mu)$. By Lemma \ref{lem:minim}, that $\mathcal{I}_{\R}(\mu)$, and there exist action ground-states.

Proceeding as in \cite[Lemma 5.3]{dovetta2020uniqueness}, we conclude that any minimizer is, up to a change of sign, a positive symmetric increasing bound-state, decreasing on the half-line. Finally, by Theorem \ref{thm:amin}, the action ground-state is unique.
\end{proof}

Finally, using Lemma \ref{lem:minim} and reasoning similarly to \cite[Proposition 4.1]{adami2015nls} (which deals with energy ground-states), one shows the following general criterium, which we shall use in Section \ref{sec:large_ell}.
\begin{proposition}\label{ThresholdPhenomProp}
	Let $\G$ be a non-compact graph with \textit{at least} a terminal edge of length $\ell$. Then there exists $C_p^*>0$ such that, if $\ell\geq C_p^*$, the level $\mathcal{I}_{\G}(\mu)$ is attained for any $\mu>0$. In particular, there exist action ground-states.
\end{proposition}

\subsection{Characterization of bound-states for small $\ell$}\label{ExUniqofPosSolution}

In this section we prove Theorems \ref{thm:small_ell}. {In this section, we take $\mathcal{G}$ a regular single-knot graph of length $\ell$ with a \emph{fixed} number of half-lines ($H$),  loops ($L$) and  pendants ($P$).}

The first part of Theorems \ref{thm:small_ell}  deals with positive solutions, so we let  $u=u_\ell$ to be a bound-state such that $u(\mathbf{0})> 0$. As already mentioned, there exists $y=y(\ell)> 0$ such that
$$
u_h=\varphi(\cdot + y)\mbox{ or }u_h=\varphi(\cdot - y),\quad \text{ for every half-line $h$.}
$$
Let us now take $H^+$ (resp. $H^-=H^-$) to be the number of half-lines on which the first (resp. second) alternative holds. Our first goal  is to prove that, for $\ell$ small, such a bound-state is unique (up to permutations on the half-lines).

\begin{lemma}\label{lemma:loops_sym}
Let $\ell\leq \ell_1^*=\frac{\pi}{\sqrt{(p-2)}}$.  If $u$ is a bound-state in $\mathcal{G}$ and $\kappa$ is a loop of $\mathcal{G}$ where $u_\kappa>0$, then $u_\kappa$ is symmetric in  $\kappa$ with respect to the center of the loop.
\end{lemma}
\begin{proof}
Let $\kappa$ be a loop of length $2\ell$.  Assume $u_\kappa$ is not symmetric with respect to $x=\ell$. By a phase-plane analysis, then it should have at least one period. Since $u_\kappa$ is positive, then necessarily $F(u_\kappa,u_\kappa')=C_\kappa>0$, and the orbit is in the interior of the homoclinic. In particular, there exist two points $0< x_{1,\ell}<x_{2,\ell}\leq 2\ell$,  with $x_{2,\ell}-x_{1,\ell}\leq \ell$, where $u'$ vanishes, and $u$ is monotone decreasing in $(x_{1,\ell},x_{2,\ell})$. By Theorem \ref{th:SolIVPGeralTheta}-1(b), $\ell> \ell_1^*$.
\end{proof}

\begin{lemma}\label{lemma:uniformbounds}
There exists $M>0$ and $\bar \ell>0$ such that, if $\ell\in (0,\bar \ell)$ and $u$ is a positive bound-state in $\mathcal{G}$, then 
\[
\|u\|_{L^\infty(\Kcal)}\leq M.
\]
\end{lemma}
\begin{proof}

Let $\ell<\ell_1^*$, so that any bound-state is symmetric in each possible loop { with zero derivative at the middle point,} by Lemma \ref{lemma:loops_sym}.  {By cutting each loop at its middle point}{, we obtain two pendants. }Thus, we assume from now on that we are dealing with graphs whose compact core only has pendants, each one identified with $[0,\ell]$, and $u_\kappa'(\ell)=0$ for every $\kappa\in \Kcal$.

By contradiction, assume there exists $(u_\ell)_{\ell\geq 0}$ such that $\|u_\ell\|_{L^\infty(\Kcal)}\to \infty$ as $\ell\to 0^+$.

\smallbreak

\noindent {Step 1. } We claim that, as $\ell\to0^+$, $u_{\ell,\kappa}(\ell)$ is bounded, for every $\kappa\in \Kcal$.
	First, we check that there exists $\delta>0$, independent of $\ell$, such that $u'_{\ell,\kappa}(0)\geq-\delta$. If $u'_{\ell,\kappa}(0)\to-\infty$ as $\ell\to0^+$, by a phase-plane analysis, $u_{\ell,\kappa}$ would be sign-changing for small $\ell$, which is a contradiction. By the Neumann-Kirchoff conditions:
    $$0=\lim_{\ell\to 0^+}\left(\sum_{h\in \Hcal} u_{\ell,h}'(0)+ \sum_{\kappa\in \Kcal} u'_{\ell,\kappa}(0)\right)$$
	and since $u_{\ell,h}'(0)$ is bounded ($u_{\ell,h}$ is part of a soliton), then also $u_{\ell,\kappa}'(0)$ is bounded.
	
Since $u_{\ell,\kappa}(0)=\varphi(y(\ell))$ is bounded and 	
	\begin{equation*}
		-\frac{1}{2}{u'_{\ell,\kappa}}^2(0)+f\left(\varphi(y(\ell))\right)=f\left(u_{\ell,\kappa}(\ell)\right)
	\end{equation*}
    (because $u'_{\ell,\kappa}(\ell)=0$), then $u_{\ell,\kappa}(\ell)$ is also bounded for small $\ell$, as claimed.
	
\smallbreak

\noindent {Step 2.} There exists $M>0$, independent of $\ell$, such that $|u_{\ell,\kappa}|\leq M$ for every $\kappa\in \Kcal$.

	Let $x\in [0,\ell]$. Using again $u'_{\ell,\kappa}(\ell)=0$, 
	\begin{equation}\label{eqn.boundedbelow}
		f(u_{\ell,\kappa}(x))\geq-\frac{1}{2}{u'_{\ell,\kappa}}^2(x)+f\left(u_{\ell,\kappa}(x)\right)=f\left(u_{\ell,\kappa}(\ell)\right).
	\end{equation}
	Since $f$ is continuous, by the previous step we deduce that  the left hand side of \eqref{eqn.boundedbelow} is bounded below uniformly in $\ell$, hence the conclusion.
\end{proof}

\begin{lemma}\label{NeummanKirchofAux}
 Let $\epsilon>0$. Then, for any $\ell>0$ sufficiently small, if $u$ is a positive bound-state on $\mathcal{G}$, then $|u'_\kappa(0)|<\epsilon$ for all $\kappa\in \Kcal$.
\end{lemma}

\begin{proof}
As in the previous proof, taking $\ell<\ell_1^*$, we assume from now on that we are dealing with graphs whose compact core only has pendants, each one identified with $[0,\ell]$. 

Since $u_{\ell,\kappa}'(\ell)=0$, we have
	\begin{equation}\label{eqn.Lemma3.4}               |u_{\ell,
    \kappa}'(0)|=|u_{\ell,\kappa}'(0)-u_{\ell,\kappa}'(\ell)|=|u''_{\ell,\kappa}(\xi_\ell)|\ell=|f'(u_{\ell,\kappa}(\xi_\ell))|\ell,\quad \mbox{for some }\xi_\ell\in(0,\ell).
	\end{equation}
	By Lemma \ref{lemma:uniformbounds},  $|f'(u_{\ell,i}(\xi_\ell))|$ is uniformly bounded for $\ell>0$, and the result follows.

\end{proof}

The case where $H^+=H^-=\frac{H}{2}>0$ can be treated immediately.

\begin{proposition}\label{prop:CASEH^+=H^-}
Let $\mathcal{G}$ be a regular single-knot graph with length  $\ell$ sufficiently small, having an even number of half-lines $H$. For $H^+=H^-=H/2$  there exists a unique positive bound-state  with incidence index $\theta=(H^+-H^-)/(P+2L)=0$. This bound-state is  constant equal to 1 on the compact edges.
\end{proposition}

\begin{proof}
    By contradiction, suppose that $u$ is not constant on compact edges. Observe that, since $\theta=0$, the Neumann-Kirchoff condition at $\mathbf{0}$ reads
    \begin{equation}\label{eq:NL_auxH+=H-}
    \sum_{\kappa\in \Kcal} u_\kappa'(0)=0.
    \end{equation}
    Consider the case $0<u(\mathbf{0})\leq 1$ and take $\varepsilon>0$ small. From Lemma \ref{NeummanKirchofAux} and \eqref{eq:NL_auxH+=H-}, up to a subsequence there exists at least one edge $\kappa\in \Kcal$ on which $0\leq u_{\ell,\kappa}'(0)\leq \epsilon$. Since $u_\ell$ is positive, this implies that the orbit of $u_{\ell,\kappa}$ in the $(u,u')$-plane is enclosed in a rectangle $[0,A]\times [0,B]$, with $A>(p/2)^{1/(p-2)}$. Let $x_{0,\ell}\ge0$ be the first solution of $u_{\ell,\kappa}(x)=1$ and $x_{1,\ell}$ be the first positive zero of $u'_\kappa$. This implies that $u_{\ell,\kappa}$ is a strictly increasing positive solution of
    $$
    -u''+u=u^{p-1},\quad u(x_{0,\ell})=1,\quad u'(x_{1,\ell})=0.
    $$
    We have, using the change of variables $u(x)=t$,
    \begin{equation}\label{eq:newL}
    \ell\geq x_{1,\ell}-x_{\ell,0}=\frac{1}{\sqrt{2}}\int_1^
    {u(x_{1,\ell})}\frac{dt}{\sqrt{f(t)-f(u(x_{1,\ell}))}}.
    \end{equation}
The map 
\[
z\in (1,A]\mapsto \frac{1}{\sqrt{2}}\int_1^
    {z}\frac{dt}{\sqrt{f(t)-f(z)}}=\frac{1}{\sqrt{2}}\int_0^1 \frac{z-1}{\sqrt{f(z+(1-z)s)-f(z)}}\, ds
\]
is continuous and, since
 $$
    \lim_{z\to 1} \frac{(z-1)^2}{f(z(1-s)+z)-f(z)} = \lim_{z\to 1}\frac{2(z-1)}{f'(z(1-s)+z)(1-s)-f'(z)} = \frac{2}{f''(1)(1-s)^2-f''(1)},
    $$
we have
\begin{align*}
    \lim_{z\to 1^+} \int_0^1 \frac{z-1}{\sqrt{f(z+(1-z)s)-f(z)}}\, ds = \int_0^1 \frac{\sqrt{2}}{\sqrt{f''(1)((1-s)^2-1)}}ds >0.
\end{align*}
 This, combined with \eqref{eq:newL}, provides a contradiction for small $\ell$. The case $1<u(\mathbf{0})<(p/2)^{1/(p-2)}$ can be handled analogously, taking this time an edge $\kappa$ where $-\epsilon\leq u_{\ell,\kappa}(0)\leq 0$.

    It remains to exclude the case $u(\mathbf{0})=(p/2)^{1/(p-2)}$: in such case, there exists an edge on which $-\epsilon< u'_{\ell,\kappa}(0)  \leq 0$. If $u_{\ell,\kappa}'(0)< 0$ then, in the phase-plane portrait, the orbit of $u_{\ell,\kappa}$ must traverse the forth quadrant slightly below the homoclinic orbit, and it is sign changing solution. If $u_{\ell,\kappa}'(0)=0$, then $u_{\ell,\kappa}$ is a decreasing half-soliton, whose derivative never vanishes for $x>0$. In either case, we obtain a contradiction.
\end{proof}

\begin{lemma}\label{Solucoes com L pequeno}
Given $\epsilon>0$, if $\ell$ is sufficiently small and $u$ is a bound-state on  $\mathcal{G}$  such that $H^+\neq H^-$, $H=H^++H^-$, then
$$
|u(\mathbf{0})-z|<\epsilon, \quad \text{ for } z\in\{0,\varphi(0)\}.
$$
\end{lemma}

\begin{proof}
By contradiction, suppose that there exists a sequence of bound-states $u_\ell$, with $\ell\to 0$, such that
$$
|u_\ell(\mathbf{0})|>\epsilon \quad \mbox{and}\quad  |u_\ell(\mathbf{0})-\varphi(0)|>\epsilon.
$$
Passing if necessary to another subsequence, we can assume that  $H^+,H^-$ are constant.
Since 
   $$
    u_{\ell,h}=\varphi(\cdot\pm y(\ell))\qquad \forall h\in \Hcal
    $$
    for some $y(\ell)>0$, we conclude that $y(\ell)$ is bounded away from 0 and infinity. Hence, up to a subsequence, $y(\ell)\to y_0\neq 0$ as $\ell\to 0$, and
        $$
    u_{\ell,h}\to \varphi(\cdot\pm y_0) \quad \mbox{in}\quad  W^{1,\infty}_{loc}([0,\infty)).
    $$
    Using the Neumann-Kirchoff condition at $\mathbf{0}$ and Lemma \ref{NeummanKirchofAux},
    $$
    0=\lim_{\ell\to 0} \sum_{h\in \Hcal } u_{\ell,h}'(0) + \sum_{\kappa \in \Kcal} u'_{\ell,\kappa}(0) = (H^+-H^-)\varphi'(y_0)\neq 0
    $$
    a contradiction. 
\end{proof}

Consider the IVP
\begin{equation}\label{Theo.IVP.1}
	-u''+u= u^{p-1}\ \text{in}\ \R,\quad
	u(0)=a,\quad
	u'(0)=b.
\end{equation}
Given any solution $u$ (which is globally defined and depends in $C^1$ manner on the initial conditions $a,b$),  and define $F:\R^3\to\R$ by
\begin{equation}\label{FuncF}
	F(a,b,\ell)=u'(\ell).
\end{equation}

\begin{lemma}\label{Para l pequeno solucoes sao simetricas}
	For $\varepsilon>0$ sufficiently small, 
	consider $F$ as in \eqref{FuncF} and let $z\in\{0,\varphi(0)\}$ be fixed. 
	Then, there exist neighborhoods $V_z$ of $(z,0,0)$ and $U_z\subseteq \left(z-\varepsilon,z+\varepsilon\right)\times(-\varepsilon,\varepsilon)$ of $(z,0)$, and a $C^1$-function $g:U_z\to \R$ for which: $F(a,b,\ell)=0$ if and only if $b=g(a,\ell)$  for any $(a,b,\ell)\in V_z$. 
\end{lemma}
\begin{proof}
	This is a direct consequence of the Implicit Function Theorem applied to the $C^1$-function $F$ at the point $(z,0,0)$, which satisfies $F(z,0,0)=0$, and 
	\begin{equation*}
		\frac{\partial F}{\partial b}(z,0,0)=\frac{\partial u'}{\partial b}(0)=\left(\frac{\partial u}{\partial b}\right)'(0)=1\neq 0,
	\end{equation*}
	since  $v:=\left(\frac{\partial u}{\partial b}\right)'$ solves either
	\begin{equation*}
			-v''+v=(p-1)\varphi^{p-2}v,\quad  \text{ if $z=\varphi(0)$,}\qquad \text{ or }\qquad
			-v''+v=0, \quad \text{ if $z=0$,}
	\end{equation*}
    and $v(0)=0$, $v'(0)=1$.
\end{proof}

\begin{lemma}\label{lem_tfi2}
Let $\theta> 0$ and consider the map $G:(-\infty, \varphi(0)]\times \R\to \R$ defined as
\[
    G(a,\ell)=F(a,\theta\sqrt{2f(a)}\sign(a),\ell).
    \]
Then:
    \begin{itemize}
        \item there exist $\epsilon>0$, a neighborhood $\widetilde V_0$ of $(0,0)$ and a $C^1$-function $h:(-\epsilon,\epsilon)\to \R$ such that, over $\widetilde V_0$,
\[
    G(a,\ell)=0\quad \mbox{if and only if}\quad  a=h(\ell).
  \]
  \item there exist $\epsilon>0$, a neighborhood $\widetilde V_{\varphi(0)}$ of $(\varphi(0),0)$ and a $C^1$-function $h:[0,\epsilon)\to \R$ such that, over $\widetilde V_{\varphi(0)}\cap \{a\le \varphi(0)\}$,
  \[
    G(a,\ell)=0\quad \mbox{if and only if}\quad  a=h(\ell).
  \]
    \end{itemize}
\end{lemma}

\begin{proof}
We begin with the case $z=0$. Observe that $a\mapsto \theta\sqrt{2f(a)}\sign(a)=\theta a\sqrt{1-\frac{|a|^{p-2}}{p}}$ is of class $C^1$, and 
\[
\frac{\partial G}{\partial a}(0,0)=\theta\neq 0.
\] 
By the Implicit Function Theorem, near $(a,\ell)=(0,0)$, $a$ is locally determined as a function of $\ell$, $h(\ell)$, and the conclusion follows. 

Now we move to the case $z=\varphi(0)$. Since $z>1$, then for $a$ near $z$, the equation $b=\theta\sqrt{2f(a)}$ is equivalent to $a=f_2^{-1}(b^2/(2\theta^2))$. Instead of working directly with $G$, we deal with
$$
\widetilde{G}(b,\ell):=F(f_2^{-1}(b^2/(2\theta^2)), b, \ell).
$$
We apply the Implicit Function Theorem near $(b,\ell)=(0,0)$, where $\widetilde G$ is of class $C^1$: it suffices to compute
$$
\frac{\partial \widetilde{G}}{\partial b}(0,0) = \left(\frac{\partial u}{\partial b}\right)'(0)=1.
$$
As such, $\widetilde{G}(b,\ell)=0$ determines, for $\ell$ small and $b$ near $0$, $b$ as a function of $\ell$. Since $b\mapsto a=f_2^{-1}(2b^2/\theta^2)$ is a diffeomorphism from a right-neighborhood of $b=0$ to a left-neighborhood of $a=\varphi(0)$, the proof is finished.
\end{proof}

\begin{proof}[Proof of Theorem \ref{thm:small_ell} 
] 

\noindent \textit{Items 1.-3.}:\newline
    \noindent\textit{Step 1. Positive solutions are core symmetric for $\ell>0$ small.} For $\ell<\ell_1^*$, any positive bound-state $u$ is symmetric in each possible loop, by Lemma \ref{lemma:loops_sym}. Therefore, on each loop $\kappa$ of length $2\ell$, we have $u_\kappa'(\ell)=0$. Therefore we can cut the loop at  $x=\ell$, creating two pendants; reasoning in this way, we may assume without loss that $\mathcal{G}$ only pendants of length $\ell>0$. Let us check that $u_{\kappa_1}\equiv u_{\kappa_2}$ for every $\kappa_1,\kappa_2\in \Kcal$.  
    
    Take $\epsilon$ from Lemma \ref{Para l pequeno solucoes sao simetricas}. By Lemmas \ref{NeummanKirchofAux} and \ref{Solucoes com L pequeno}, for $\ell\leq \bar \ell$ sufficiently small (depending only on $\epsilon$), any bound-state $u$ on a regular single-knot graph of length $\ell$ satisfies
    $$
    |u_{\kappa}'(0)|<\epsilon  \text{ for } \kappa \in \Kcal ,\quad \text{ and } \quad  |u(\mathbf{0})-z|<\epsilon \text{ for } z\in\{0,\varphi(0)\}.
    $$
    By Lemma \ref{Para l pequeno solucoes sao simetricas}, $u_{\kappa}$ is uniquely determined by $\ell$ and $u(\mathbf{0})$. In particular, $u$ is core-symmetric and, therefore,
    $$
    u_\kappa'(0)=\frac{H^+-H^-}{P+2L}\sqrt{2f(u(\mathbf{0}))}=:\theta \sqrt{2f(u(\mathbf{0}))}\quad \text{ for every } \kappa\in \Kcal.
    $$ 
\smallbreak
    
    \textit{Step 2.  Nonexistence result of positive solutions for $H^->H^+$.} This is a direct consequence of Theorem \ref{thm:amin}-3.-5.. Indeed, for $H^->H^+$ we have $\theta<0$ and there are no decreasing positive solutions for $\ell<\ell_2^*(\theta)$.
\smallbreak

\textit{Step 4. Existence and uniqueness result for $H^+=H^-$}. This is a direct consequence of Proposition \ref{prop:CASEH^+=H^-}.

\smallbreak

    \textit{Step 5. Existence and uniqueness result for $H^+>H^-$.} 
    In this situation $\theta>0$, and the existence follows directly from Theorem \ref{thm:amin}. As for the uniqueness, by Lemma \ref{lem_tfi2}, near $z=0$ and $z=\varphi(0)$, the value of $u(\mathbf{0})$ is uniquely determined by $\ell$. If $z\sim 0$, then $u_\kappa$ must be constant equal to zero and then $u\equiv 0$, which contradicts the definition of bound-state. Then $z\sim\varphi(0)$, and $u_\kappa$ is uniquely determined by the relation $\varphi(y)=u(\mathbf{0})$.

\smallbreak 

\noindent \emph{Items (a)-(b):}\newline
We check that ground-states are necessarily signed for $\ell$ small, after which the remaining conclusions follow from Theorem \ref{thm:small_ell}-1.,2.,3.. Let $u_\ell$ be an action ground-state (recall Theorem \ref{thm:agsregular}) on  $\mathcal{G}$. On such graph, take also $v_\ell$, a positive core-symmetric and core-increasing solution with $H^+=H$ (see Theorem \ref{thm:small_ell}). Then
\[
S_{\mathcal{G}}=\left(\frac{1}{2}-\frac{1}{p}\right)\left(\sum_{h\in \Kcal} \|u_h\|_{H^1}^2+ \sum_{\kappa\in \Hcal} \|u_\kappa\|_{H^1}^2\right)\leq \left(\frac{1}{2}-\frac{1}{p}\right)\left(\sum_{h\in \Hcal} \|v_h\|_{H^1}^2+ \sum_{\kappa\in \Kcal} \|v_\kappa\|_{H^1}^2\right).
\]
Since $\|u_{\ell,h}\|_{H^1}^2, \|v_{\ell,h}\|_{H^1}^2\leq \|\varphi\|_{H^1(\R)}^2$ for every $h\in \Hcal$, and $\|v_{\ell,\kappa}\|_{H^1}$ is uniformly bounded in $\ell$ by Lemma \ref{lemma:uniformbounds}, then
\[
\sum_{\kappa\in \Kcal} \|u_{\ell,\kappa}\|_{H^1}^2 \quad \text{ is uniformly bounded in $\ell$}.
\]
In particular, given $\kappa\in \Kcal$, the orbits of $(u_{\ell,\kappa})_\ell$ are uniformly bounded, being contained in a rectangle $[-A,A]\times [-B,B]$ in the phase-plane.

Consider the case $u_{\ell}(\mathbf{0})>0$. If $u$ is sign-changing, then necessarily  $u_{\ell,\kappa}$ is sign-changing for some $k\in \Kcal$ and $\ell\sim 0$. Then there exists $x_{0,\ell}, x_{1,\ell}$ with  $0\leq x_{0,\ell}<x_{1,\ell}\leq 2\ell$ such that $u_\kappa$ is decreasing in $(x_{0,\ell},x_{1,\ell})$, and 
    $$
    -u''+u=u^{p-1},\quad u(x_{0,\ell})=0,\ u'(x_{0,\ell})<0, \quad \text{ and } \quad u'(x_{1,\ell})=0
    $$
    (i.e., its orbit traverses the whole fourth quadrant in the phase-plane).
Then, arguing as in the proof of Lemma \ref{prop:CASEH^+=H^-},
    \begin{equation}\label{eq:newL2}
2   \ell\geq x_{1,\ell}-x_{0,\ell}=\frac{1}{\sqrt{2}}\int_0^
    {u(x_{1,\ell})}\frac{dt}{\sqrt{f(t)-f(u(x_{1,\ell}))}}\geq \delta>0.
    \end{equation}
since 
\[
z\in \left[-A,\left(\frac{p}{2}\right)^{\frac{1}{p-2}}\right)\mapsto \widetilde{L}(z)=\frac{1}{\sqrt{2}}\int_0^
    {z}\frac{dt}{\sqrt{f(t)-f(z)}}=\frac{1}{\sqrt{2}}\int_0^1\frac{zds}{\sqrt{f(sz)-f(z)}}
\]
is continuous and $\lim_{z\to (p/2)^{1/(p-2)}} \widetilde L(z)=+\infty.$
 This is a contradiction for small $\ell$. The case $u_{\ell}(\mathbf{0})<0$ is analogous.

As for the case $u_{\ell}(\mathbf{0})=0$, it can not correspond to a ground-state since its orbit necessarily transverses one entire quadrant of the phase-plane, which (as in the previous paragraph) leads to a contradiction for $\ell>0$ small.
\end{proof}

\begin{remark}\label{rmk:falha_var}
From the previous arguments we see that, for $\ell$ small, the action of a bound-state is approximately $H$ times the action of a half-soliton. In particular, when $H\ge 3$, all bound-states have action strictly larger than that of the real-line soliton. By Lemma \ref{lem:minim}, the variational problem \eqref{eq:agscharact} \textit{does not} have a solution, since $\mathcal{I}_{\G}(\mu_0)<S(\varphi)<\mathcal{S}_{\mathcal{G}}$.
\end{remark}

\subsection{Symmetry of action ground-states for large $\ell$}\label{sec:large_ell}

In this section, we prove all statements mentioned in the introduction that are related with the study for $\ell>0$ large, namely Theorems \ref{prop:existsymbreak1},  \ref{cor:symmetry} and \ref{prop:symmetricPL1}.

\begin{proof}[Proof of Theorem \ref{prop:existsymbreak1}]

    \textit{Step 1. Setup.} For  $0<z<1$ and $0<b<\sqrt{2f(z)}$, take $w_1,w_2$ to be the solutions of
    \begin{equation}\label{eq:IVP_pos}
    -w''+w=|w|^{p-1}w \text{ in } \R^+,\quad w(0)=z
    \end{equation}
    with 
    \[
    w_1'(0)=-b<0,\quad \text{ and } \quad w_2'(0)=\frac{H\sqrt{2f(z)}+(P+2L-E_0)b}{E_0}> 0.
    \]
     Then, define $v\in H^1(\mathcal{G})$ by
    \[
    \begin{cases}
    v_\kappa=w_1 & \text{for every compact edge }\kappa\notin \mathcal{G}_0\\
    v_\kappa =w_2 & \text{for } \kappa \in \mathcal{G}_0\\
    v_h=\varphi(\cdot+y) & \text{for every half-line }h\in \Hcal,
    \end{cases}
    \]
where $y>0$ is such that $\varphi(y)=z$. Since $v'_h(0)=\varphi'(y)=-\sqrt{2f(z)}$,  then $v$ satisfies the Neumann Kirchoff conditions at the vertex $\mathbf{0}$. 

Reasoning exactly as in Section \ref{sec:PositiveSol}, the expressions of the first positive zero of $w_1'$ and $w_2'$ are given respectively by 
    \[
    g_{1}(z,b):=\frac{1}{\sqrt
    2}\int_{z_1}^z\frac{dt}{\sqrt{f(t)-f(z_1)}}\qquad \text{ and } \qquad  g_{2}(z,b):=\frac{1}{\sqrt{2}}\int_z^{z_2}\frac{dt}{\sqrt{f(t)-f(z_2)}},
    \]
    where $z_1,z_2$ are given by 
    \[
    -\frac{b^2}{2}+f(z)=f(z_1)\qquad \text{ and }-\frac{(w_2'(0))^2}{2}+f(z)=f(z_2).
    \]

For the rest of the proof, we fix $\epsilon>0$ small. We show, via the Poincar\'e-Miranda theorem, that for any $\ell>0$ large we can find such $(z,b)$ such that $(g_{1}(z,b),g_{2}(z,b))=(\ell,\ell)$ - which directly implies that $v$, as defined above, is a bound-state -  and show that $v$ satisfies all the remaining conditions of the statement.  

\smallbreak

\noindent \textit{Step 2. Study of $g_1$.} For any given $0<z<1$, we claim that
    \begin{equation}\label{PoincareMiranda}
    \lim_{b\to 0} g_1(z,b)=0,\qquad \lim_{b\to 2\sqrt{f(z)}} g_1(z,b) = +\infty.
    \end{equation}
    Indeed, when $b\to 0$, we have $z_1\to z$ and
    $$
    \lim_{b\to 0}g_1(z,b) = \lim_{z_1\to z}\frac{1}{\sqrt{2}}\int_0^1 \frac{z-z_1}{\sqrt{f(z_1+s(z-z_1))-f(z_1)}}ds.
    $$
    Since
    $$
    \lim_{z_1\to z} \frac{(z-z_1)^2}{f(z_1+s(z-z_1))-f(z_1)} = \lim_{z_1\to z}\frac{-2(z-z_1)}{f'(z_1+s(z-z_1))(1-s)-f'(z_1)} =0,
    $$
 the first limit in \eqref{PoincareMiranda} follows by the dominated convergence theorem. For the second limit, it suffices to observe that, as $b\to \sqrt{2f(z)}$, we have $z_1\to 0$ and therefore
 $$
    \lim_{b\to  \sqrt{2f(z)}}g_1(z,b) = \lim_{z_1\to 0}\frac{1}{\sqrt{2}}\int_{z_1}^z \frac{dt}{\sqrt{f(t)-f(z_1)}}=\frac{1}{\sqrt{2}}\int_{0}^z \frac{dt}{\sqrt{f(t)}}=\infty.
    $$

\smallbreak

\noindent \textit{Step 3. Study of $g_2$.}  Observe that $g_2(\epsilon, \cdot):[0,\sqrt{2f(\epsilon)}]\to \R$ is continuous and thus 
$$
|g_2(\epsilon, b)|\le M(\epsilon)\quad \text{ for every } 0\le b\le\sqrt{2f(\epsilon)}.
$$
Fix $\ell>2M(\epsilon)$. There exist $z_\ast\in (0,\epsilon)$ such that 
$$
g_{2}(z,b)\geq 2\ell \qquad \mbox{ for every }  0<z<z_\ast,\ 0<b<\sqrt{2f(z_*)}.
$$
Indeed, if, on the contrary, there exists a sequence $(z_n,b_n)\to (0,0)$  such that $g_2(z_n,b_n)<2\ell$, then
$$
2\ell\ge \lim_n g_2(z_n,b_n) = \lim \frac{1}{\sqrt{2}}\int_{z_n}^{(z_n)_2}\frac{dt}{\sqrt{f(t)-f((z_n)_1)}} = \frac{1}{\sqrt{2}}\int_0^{(p/2)^{1/(p-2)}}\frac{dt}{\sqrt{f(t)}} =+\infty,
$$
which is a contradiction. 

\smallbreak

\noindent \textit{Step 4. Construction of $v$.} Let
$$
D=\{ (z,b)\in \R^2: z_\ast\leq z\leq \epsilon,\quad 0<b<\sqrt{2f(z)}\}.
$$
Over $D$, the function $G=(g_1,g_2)$ is  continuous. By Steps 2 and 3, using the Poincar\`e-Miranda theorem, we deduce that, for $\ell>2M(\epsilon)$, there exists $(z,b)=(z_\ell,b_\ell)\in D$ with
$G(z_\ell,b_\ell)=(\ell,\ell)$. In particular, we have built a nontrivial bound-state $v_\ell$ which satisfies properties 1--5. Moreover, we point out that (recalling the notation from Step 1), $z_1=z_{1,\ell}=w_{1,\ell}(\ell)$ and $z_2=z_{2,\ell}=w_{2,\ell}(\ell)$.

\smallbreak

\noindent \textit{Step 5. Conclusion.} Finally, we prove the estimate for the action. First, observe that, on each half-line $h\in \Hcal$, since $v_{h,\ell}$ is a decreasing portion of the soliton with $0<v_{h,\ell}(0)<\epsilon$,
$$
\|v_{h,\ell}\|_{L^p(\R^+)}=O(\epsilon).
$$
On the compact edges $\kappa\notin \mathcal{G}_0$,
$$
\|v_{\kappa,\ell}\|_{L^p(0,\ell)}^p = \frac{1}{\sqrt{2}}\int_{z_{1,\ell}}^{z_{\ell}} \frac{t^pdt}{\sqrt{f(t)-f(z_{1,\ell})}}\to 0 \qquad \text{ as } \ell\to \infty,
$$
so we conclude that, for $\ell$ large,
$$
\|v_{\kappa,\ell}\|_{L^p(0,\ell)}^p < \epsilon,\qquad \text{ for every } \kappa \notin \mathcal{G}_0 \text{ compact edge}.
$$
Similarly, on an edge $\kappa\in\mathcal{G}_0$,
$$
\lim_{\ell\to \infty} \|v_{\kappa,\ell}\|_{L^p(0,\ell)}^p = \frac{1}{\sqrt{2}}\int_0^{(p/2)^{1/(p-2)}}\frac{t^p}{\sqrt{f(t)}} = \frac{1}{2}\|\varphi\|_{L^p(\R)}^p.
$$
Therefore, for $\ell$ large,
$$
S(v_\ell)=\left(\frac{1}{2}-\frac{1}{p}\right)\|v\|_{L^p(\mathcal{G})}^p = \left(\frac{1}{2}-\frac{1}{p}\right)\frac{E_0}{2}\|\varphi\|_{L^p(\mathcal{G})}^p + O(\epsilon) = \frac{E_0}{2}S(\varphi) + O(\epsilon).
$$
\end{proof}

In order to prove Theorem \ref{cor:symmetry}, we first need the following asymptotic estimate for the action of symmetric bound-states.

\begin{proposition}\label{prop:actionsym}
    Given $\epsilon>0$, if $\ell$ is sufficiently large and $u$ is a {core}-symmetric bound-state, then
    $$
    S(u)\ge \min\left\{H, \frac{P+2L}{2}\right\}S(\varphi) - \epsilon.
    $$
\end{proposition}

\begin{proof}
    By contradiction, suppose that $u_\ell$ is a sequence of {core}-symmetric bound-states for which
\begin{equation}\label{eq:accaomenorquephi}
        \lim_{\ell\to \infty} S(u_\ell) < \min\left\{H, \frac{P+2L}{2}\right\}S(\varphi).
\end{equation}
{Since $u_\ell$ is core symmetric, then $u_{\ell,\kappa}'(\ell)=0$ whenever $\kappa$ is a loop. Therefore, we can cut the loop at its middle point, and assume from the start that $\mathcal{G}$ has $P+2L$ pendants.}   

The sequence $u_\ell$ is uniformly bounded in $H^1$. Since $$
    \|u_\ell\|_{L^p}^p \le \|u\|_{L^\infty}^{p-2}\|u_\ell\|_{L^2}^2 \le \|u\|_{L^\infty}^{p-2}\|u_\ell\|_{H^1}^2 = \|u_\ell\|_{L^\infty}^{p-2}\|u_\ell\|_{L^p}^p\leq C\|u_\ell\|_{L^p}^p,
    $$
    we conclude that $\|u_\ell\|_{L^\infty} \ge {1/C}$ for all $\ell>0$. {We have that $u_\ell(\mathbf{0})\neq 0$, otherwise $u=0$ on all half-lines, and $u$ must be sign-changing on $\Kcal$, thus not core-symmetric. Without loss of generality, $u_\ell(\mathbf{0})> 0$. }
    
    Let $x_\ell\in \mathcal{G}$ to be a point of global maximum for $u_\ell$. We now split the analysis depending on the behavior of $x_\ell$ as $\ell\to\infty$.
\smallbreak

\noindent    \textit{Case 1.} There exists a compact set $\Lambda$ such that $x_\ell\in \Lambda$ for all $\ell$. For $\kappa\in \Kcal$, we extend $u_{\ell,\kappa}$, defined on $[0,\ell]$, to $v_{\ell,\kappa}$ defined on $\R^+$ as
    $$
    v_{\ell,\kappa}(x)=\begin{cases}
        u_{\ell,\kappa}(x), & x<\ell \\ u_{\ell,\kappa}(\ell)e^{-(x-\ell)}, & x>\ell.
    \end{cases}
    $$
    Letting $\mathcal{G}_\infty$ be the star graph with $H+P+2L$ half-lines, consider $v_\ell\in H^1(\mathcal{G}_\infty)$ which coincides with $u_{\ell, h}$ on $H$ half-lines and with $v_{\ell,\kappa}$ on the remaining $P+2L$ half-lines. One can easily check that $v_\ell$ is bounded in $H^1(\mathcal{G}_\infty)$, with maximum attained in $x_\ell\in K$ and $\|v_{\ell,\kappa}\|_{L^\infty}\ge 1$. Arguing as in Lemma \ref{lemma:minimizing_sequences}, up to a subsequence,  $v_\ell\rightharpoonup v$, and $v\in H^1(\mathcal{G}_\infty)\setminus\{0\}$ is a bound-state. In particular, if $H$ is odd, then $v_h=\varphi$, for all $h\in\mathcal{H}$; if $H$ is even, one may find a pairing of half-lines such that $v$ is equal to a soliton on the union of each pair of half-lines (see \cite[pg. 3]{Adami_star2012} for more details). This implies that
    $$
    \lim_{\ell\to \infty} S(u_\ell) = \lim_{\ell\to \infty} \frac{p-2}{2p}\|u_\ell\|_{L^p(\mathcal{G})}^p = \lim_{\ell\to \infty} \frac{p-2}{2p}\|v_\ell\|_{L^p(\mathcal{G}_\infty)}^p \ge  \frac{p-2}{2p}\|v\|_{L^p(\mathcal{G}_\infty)}^p  = \frac{H+P+2L}{2}S(\varphi),
    $$
which is a contradiction.

\smallbreak

\noindent \textit{Case 2.} $x_\ell\in \mathcal{H}$ and $x_\ell \to \infty$. In this case $u_{\ell,\kappa}'(0)\geq 0$ and  $u_{\ell,h} = \varphi(\cdot - y(\ell))$, so we must have $y(\ell)\to -\infty$ and
$$
\lim_{\ell\to \infty} S(u_\ell) \ge HS(\varphi)
$$
which is absurd.

\smallbreak

\noindent \textit{Case 3.} $x_\ell\in \mathcal{K}$ and there exists $C>0$ such that $|\ell-x_\ell| < C$ for every $\ell$. We define $w_{\ell,\kappa}$ as
$$
w_{\ell,\kappa}(x)=\begin{cases}
    u_{\ell,\kappa}(\ell-x),& 0<x<\ell, \\
    u_{\ell,\kappa}(0)e^{-(x-\ell)},& x>\ell.
\end{cases}
$$
In particular, $w_{\ell,\kappa}$ is bounded in $H^1(\R^+)$. Up to a subsequence, $w_{\ell,\kappa}\rightharpoonup w_\kappa\neq 0$, for some $w_\kappa$ bound-state in $\R^+$. By strong convergence on compact sets, $w_\kappa'(0)=0$, which implies that $w_\kappa$ is a half-soliton. In particular,
$$
\lim_{\ell \to \infty} S(u_\ell) \ge \lim_{\ell \to \infty} \sum_{\kappa\in\mathcal{K}}\frac{p-2}{2p}\|u_{\ell,\kappa}\|_{L^p(0,\ell)}^p  = \lim_{\ell \to \infty} \sum_{\kappa\in\mathcal{K}}\frac{p-2}{2p}\|w_{\ell,\kappa}\|_{L^p(R^+)}^p = \frac{P+2L}{2}S(\varphi),
$$
which contradicts \eqref{eq:accaomenorquephi}.

\smallbreak

\noindent \textit{Case 4.} $x_\ell\in \mathcal{K}$, $x_\ell\to \infty$ and $\ell-x_{\ell}\to \infty$. In this case, we consider $z_{\ell,\kappa}\in H^1(\R)$ defined as
$$
z_{\ell,\kappa}(x)=\begin{cases}
    u_{\ell,\kappa}(x+x_\ell),& -x_\ell<x<\ell-x_\ell, \\
    u_{\ell,\kappa}(0)e^{x+x_\ell},& x<-x_\ell.\\ u_{\ell,\kappa}(\ell)e^{-x+\ell-x_\ell},& x>\ell-x_\ell.
\end{cases}
$$
Similar to the previous case, up to a subsequence, $z_{\ell,\kappa}\rightharpoonup z_\kappa\neq 0$, where $z_\kappa$ is a bound-state on the real-line. In particular, $z_\kappa$ coincides with a soliton and thus
$$
\lim_{\ell \to \infty} S(u_\ell) \ge \lim_{\ell \to \infty} \sum_{\kappa\in\mathcal{K}}\frac{p-2}{2p}\|u_{\ell,\kappa}\|_{L^p(0,\ell)}^p  = \lim_{\ell \to \infty} \sum_{\kappa\in\mathcal{K}}\frac{p-2}{2p}\|z_{\ell,\kappa}\|_{L^p(R^+)}^p = (P+2L)S(\varphi),
$$
contradicting \eqref{eq:accaomenorquephi}.
 \end{proof}

\begin{proof}[Proof of Theorem \ref{cor:symmetry}]
    By contradiction, suppose there exists a {core}-symmetric action ground-state  $u_\ell$, for $\ell$ large. If $P\neq 0$, fix $P_0=1$, $L_0=0$ and $\epsilon>0$ small. For $\ell$ sufficiently large, take $v_\ell$ to be the bound-state given by Theorem \ref{prop:existsymbreak1}. Then, by Proposition \ref{prop:actionsym}, 
    $$
    S(u_\ell)\ge \min\left\{H,\frac{P+2L}{2}\right\}S(\varphi)-\epsilon \geq S(\varphi)-\epsilon = \frac{1}{2}S(\varphi)+ S(v_\ell)+O(\epsilon), $$
    contradicting the definition of action ground-state for sufficiently small $\epsilon$.

    If $P=0$, we choose instead $P_0=0$ and $L_0=1$. In this case, the bound-state $v_\ell$ given by Theorem \ref{prop:existsymbreak1} satisfies )
    \[
    S(u_\ell)\geq \min\{H, L\}S(\varphi) - \epsilon \geq  2S(\varphi)-\epsilon =S(\varphi)+S(v_\ell)+O(\epsilon),
    \]
    which is again a contradiction.
\end{proof}

\begin{remark}\label{rmk:P>0}
    In the case where $P\neq 0$, one may also prove Theorem \ref{cor:symmetry} using the variational approach \eqref{eq:agscharact}, which we sketch below. 
    
    On the one hand, Proposition \ref{ThresholdPhenomProp} implies that, for $\ell$ large, the action ground-state is a minimizer of $\mathcal{I}_\mathcal{G}(\mu_0)$. Arguing as in the proof of Proposition \ref{ThresholdPhenomProp}, one may construct a competitor $w_\ell$ with
    $$
    \mathcal{I}(w, \mathcal{G}) = \mathcal{I}_{\R^+}(\mu_0)+\epsilon,\quad \epsilon\mbox{ small},
    $$
    which implies, in particular, that $ \mathcal{I}(u_\ell, \mathcal{G}) \le \mathcal{I}_{\R^+}(\mu_0) + \epsilon$.

    On the other hand, if $u_\ell$ were core-symmetric, as $\ell\to \infty$, $u$ would have to converge to half-solitons on each compact edge and
    $$
     \mathcal{I}_{\R^+}(\mu_0)+\epsilon>\mathcal{I}(u,\mathcal{G})\sim (P+2L)\mathcal{I}_{\R^+}\left(\frac{\mu_0}{P+2L}\right) = (P+2L)^{1-2/p}\mathcal{I}_{\R^+}\left(\mu_0\right) 
    $$
    which is a contradiction for $P+L>1$.
\end{remark}

\begin{proof}[Proof of Theorem \ref{prop:symmetricPL1}]

\noindent \textit{Proof for $P=1$ and $L=0$.} Suppose that $u_\ell$ is an action ground-state which is not symmetric; in this case, it means it is not symmetric on the half-lines. If $u_\ell(\mathbf{0})=0$, then $u\equiv 0$ on each half-line, contradicting its lack of symmetry. If $u(\mathbf{0})>0$, then $H^+, H^->0$, which implies the existence of two half-lines $h_1, h_2$ on which union $u_\ell$ coincides with a soliton, hence $S(u_\ell)\geq S(\varphi)$.

    On the other hand, by Theorem \ref{th:SolIVPGeralTheta} and Proposition \ref{prop:2.61}, taking $\theta=H$, there exists a positive symmetric bound-state $v_\ell$ with $\lim_{\ell\to \infty} v_{\kappa,\ell}(0)=0$ and $\lim_{\ell\to \infty} v_{\kappa,\ell}(\ell)=\varphi(0)$ {(indeed, $L_1(\varphi(v_{\ell,h}(0)))=\ell\to \infty$)}. Then
$$
S(\varphi)\le \lim_{\ell\to \infty} S(u_\ell) \le  \lim_{\ell\to \infty} S(v_\ell) = \frac{1}{2}S(\varphi)
$$
which is a contradiction, since $S(\varphi)>0$.

\smallbreak

\noindent \textit{Proof for $P=0$ and $L=1$ (with $H\geq 2$).}  As before, Theorem \ref{th:SolIVPGeralTheta} and Proposition \ref{prop:2.61} imply the existence of a positive symmetric bound-state with $H^+=H$ satisfying $\lim_{\ell\to \infty} v_{\kappa,\ell}(0)=0$ and $\lim_{\ell\to \infty} v_\kappa(\ell)=\varphi(0)$. In particular,
\begin{equation}\label{eq:limSv2}
        \lim_{\ell\to\infty}S(v_\ell)= S(\varphi).
\end{equation} 
Let $u_\ell$ be a non-symmetric action ground-state. If $u(\mathbf{0})=0$, then $u$ is a ground-state on the loop. This implies that $u>0$, which is absurd. Therefore, we may assume that $u(\mathbf{0})>0$. 

If $u_\ell$ were not symmetric on the half-lines, then, as in the previous case $u_\ell$ coincides with a soliton on $h_1\cup h_2$.
Letting $\mathcal{G'}=\mathcal{G}\setminus (h_1\cup h_2)$, $u_\ell$ must be a nontrivial bound-state on $\mathcal{G'}$. Since
\begin{equation}\label{eq;gn}
    \|u_\ell\|_{H^1(\mathcal{G}')}^2=\|u_\ell\|_{L^p(\mathcal{G}')}^p\le C \|u_\ell\|_{H^1(\mathcal{G}')}^p, \quad C\mbox{ independent on }\mathcal{G}',
\end{equation}
this implies that $\|u_\ell\|_{L^p(\mathcal{G}')}^p>\delta $ for some universal $\delta>0$. As such,
$$
S(u_\ell) = \frac{p-2}{2p}\|u_\ell\|_{L^p(\mathcal{G})}^p = \frac{p-2}{2p}\left(\|u_\ell\|_{L^p(h_1\cup h_2)}^p + \|u_\ell\|_{L^p(\mathcal{G}')}^p\right) \ge S(\varphi) + \frac{p-2}{2p}\delta.
$$
Together with \eqref{eq:limSv2}, this contradicts the minimality of $u_\ell$. Therefore $u_\ell$ is symmetric on the half-lines.

Finally, suppose that $u_\ell$ is not symmetric on the loop. Then it must be periodic. In particular, $u$ must be a bound-state on $\mathcal{H}$, which implies that $u_\ell$ is equal to a soliton on pairs of half-lines. Hence, by \eqref{eq;gn},
$$
S(u_\ell)\ge \frac{H}{2}S(\varphi) + \|u_\ell\|_{L^p(\mathcal{K})}^p \ge \frac{H}{2}S(\varphi) + \frac{p-2}{2p}\delta.
$$
Recalling that $H\geq 2$, this contradicts \eqref{eq:limSv2} and the minimality of $u_\ell$.
\end{proof}

\paragraph{Acknowledgments.}
F. Agostinho, S. Correia and H. Tavares are partially supported by the Portuguese government through FCT - Funda\c c\~ao para a Ci\^encia e a Tecnologia, I.P., project UIDB/04459/2020 with DOI identifier 10-54499/UIDP/04459/2020 (CAMGSD), and project PTDC/MAT-PUR/1788/2020 with DOI identifier 10.54499/PTDC/MAT-PUR/1788/2020 (project NoDES). F. Agostinho was also partially supported by Funda\c c\~ao para a Ci\^encia e Tecnologia, I.P., through the PhD grant UI/BD/150776/2020. H. Tavares is also partially supported by FCT under the project  2023.13921.PEX, with DOI identifier  https://doi.org/10.54499/2023.13921.PEX (project SpectralOPs).

\bibliographystyle{plain}
\bibliography{AST}

	{\noindent Francisco Agostinho, Sim\~ao Correia and Hugo Tavares}\\
{\footnotesize
	Center for Mathematical Analysis, Geometry and Dynamical Systems,\\
	Instituto Superior T\'ecnico, Universidade de Lisboa\\
    Department of Mathematics,\\
	Av. Rovisco Pais, 1049-001 Lisboa, Portugal\\
	\texttt{francisco.c.agostinho@tecnico.ulisboa.pt}\\ \texttt{simao.f.correia@tecnico.ulisboa.pt}\\ \texttt{hugo.n.tavares@tecnico.ulisboa.pt}
}

\end{document}